\documentclass[11pt]{amsart}
\usepackage{amssymb, latexsym, graphicx, mathrsfs, enumerate, amsmath, amsthm}
\usepackage{color}
\input xy
\xyoption{all}

\newlength{\margins}
\usepackage[top=\margins,bottom=\margins,left=\margins,right=\margins]{geometry}

\numberwithin{equation}{section}
\newtheorem{Thm}{Theorem}[section]

\newtheorem{Lem}[Thm]{Lemma}

\theoremstyle{definition}
\newtheorem{Def}[Thm]{Definition}

\newtheorem{Rmk}[Thm]{Remark}

\begin{document}

\title[Group schemes and local densities of ramified hermitian lattices when $p=2$ Part I]
{Group schemes and local densities of ramified hermitian lattices in residue characteristic 2 Part I}

\author[Sungmun Cho]{Sungmun Cho}
\address{Department of
Mathematics, University of Toronto, CANADA}
\email{sungmuncho12@gmail.com}

\maketitle

\begin{abstract}
The obstruction to the local-global principle for a  hermitian lattice $(L, H)$ can be quantified by computing the mass of $(L, H)$.
The mass formula expresses the mass of $(L, H)$ as a product of local factors, called the local densities of $(L, H)$.
The local density formula  is known except in the case of a ramified hermitian lattice of residue characteristic 2.

Let $F$ be a finite unramified field extension of $\mathbb{Q}_2$.
Ramified quadratic extensions $E/F$ fall into two cases that we call \textit{Case 1} and \textit{Case 2}.
In this paper, we  obtain the local density formula for a ramified hermitian lattice in \textit{Case 1},
by constructing a smooth integral group scheme model for an appropriate unitary group.
Consequently, this paper, combined with the paper \cite{GY} of W. T. Gan and J.-K. Yu, allows the computation of the mass formula for a  hermitian lattice $(L, H)$  in \textit{Case 1}.

\end{abstract}
\let\thefootnote\relax\footnote{Primary MSC 11E41, MSC 11E95, MSC 14L15, MSC 20G25; Secondary MSC 11E39, MSC 11E57}

\tableofcontents

\section{Introduction}

\subsection{Introduction}
The subject of this paper is old and has intrigued many  mathematicians.
If $(V, H)$ and $(V^{\prime}, H^{\prime})$ are two hermitian $k^{\prime}$-spaces (or quadratic $k$-spaces),
where $k$ is a number field and $k^{\prime}$ is a quadratic field extension of $k$, then it is well known that
they are isometric if and only if for all places $v$, the localizations $(V_v, H_v)$ and $(V^{\prime}_v, H^{\prime}_v)$ are isometric.
That is, the local-global principle holds for hermitian spaces and quadratic spaces.
It is  natural to ask whether or not the local-global principle holds for a hermitian $R^{\prime}$-lattice (or a quadratic $R$-lattice) $(L,H)$,
where $R^{\prime}$ and $R$  are the rings of integers of $k^{\prime}$ and $k$ respectively.
In general, the answer to this question is no.
However, there is a way, namely, the mass of $(L,H)$,  to quantify the obstruction to the local-global principle.
 An essential tool for computing the mass of a quadratic or hermitian lattice is the mass formula.
The mass formula expresses the  mass of  $(L, H)$ as a product of local factors,
called the local densities of $(L, H)$.

Therefore, it suffices to find the explicit local density formula  in order to obtain the mass formula
and accordingly to quantify the obstruction to the local-global principle.


For a quadratic lattice, the local density formula was first computed by G. Pall \cite{P} (for $p\neq 2$) and G. L. Watson \cite{Wa} (for $p=2$).
For an expository sketch of their approach, see \cite{K}.
There is another proof of Y. Hironaka and F. Sato \cite{HS} computing the local density when $p\neq 2$.
They treat an arbitrary pair of lattices, not just a single lattice, over $\mathbb{Z}_p$ (for $p\neq 2$).
J. H. Conway  and J. A. Sloane \cite{CS} further developed the formula for any $p$ and gave a heuristic explanation for it.
Later, W. T. Gan and J.-K. Yu \cite{GY} (for $p\neq 2$) and S. Cho \cite{C1} (for $p=2$) provided a simple and conceptual proof of Conway-Sloane's formula
by explicitly constructing  a smooth affine group scheme $\underline{G}$ over $\mathbb{Z}_2$ with  generic fiber $\mathrm{Aut}_{\mathbb{Q}_2} (L, H)$,
 which satisfies $\underline{G}(\mathbb{Z}_2)=\mathrm{Aut}_{\mathbb{Z}_2} (L, H)$.

There has not been as much work done in computing local density formulas for hermitian
lattices as in the case of quadratic lattices.
Although the local density formula for a quadratic lattice with $p=2$ was first proved in the author's paper \cite{C1},
the formula was proposed in Conway-Sloane's paper \cite{CS}.
However, the local density formula for a ramified hermitian lattice with $p=2$ has not  been proposed yet
and therefore, the mass formula, when the ideal (2) is ramified in $k^{\prime}/k$, has not been known.

 Hironaka  obtained the local density formula for an unramified hermitian lattice in the papers  \cite{H1} and \cite{H2}.
 In addition, M. Mischler computed the formula for a ramified hermitian lattice ($p\neq 2$) under restricted conditions in \cite{Mis}.
Later, in the paper \cite{GY} mentioned above, Gan and Yu  found a conceptual and elegant proof of the local density formula for an unramified hermitian lattice without any restriction on $p$,
and for a ramified hermitian lattice with the restriction $p\neq 2$, by explicitly constructing  certain smooth affine group schemes
 (called smooth integral models) of a unitary group.
Ramified quadratic extensions $E/F$ fall into two cases that we call \textit{Case 1} and \textit{Case 2} (cf. Section 2.1),
 where $F$ is an unramified finite extension of $\mathbb{Q}_2$, depending on lower ramification groups $G_i$'s of the Galois group $\mathrm{Gal}(E/F)$  as follows:
\[
\left\{
  \begin{array}{l }
 \textit{Case 1}: G_{-1}=G_{0}=G_{1}, G_{2}=0;\\
 \textit{Case 2}: G_{-1}=G_{0}=G_{1}=G_{2}, G_{3}=0.
    \end{array} \right.
\]
 These two cases should be handled independently because of technical difficulty and complexity.
 The methodologies of the two cases are basically the same.
 But \textit{Case 2} is much more difficult than \textit{Case 1}.

The main contribution of this paper is to get an explicit formula for the local density of a hermitian $B$-lattice $(L, h)$ in \textit{Case 1},
by explicitly constructing a certain smooth group scheme associated to it that serves as an integral model
for the unitary group associated to $(L\otimes_AF, h\otimes_AF)$ and
  by investigating its special fiber,
where $B$ is a  ramified quadratic extension of $A$ and $A$ is an unramified finite extension of $\mathbb{Z}_2$ with $F$ as the quotient field of $A$.
The local density formula in \textit{Case 2} is handled in \cite{C2}.


In  conclusion, this paper, combined with \cite{GY} and \cite{C1},
allows the computation of the mass formula for a hermitian $R^{\prime}$-lattice $(L, H)$
 when $k_v/\mathbb{Q}_2$ is unramified, and  $k'_{v'}/k_v$  satisfies \textit{Case 1} or is unramified.
 Here, $k'_{v'}$ (resp. $k_v$) is the completion of $k'$ (resp. $k$) at the place $v'$ (resp. $v$), where
 $v'$ lies over $v$ and $v$ lies over the ideal $(2)$.
 As the simplest case, we can compute the mass formula for an arbitrary hermitian lattice explicitly
  when  $k$ is $\mathbb{Q}$ and $k^{\prime}$ is  any quadratic field extension of $\mathbb{Q}$
  such that the completion of $k'$ at any place lying over the ideal $(2)$ satisfies \textit{Case 1}
   or is unramified over  $\mathbb{Q}_2$.

Let us briefly comment on the proofs. A key input into the local density formula is
\begin{equation}
\lim_{N\rightarrow \infty} f^{-N~dim G}\#\underline{G}'(A/\pi^N A),
\end{equation}
where $f$ is the cardinality of the residue field of $A$, $\pi$ is a uniformizer in $A$, and
$\underline{G}'$ is the naive integral model for the unitary group $G$ associated to $(L\otimes_AF, h\otimes_AF)$,
which represents the functor $R\mapsto \mathrm{Aut}_{B\otimes_AR}(L\otimes_AR, h\otimes_AR)$.

Now if we are lucky enough that $\underline{G}'$ is smooth, then the limit in (1.1) would stabilize at $N=1$,
which would reduce us to simply finding $\underline{G}'(\kappa)$, $\kappa$ denoting the residue field of $A$.
A key observation of Gan and Yu is that, even when $\underline{G}'$ is not smooth, one can employ a certain smooth group scheme
$\underline{G}$ lurking in the background, which is a smooth integral model of $G$ satisfying that
$\underline{G}(R)=\underline{G}'(R)$ for every \'etale $A$-algebra $R$.
The existence and uniqueness of such a $\underline{G}$ is guaranteed by the general
theory of group smoothening. Then the problem essentially reduces to constructing $\underline{G}$ explicitly, so
that one can compute the cardinality of the group $\underline{G}(\kappa)$ of $\kappa$-points of its special fiber.
This tells us
what the analog of (1.1) for $\underline{G}$ is, and further, it so turns out that one can deduce the expression (1.1)
from its analog for $\underline{G}$.
For a detailed explanation about this, see Section 3 of \cite{GY}.

Let us now describe, therefore, how we construct $\underline{G}$ and study its special fiber. As $\underline{G}'$ fails to
be smooth one must impose more equations than merely the ones related to the preservation of $(L\otimes_AR, h\otimes_AR)$.
Towards this, note that there exist several sublattices $L'$ of $L$ such that any
element of $\mathrm{Aut}_B(L, h)$ automatically also preserves $L'$ (and such that, for any \'etale $A$-algebra $R$,
any element of $\mathrm{Aut}_{B\otimes_AR}(L\otimes_AR, h\otimes_AR)$  automatically also preserves $L'\otimes_AR$).
For instance, the sublattice $L'$ of elements $x\in L$ such that $h(x, L)$ belongs to a given ideal of $B$ necessarily satisfies
this property. This gives us additional equations to impose - these equations leave the group of $R$-
points for any \'etale $A$-algebra $R$ untouched, while taking us closer to smoothness. It so happens that
taking sufficiently many sublattices $L'$ into consideration, and imposing further restrictions arising
from the behavior of an element of $\mathrm{Aut}_{B\otimes_AR}(L\otimes_AR, h\otimes_AR)$ on some of their quotients, do leave
us with enough equations to ensure that the group scheme $\underline{G}$ defined by them is smooth. This step
already turns out to be much harder for $p = 2$ than for odd $p$, since in this case there are many
more isomorphism classes of hermitian lattices. Another source of complication is the fact that the
equations involve quadratic forms over the residue field $\kappa$ of $A$ that arise as quotients of some of the
lattices $L'$ mentioned above (the theory of quadratic forms over finite fields is more complicated in
characteristic 2 than in other characteristics).

Now let us describe some of the ideas involved in the computation of the special fiber $\tilde{G}$ of $\underline{G}$. Since the
quotients of some pairs of lattices of the form $L'$ alluded to in the above paragraph naturally support
symplectic or quadratic forms, it is not hard to construct a map $\varphi$ from $\tilde{G}$ to a suitable product
of symplectic and orthogonal groups. This step occurs in \cite{GY} too. However, $p$ being even for
us poses at least two new difficulties. Firstly, although this product of symplectic and orthogonal
groups contains the identity component of the maximal reductive quotient of $\tilde{G}$, this fact seems
to be difficult to prove directly. Rather, we prove this fact indirectly, by explicitly computing the
dimension of the kernel of $\varphi$. Secondly, $\varphi$ does not quite define the maximal reductive quotient of $\tilde{G}$
- this maximal reductive quotient is built up from $\varphi$ together with a few additional homomorphisms
$\tilde{G}\rightarrow \mathbb{Z}/2\mathbb{Z}$.

Our construction of these homomorphisms $\tilde{G}\rightarrow \mathbb{Z}/2\mathbb{Z}$ is quite indirect. A typical such homomorphism
is constructed in the following manner. We define a certain new hermitian lattice, say $(L'', h'')$,
starting from $(L, h)$. This lattice naturally gives us a homomorphism $\tilde{G}\rightarrow \tilde{G}''$,
where $\tilde{G}''$ is the special
fiber of the smooth integral model obtained by applying our construction to $(L'', h'')$ in place
of $(L, h)$. The analog $\varphi''$ of $\varphi$ defines a map from $\tilde{G}''$ to (a product of symplectic and orthogonal
groups, and in particular) an orthogonal group, and, composing with the Dickson invariant, one
gets a homomorphism $\tilde{G}''\rightarrow \mathbb{Z}/2\mathbb{Z}$. Precomposing this with the homomorphism $\tilde{G}\rightarrow \tilde{G}''$
 yields a homomorphism $\tilde{G}\rightarrow \mathbb{Z}/2\mathbb{Z}$.
  This is the manner in which all our homomorphisms $\tilde{G}\rightarrow \mathbb{Z}/2\mathbb{Z}$ are
constructed.

To show that the candidate for the maximal reductive quotient of $\tilde{G}$ obtained from $\varphi$ and the morphisms
$\tilde{G}\rightarrow \mathbb{Z}/2\mathbb{Z}$ is indeed the maximal reductive quotient, one shows that its kernel is isomorphic,
as an affine variety, to an affine space over $\kappa$ - this implies by a theorem of Lazard that the kernel
of our candidate for maximal reductive quotient is indeed a connected unipotent group scheme, as
desired.

Our main results are Theorem 3.8, Theorem 4.12 and Theorem 5.2. Theorem 3.8 shows that the
group scheme $\underline{G}$ we construct is indeed the sought after smooth group scheme over $A$, Theorem 4.12
gives the maximal reductive quotient of $\tilde{G}$, and Theorem 5.2 (supplemented by Remark 5.3) gives
us the final local density formulas as follows.
The local density of ($L,h$) is
$$\beta_L=f^N \cdot f^{-\mathrm{dim~} G} \#\tilde{G}(\kappa).$$
Here, $N$ is a certain integer which can be found in Theorem 5.2 and $\#\tilde{G}(\kappa)$ can be computed explicitly based on Remark 5.3.(1)  and Theorem 4.12.

Appendix B is devoted to illustrating our method with a simple example - that
of the case where $L = B\cdot e$ is of rank one and $h$ is defined by $h(le, l'e)=\sigma(l)l'$, $\sigma$ being the unique
nontrivial element of $\mathrm{Gal}(E/F)$. Appendix B.1 describes how the usual approach that works when
$p\neq 2$ (and yields the obvious integral model for the `norm one' torus associated to $B/A$) fails when
$p = 2$, and how one may fix this from `first principles', without using any of our techniques. We hope
this helps clarify some of the issues involved. Appendix B.2 illustrates how our construction specializes
to this case; we hope that the simplicity of this case may better motivate our general construction.
Some readers may therefore prefer to look at Appendix B before perusing the general constructions of
Sections 3 and 4 and Appendix A.

This paper is organized as follows.
 We first state a structure theorem for integral hermitian forms in Section 2.
  We then give an explicit construction of   $\underline{G}$
 (in Section 3) and study its special fiber (in Section 4) in \textit{Case 1}.
 Finally,
 we obtain an explicit formula for the local density in Section 5 in \textit{Case 1}.
In Appendix B, we provide an example to describe the smooth integral model and its special fiber
and to compute the local density for a unimodular lattice of rank 1.

The reader might want to skip to Appendix B and at least go to Appendix B.1 to get a first glimpse into why the case of $p=2$ is really different.
Some of the ideas behind our construction can be seen in the simple example illustrated in Appendix B.2.

The construction of smooth integral models and the investigation of their special fibers in this paper basically follow the arguments in \cite{GY} and \cite{C1}.

  As in \cite{GY}, the smooth group schemes constructed in this paper should be of independent interest.

\subsection{Acknowledgements}
The author greatly thanks the referee  for putting incredible time and effort into reading  this paper, and for providing a lot of valuable feedback.
The author owes a special debt to Professor Brian Conrad for his immense patience and copious suggestions, which helped make this paper substantially more pleasant to read.
This paper originated from the paper \cite{GY} and the author's Ph.D. dissertation, and the author would like to express
his deep appreciation to Professor Wee Teck Gan and Professor Jiu-Kang Yu.
The author would like to thank his Ph.D. thesis advisor Jiu-Kang Yu for many valuable comments.
In addition, the author would like to thank Professor Wai Kiu Chan, Professor Benedict H. Gross, and Professor Gopal Prasad for their interest in this project and their encouragement.
The author would like to thank Radhika Ganapathy,  Bogume Jang, Manish Mishra, Marco Rainho and Sandeep Varma
 for carefully reading a draft of this paper to help reduce the typographical errors and improve the presentation of this paper.

\section{Structure theorem for hermitian lattices and notations}

\subsection{Notations}
Notations and definitions in this section are taken from \cite{C1}, \cite{GY} and \cite{J}.
\begin{itemize}
\item Let $F$ be an unramified finite extension of $\mathbb{Q}_2$ with $A$  its ring of integers and $\kappa$  its residue field.
\item  Let $E$ be a ramified quadratic field extension of $F$ with $B$ its ring of integers. 
\item Let $\sigma$ be the non-trivial element of the Galois group $\mathrm{Gal}(E/F)$.
\item The lower ramification groups $G_i$'s of the Galois group $\mathrm{Gal}(E/F)$ satisfy one of the following:
\[
\left\{
  \begin{array}{l }
 \textit{Case 1}: G_{-1}=G_{0}=G_{1}, G_{2}=0;\\
 \textit{Case 2}: G_{-1}=G_{0}=G_{1}=G_{2}, G_{3}=0.
    \end{array} \right.
\]
We explain  the above briefly.
Based on Section 6 and Section 9 of \cite{J}, there is a suitable choice of a uniformizer $\pi$ of $B$ as follows.
In \textit{Case 1},  $E=F(\sqrt{1+2u})$ for some unit $u$ of $A$ and $\pi=1+\sqrt{1+2u}$.
Then $\sigma (\pi) =\epsilon \pi$, where $\epsilon\equiv 1$ mod $\pi$ and $\frac{\epsilon-1}{\pi}$ is a unit in $B$.
So we have that $\sigma(\pi) +\pi, \sigma(\pi)\cdot\pi \in (2)\backslash (4)$.\\
 In \textit{Case 2},  $E=F(\pi)$.
 Here, $\pi=\sqrt{2\delta}$, where $\delta\in A $ and $\delta\equiv 1 \mathrm{~mod~}2$.
 Then  $\sigma(\pi)=-\pi$.\\
 From now on, a uniformizing element $\pi$ of $B$, $u$, and $\delta$ are fixed as explained above throughout this paper.

The constructions of smooth integral models associated to these two cases are different  and  we will treat them independently.
\item We consider a $B$-lattice $L$ with a hermitian form $$h : L \times L \rightarrow B,$$
 where $h(a\cdot v, b \cdot w)=\sigma(a)b\cdot h(v,w)$ and $h(w,v)=\sigma(h(v,w))$. Here, $a, b \in B$ and $v, w \in L$.
 We denote by a pair $(L, h)$ a hermitian lattice.
We assume that $V=L\otimes_AF$ is nondegenerate with respect to $h$.
\item We denote by $(\epsilon)$ the $B$-lattice of rank 1 equipped with the hermitian form having Gram matrix $(\epsilon)$.
We use the symbol $A(a, b, c)$ to denote the $B$-lattice $B\cdot e_1+B\cdot e_2$
with the hermitian form having Gram matrix $\begin{pmatrix} a & c \\ \sigma (c) & b \end{pmatrix}$.
For each integer $i$,
the lattice of rank 2 having  Gram matrix $\begin{pmatrix} 0 & \pi^i \\ \sigma(\pi^i)  & 0 \end{pmatrix}$ is
called the $\pi^i$-modular hyperbolic plane and denoted by $H(i)$.
\item A  hermitian lattice $L$ is the orthogonal sum of sublattices $L_1$ and $L_2$, written $L=L_1\oplus L_2$, if $L_1\cap L_2=0$, $L_1$ is orthogonal to $L_2$ with respect to the hermitian form $h$, and $L_1$ and $L_2$ together span $L$.
\item The ideal in $B$ generated by $h(x,x)$ as $x$ runs through $L$ will be called the norm of $L$ and written $n(L)$.
\item By the scale $s(L)$ of $L$, we mean the ideal generated by the subset $h(L,L)$ of $B$.
\item We define the dual lattice of $L$, denoted by $L^{\perp}$, as
 $$L^{\perp}=\{x \in L\otimes_A F : h(x, L) \subset B \}.$$

\end{itemize}

\begin{Def}
For a given hermitian lattice $L$,
\begin{enumerate}
\item[a)]
For any non-zero scalar $a$, define $aL=\{ ax|x\in L \}$.
It is also a lattice in the space $L\otimes_AF$. Call a vector $x$ of $L$ maximal in $L$
if $x$ does not lie in $\pi L$.
\item[b)]
$L$ will be called $\pi^i$-modular if the ideal generated by the subset $h(x, L)$ of $E$ is $\pi^iB$ for every maximal vector $x$ in $L$.
Note that $L$ is $\pi^i$-modular if and only if $L^{\perp}=\pi^{-i}L$.
We can also see that $H(i)$ is $\pi^i$-modular.
\item[c)] Assume that $i$ is even. A $\pi^i$-modular lattice $L$ is \textit{of parity type I} if $n(L)=s(L)$; otherwise \textit{of parity type II}.
The zero lattice is considered to be  \textit{of parity type II}.
We caution that we do not assign \textit{parity type} to a $\pi^i$-modular lattice $L$ with $i$ odd.
\end{enumerate}
\end{Def}

\subsection{A Structure Theorem for Integral hermitian Forms}

We  state a structure theorem for $\pi^i$-modular lattices in this subsection.
Note that if $L$ is $\pi^{2i}$-modular (resp. $\pi^{2i+1}$-modular), then $\pi^{-i}L\subset L\otimes_AF$ is $\pi^{0}$-modular (resp. $\pi^{1}$-modular).
We will emphasize this in Remark 2.3.(a) again.
Thus it is enough to provide a structure theorem for $\pi^0$-modular or $\pi^1$-modular  lattices.

 \begin{Thm}
 Let $i=0$ or $1$.
\begin{enumerate}
\item[a)] Let $L$ be a $\pi^i$-modular lattice of rank at least $3$.
 Then $L=\bigoplus_{\lambda}H_{\lambda}\oplus K$, where $K$ is  $\pi^i$-modular of rank 1 or 2, and each $H_{\lambda}= H(i)$.
\item[b)]  
We denote by $(1)$ or $(2)$  the ideal of $B$ generated by the element $1$ or $2$, respectively.
Assume that $K$ is $\pi^i$-modular of rank 1 or 2. Then, depending on $i$, the rank of $K$, the case that $E/F$ falls into,
the parity type of $L$ (when applicable), and $n(L)$ which is the norm of $L$, we may take $K$ to be of the following form :


\begin{center}
    \begin{tabular}{| l | l | l | l | l |l |}
    \hline
    Rank of $K$ & $i$ & $E/F$ & Parity type of $L$ & n(L)  &   Form for $K$  \\ \hline
    1 & 0 & Case 1 & (necessarily) $I$ & (necessarily) (1)  &  $(a)$, $a\in A$, $a \equiv 1$ mod 2   \\ \hline
 1 & 0 & Case 2 & (necessarily) $I$ & (necessarily) (1) &  $(a)$, $a\in A$, $a \equiv 1$ mod 2   \\ \hline
   2 & 0 & Case 1 &  $I$ & (necessarily) (1) &  $A(1, 2b, 1)$, $b\in A$   \\ \hline
    2 & 0 & Case 2 &  $I$ & (necessarily) (1) &  $A(1, 2b, 1)$, $b\in A$   \\ \hline
    2 & 0 & Case 1 &  $II$ & (necessarily) (2) &  $H(0)$   \\ \hline
    2 & 0 & Case 2 &  $II$ & (necessarily) (2) &  $A(2\delta, 2b', 1)$, $b'\in A$   \\ \hline
    2 & 1 & Case 1 &   & (necessarily) (2) &  $A(2, 2a, \pi)$, $a\in A$   \\ \hline
    2 & 1 & Case 2 &   & (2) &  $A(4a, 2\delta, \pi)$,  $a\in A$  \\ \hline
    2 & 1 & Case 2 &   & (4) &  $H(1)$ \\ \hline
    \end{tabular}
\end{center}
\textit{ }


 \end{enumerate}
 \end{Thm}

\begin{proof}
Part $a)$ is proved in Proposition 10.3 of \cite{J}.

For Part $b)$, when the rank of $K$ is $1$, it is clear that $K\cong (a')$ for a certain unit $a'\in A$ with a basis $e$.
Since the residue field $\kappa$ is perfect, there is a unit element $a''$ in $A$ such that $a'\equiv a''^2$ mod $2$.
The reader can check that replacing $e$ by $(1/a^{''})e$ realizes $K$ in the manner dictated by the theorem.

From now on, we assume that the rank of $K$ is $2$.
Suppose that $i=0$. Then $n(K)=(1)$ or $n(K)=(2)$ since $n(K)\supseteq n(H(0))=(2)$
(Proposition 9.1.a) and (9.1) with $k=0$ in \cite{J}).
If $n(K)=(1)$, then we can use Proposition 10.2 of \cite{J} so that $K\cong A(1, a, 1)$ with respect to a basis $(e_1, e_2)$.
Furthermore, the determinant $a-1$ is a unit in $A$.
To show this,  we observe that  $K$ has an orthogonal basis, since $n(K)=s(K)=(1)$ (Proposition 4.4 in \cite{J}), and so the determinant should be a unit in order for $K$ to be $\pi^0$-modular.
 Since the residue field $\kappa$ is perfect, there is a unit element $\beta$ in $A$ such that $a-1\equiv \frac{1}{\beta^2}$ mod $2$.
 We now choose an another basis $(e_1, (1-\beta)e_1+\beta e_2)$. With this basis, it is easy to see that $K\cong  A(1, 2b, 1)$ for a certain $b\in A$.

 Now assume that $n(K)=(2)$ so that we cannot use Proposition 10.2 of \cite{J}.
 We choose a basis of K so that $K \cong A(x, y, 1)$ for some $x, y \in A$. Since $n(K)=(2)$, $x, y$ should be contained in the ideal $(2)$.
 Thus $K \cong A(2a, 2b, 1)$ for some $a, b \in A$.
 Furthermore, in \textit{Case 1}, if $n(K)=(2)$ then $K \cong H(0)$ by Proposition 9.2 a), b) of \cite{J}.

 The remaining which we need to prove when $i=0$ is then  that $K=A(2\delta, 2b', 1)$,  for certain $b'\in A$, in \textit{Case 2} if $n(K)=(2)$.
 By Proposition 9.2 a) of \cite{J}, if $K$ is isotropic then $K\cong H(0)$ so that we can choose $b'=0$.
 Furthermore, the lattice $K$ with $n(K)=(2)$ is determined by its determinant up to isomorphism (Proposition 10.4 in \cite{J}).
 Since the determinant of $K$, denoted by $d(K)$, is a unit and is well-defined modular $N_{B/A}B^{\times}$,
 there are  at most two cases of $d(K)$ from the fact that $|A^{\times}/N_{B/A}B^{\times}|=2$.
 Here,  $B^{\times}$ or $A^{\times}$ is the unit group of $B$ or $A$ and $N_{B/A}B^{\times}$ is the norm of $B^{\times}$.
 We observe that $d(A(2\delta, 0, 1))$ and $d(A(2\delta, 2d/\delta, 1))$, which are clearly $\pi^0$-modular,  give different classes in $A^{\times}/N_{B/A}B^{\times}$,
 where $d$ is as defined in Lemma 2.4.
 Thus, a lattice $K$ with $n(K)=(2)$ in \textit{Case 2} should be isomorphic to one of these two,
 in other words, such $K$ is isomorphic to $K=A(2\delta, 2b', 1)$ with $b'=0$ or $b'=d/\delta$.

We next suppose that $i=1$.
In \textit{Case 1}, $n(H(1))=(2)$ and so $n(K)=(2)$ since $s(K)\supseteq n(K) \supseteq n(H(1))$.
Thus $K$ is also determined by  its determinant up to isomorphism (Proposition 10.4 in \cite{J}).
This fact implies that there are at most two cases for $K$
since the determinant of $K$ divided by $2$ is  a unit in $A$ and the cardinality of $A^{\times}/N_{B/A}B^{\times}$ is $2$.
By Lemma 2.5, $d(A(2, 0, \pi))$ and $d(A(2, 2ud, \pi))$, which are clearly $\pi^1$-modular, give different classes in $A^{\times}/N_{B/A}B^{\times}$,
 where $u$ and $d$ are as defined in Lemma 2.5.
 Thus, a lattice $K$ with $n(K)=(2)$  in \textit{Case 1} should be isomorphic to one of these two,
in other words, such $K$ is isomorphic to $d(A(2, 0, \pi))$ or $d(A(2, 2ud, \pi))$.

In \textit{Case 2}, $n(H(1))=(4)$ and so $n(K)=(2)$ or $n(K)=(4)$.
If $n(K)=(2)$, then we can use Proposition 10.2. b) of \cite{J} (take $m=1$) so that $K\cong A(2\delta, 4a, \pi)$ with basis $(e_1, e_2)$.
If we use a basis $(e_2, -e_1)$, then  $K\cong A(4a, 2\delta, \pi)$.
If $n(K)=(4)$, then by Proposition 9.2 b), a), $K\cong H(1) \cong A(0, 4\delta, \pi)$.

These complete the proof.
\end{proof}

 \begin{Rmk}
 \begin{enumerate}
 \item[a)]
  If $L$ is $\pi^i$-modular, then $\pi^j L$ is $\pi^{i+2j}$-modular for any integer $j$.
   Thus, the above theorem implies its obvious generalization to the case where $i$ is allowed to be any element of $\mathbb{Z}$.
\item[b)] (Section 4 in \cite{J}) For a general lattice $L$, we have a Jordan splitting, namely $L=\bigoplus_i L_i$ such that $L_i$ is $\pi^{n(i)}$-modular  and the sequence $\{n(i)\}_i$ increases.
Jordan splittings $L=\bigoplus_{1\leqq i \leqq t} L_i$ and $K=\bigoplus_{1\leqq i \leqq T} K_i$ will be said to be of the same type
if $t=T$, and  $s(L_i)=s(K_i)$, rank $L_i$ = rank $K_i$, and $n(L_i)=s(L_i)$ if and only if $n(K_i)=s(K_i)$ for each $i$.
Jordan splitting is not unique but partially canonical in the sense that two Jordan splittings of isometric lattices are always of the same type.
\item[c)]
If we allow some of the $L_i$'s to be zero, then we may assume that $n(i) = i$ for all $i$.
In other words, for all $i\in \mathbb{N}\cup \{0\}$  we have $s(L_i)=(\pi^i)$, and, more precisely, $L_i$ is $\pi^i$-modular.
Then we can rephrase the above remark b) as follows: Let $L=\bigoplus_i L_i$ be a Jordan splitting with $s(L_i)=(\pi^i)$ for all $i\geq 0$.
Then the scale, rank and parity type of $L_i$ depend only on $L$.
We will deal exclusively with a Jordan splitting satisfying $s(L_i)=(\pi^i)$ from now on.
\end{enumerate}
 \end{Rmk}

\begin{Lem}Assume that $B/A$ satisfies \textit{Case 2}.
There is an element $d\in A^{\times}$ such that $1-4d$ and $1$ give different classes in $A^{\times}/N_{B/A}B^{\times}$.
\end{Lem}

\begin{proof}
Using  our knowledge of the lower ramification groups $G_i$ for $\mathrm{Gal}(E/F)$,
we can compute the higher ramification groups $G^i$ for the same extension:
\[G^{-1}=G^{0}=G^{1}=G^{2}, G^{3}=0.\]
Let $U^i=1+(2)^i$ be the  $i$-th higher unit group in $F$ with $i\geq 1$.
Then by \textit{local class field theory},
the image of $G^i$ under the isomorphism
$\mathrm{Gal}(E/F) \cong F^{\ast}/N_{E/F}E^{\ast}$
is $U^i/(U^i\cap N_{E/F}E^{\ast})$.
 We apply this when $i$ is $2$.
 Then we can easily verify the existence of a $d$  as  stated in the lemma.
\end{proof}

\begin{Lem} Assume that $B/A$ satisfies \textit{Case 1}.
There is an element $d\in A^{\times}$ such that $1+2d$ and $1$ give different classes in $A^{\times}/N_{B/A}B^{\times}$.
\end{Lem}

\begin{proof}
The proof of this lemma is similar to that of the above lemma.
In this case the higher ramification groups are as follows:
\[G^{-1}=G^{0}=G^{1}, G^{2}=0.\]
Again we use \textit{local class field theory} as explained in the proof of the above lemma but with   $i=1$.
 Then we can easily verify the existence of a $d$ as stated in the lemma.
\end{proof}

 \subsection{Lattices}

In this subsection, we will define several lattices and associated notation. A hermitian lattice $(L, h)$ is given.
We denote by $(\pi^l)$ the scale $s(L)$ of $L$.

\begin{itemize}
\item[(1)]  $A_i$, $\{x\in L \mid h(x,L) \in \pi^iB\}.$
\item[(2)] $X(L)$, the sublattice of $L$ such that
 $X(L)/\pi L$ is the radical of the symmetric bilinear form $\frac{1}{\pi^l}h$ mod $\pi$ on $L/\pi L$.
\end{itemize}

Let $l=2m$ or $l=2m-1$.
We consider the function defined over $L$
$$\frac{1}{2^m}q : L\longrightarrow A, x\mapsto \frac{1}{2^m}h(x,x).$$
Then $\frac{1}{2^m}q$  mod 2 defines a quadratic form $L/\pi L \longrightarrow \kappa$.
It can be easily checked that $\frac{1}{2^m}q$ mod 2 on $L/\pi L$ is an additive polynomial if $l=2m$, or if $l=2m-1$ and $E/F$ satisfies
\textit{Case 2}.
Otherwise, that is, if $l=2m-1$ and $E/F$ satisfies \textit{Case 1}, it is not  additive.
We define a lattice $B(L)$ as follows.
\begin{itemize}
\item[(3)] If $\frac{1}{2^m}q$ mod 2 on $L/\pi L$ is an additive polynomial, then
$B(L)$ is defined to be the sublattice of $L$ such that
 $B(L)/\pi L$ is the kernel of the additive polynomial  $\frac{1}{2^m}q$ mod 2 on $L/\pi L$.
 If $\frac{1}{2^m}q$ mod 2 on $L/\pi L$ is not an additive polynomial, then
 $B(L)=L$.
\end{itemize}

To define a few more lattices,  we need some preparation as follows.

Assume $B(L)\varsubsetneq L$ and $l$ is even. Then the bilinear form $\frac{1}{(\pi\cdot\sigma(\pi))^{\frac{l}{2}}}h$ mod $\pi$
on the $\kappa$-vector space $L/X(L)$ is nonsingular symmetric and non-alternating.
It is well known that there is a unique vector $e \in L/X(L)$ such that $(\frac{1}{(\pi\cdot\sigma(\pi))^{\frac{l}{2}}}h(v,e))^2=\frac{1}{(\pi\cdot\sigma(\pi))^{\frac{l}{2}}}h(v,v)$ mod $\pi$ for every vector $v \in L/X(L)$.
Let $\langle e\rangle$ denote  the 1-dimensional vector space spanned by the vector $e$ and denote by $e^{\perp}$ the 1-codimensional subspace of $L/X(L)$ which is orthogonal to the vector $e$ with respect to  $\frac{1}{(\pi\cdot\sigma(\pi))^{\frac{l}{2}}}h$ mod $\pi$.
Then $$B(L)/X(L)=e^{\perp}.$$
If $B(L)= L$, then the bilinear form $\frac{1}{(\pi\cdot\sigma(\pi))^{\frac{l}{2}}}h$ mod $\pi$ on the $\kappa$-vector space $L/X(L)$
is nonsingular symmetric and alternating. In this case, we put $e=0\in L/X(L)$
and note that it is characterized by the same identity.\\
The remaining lattices we need for our definition are:
\begin{itemize}
\item[(4)] $W(L)$,
the sublattice of $L$ such that \[
\left\{
  \begin{array}{l l}
 \textit{$W(L)/X(L)=\langle e\rangle$} & \quad \textit{if $l$ is even};\\
 \textit{$W(L)=X(L)$} & \quad \textit{if $l$ is odd}.
    \end{array} \right.
\]
\item[(5)] $Y(L)$, the sublattice of $L$ such that $Y(L)/\pi L$ is the radical of
\[
\left\{
  \begin{array}{l l}
 \textit{the alternating bilinear form $\frac{1}{2^{m}}h$
  mod $\pi$ on $B(L)/\pi L$} & \quad \textit{if $l=2m$};\\ 
 \textit{the alternating bilinear form $\frac{1}{\pi}\cdot\frac{1}{2^{m-1}}h$ mod $\pi$ on $B(L)/\pi L$} & \quad \textit{if $l=2m-1$ in \textit{Case 2}}.
    \end{array} \right.
\]
\end{itemize}

The last lattice we need to define in this subsection is
\begin{itemize}

 \item[(6)] $Z(L)$, the sublattice of $L$ such that $Z(L)/\pi L$ in \textit{Case 1} or $Z(L)/\pi B(L)$ in  \textit{Case 2} is the radical of
 \[
\left\{
  \begin{array}{l l}
 \textit{the quadratic form $\frac{1}{2^{m}}q$ mod $2$ on $L/\pi L$}  & \quad \textit{if $l=2m-1$ in \textit{Case 1}};\\
 \textit{the quadratic form $\frac{1}{2^{m+1}}q$ mod $2$ on $B(L)/\pi B(L)$} & \quad  \textit{if $l=2m$ in \textit{Case 2}}.
    \end{array} \right.
\]
\end{itemize}
(see, e.g.,  page 813 of \cite{Sa} for the notion of the radical of a quadratic form on a vector space over a field of characteristic 2.)

\begin{Rmk}
 \begin{enumerate}
 \item[a)]  We can associate
the 5 lattices above $(B(L), W(L), X(L), Y(L), Z(L))$
with $(A_i, h)$ in place of $L$.
 Denote the resulting lattices by $B_i,W_i,X_i,Y_i,Z_i.$
 \item[b)] As $\kappa$-vector spaces, the dimensions of $A_i/B_i$ and  $W_i/X_i$ are at most 1.
 \end{enumerate}
\end{Rmk}

Let $L=\bigoplus_i L_i$ be a Jordan splitting. When $i$ is even, we assign a type to each $L_i$ as follows:
\[\left \{
  \begin{array}{l l}
   I & \quad  \textit{if $L_i$ is of parity type I};\\
   I^o & \quad \textit{if $L_i$ is of parity type I and the rank of $L_i$ is odd};\\
  I^e & \quad \textit{if $L_i$ is of parity type I and the rank of $L_i$ is even};\\
\textit{II} & \quad  \textit{if $L_i$ is of parity type II}.\\
    \end{array} \right.\]
    When $i$ is odd,  we say that $L_i$ is
\[
\left\{
  \begin{array}{l l}
  \textit{of type II}    & \quad  \textit{if $E/F$ satisfies \textit{Case 1}};\\
  \textit{of type I}    &   \quad  \textit{if $E/F$ satisfies \textit{Case 2} and $A_i\varsupsetneq B_i$};\\
  \textit{of type II}    &   \quad  \textit{if $E/F$ satisfies \textit{Case 2} and $A_i=B_i$}.
    \end{array} \right.
\]

    In addition, when $i$ is even, we say that $L_i$ is
\[\left \{
  \begin{array}{l l}
  \textit{bound of type I} & \quad  \textit{if $L_i$ is \textit{of type I} and either $L_{i-2}$ or $L_{i+2}$ is \textit{of type I}};\\
  \textit{bound of type II} & \quad  \textit{if $L_i$ is \textit{of type II} and either $L_{i-1}$ or $L_{i+1}$ is \textit{of type I}};\\
 \textit{free} & \quad \textit{otherwise}.\\
    \end{array} \right.\]

When $i$ is odd, we say that $L_i$ is
\[\left \{
  \begin{array}{l l}
  \textit{bound} & \quad  \textit{if either $L_{i-1}$ or $L_{i+1}$ is \textit{of type I}};\\
 \textit{free} & \quad \textit{otherwise}.\\
    \end{array} \right.\]

Notice that the type of each $L_i$ is determined canonically regardless of the choice of a Jordan splitting.

\subsection{Sharpened Structure Theorem for Integral hermitian Forms}

While Theorem 2.2 lets us work with a restricted set of candidates for each $L_i$, further pruning is facilitated by the type of each $L_i$.
For this, we need a series of lemmas.
\begin{Lem} (Proposition 9.2 in \cite{J})
Let $L$ be a $\pi^i$-modular lattice of rank 2 with $n(L)=n(H(i))$.
Then $L\cong H(i)$
in \textit{Case 2} with $i$ odd and in \textit{Case 1} with $i$ even. \qed
\end{Lem}

\begin{Lem} (Proposition 4.4 in \cite{J})
A $\pi^i$-modular lattice $L$ has an orthogonal basis if $n(L)=s(L)$. \qed
\end{Lem}

\begin{Lem}
Assume that $E/F$ satisfies \textit{Case 2}.
\begin{enumerate}
\item Let $L=A(4a, 2\delta, \pi)\oplus (2c)$ with respect to a basis $(e_1, e_2, e_3)$, where $c\equiv 1 \mathrm{~mod~}2$. 
Then $L\cong H(1)\oplus (2c^{\prime})$ where $c^{\prime}\equiv 1 \mathrm{~mod~}2$.
\item Let $L=A(4a, 2\delta, \pi)\oplus (c)$ with respect to a basis $(e_1, e_2, e_3)$, where $c\equiv 1 \mathrm{~mod~}2$.
Then $L\cong H(1)\oplus (c^{\prime})$ where $c^{\prime}\equiv 1 \mathrm{~mod~}2$.
\end{enumerate}
\end{Lem}

\begin{proof}
For $(1)$, We work with the basis $(e_1-\frac{2a\pi}{\delta}e_2, e_2+e_3, \frac{c\pi}{\delta}e_1+e_3)$ of $L$.
With respect to this basis, $L\cong A(-4a-16a^2, 2(\delta+c), \pi(1+4a))\oplus (2c(1-\frac{4ac}{\delta}))$.
Moreover, $n(A(-4a-16a^2, 2(\delta+c), \pi(1+4a)))=n(H(1))=(4)$. This, by combining with the above lemma, completes the proof.

For $(2)$,
the sublattice of $L$ spanned by $(e_1, e_2, \pi e_3)$ is isomorphic to $A(4a, 2\delta, \pi)\oplus (2c')$ where $c'\equiv 1 \mathrm{~mod~}2$.
If we apply $(1)$ to this sublattice
by choosing a basis $(e_1-\frac{2a\pi}{\delta}e_2, e_2+\pi e_3, \frac{c'\pi}{\delta}e_1+\pi e_3)$,
then $A(4a, 2\delta, \pi)\oplus (2c')$  is isomorphic to  $H(1)\oplus (2c^{\prime\prime})$ where $c^{\prime\prime}\equiv 1 \mathrm{~mod~}2$.
Now the sublattice spanned by $(e_1-\frac{2a\pi}{\delta}e_2, e_2+\pi e_3, \frac{1}{\pi}(\frac{c'\pi}{\delta}e_1+\pi e_3))$ of $L$,
 which is the same as $L$, is isomorphic to $H(1)\oplus (-c^{\prime\prime}/\delta)$.
\end{proof}

The above lemmas will contribute to the proof of Theorem 2.10 below in the following manner.
For a given Jordan splitting $L=\bigoplus_i L_i$ in \textit{Case 2}, assume that $L_1$ is \textit{bound of type I}.
Theorem 2.2 tells us that
there are two different possibilities  for $L_1$ as a hermitian lattice and if $L_1=\oplus H(1)$
then the conclusion of the as yet unstated Theorem 2.10, for $i=1$, will follow.
If $L_1=\oplus H(1)\oplus A(4a, 2\delta, \pi)$
and either $L_0$ or $L_2$ is \textit{of type $I^o$}, then by Lemma 2.9 and the above paragraph, $L_0\oplus L_1 \oplus L_2=L_0'\oplus L_1' \oplus L_2'$ such that
$L_1'=\oplus H(1)$ and the type of $L_0$ or $L_2$ is the same as that of $L_0'$ or $L_2'$ respectively.
In case either $L_0$ or $L_2$ is \textit{of type $I^e$}, say $L_2$ is \textit{of type $I^e$},
 $L_2=(\oplus H(2))\oplus (2a)\oplus (2b)$ where $a, b\equiv 1 \mathrm{~mod~}2$ by Lemma 2.8.
Then we use Lemma 2.9 to $L_1 \oplus (2b)$ so that $L_1 \oplus (2b)=(\oplus H(1))\oplus (2b')$ with $b'\equiv 1 \mathrm{~mod~}2$.
Thus $L_1\oplus L_2=L_1'\oplus L_2'$ where $L_1'=\oplus H(1)$ and the type of $L_2'=(\oplus H(2))\oplus (2a)\oplus (2b')$ is the same as that of $L_2$.
In conclusion, $L=L_0'\oplus L_1' \oplus L_2'\oplus (\bigoplus_i L_i)$ is another Jordan splitting of $L$ and in this case, $L_1'=\oplus H(1)$.
Therefore, if $L_1$ is \textit{bound of type I} in \textit{Case 2},
then $L_1$ can always be replaced by $\oplus H(1)$.
Based on this combined with Theorem 2.2, we have the following structure theorem:

\begin{Thm}
There is a suitable choice of a Jordan splitting of the given lattice $L=\bigoplus_i L_i$ such that $L_i=\bigoplus_{\lambda}H_{\lambda}\oplus K$,
where each $H_{\lambda}= H(i)$ and $K$ is  $\pi^i$-modular of rank 1 or 2 with the following descriptions.
Let $i=0$ or $i=1$. Then

 \begin{enumerate}
 \item[a)] In \textit{Case 1},
\[
K=\left\{
  \begin{array}{l l}
  \textit{$(a)$ where $a \equiv 1$ mod 2}    & \quad  \textit{if $i=0$ and $L_0$ is \textit{of type $I^o$}};\\
  \textit{$A(1, 2b, 1)$}    & \quad  \textit{if $i=0$ and $L_0$ is \textit{of type $I^e$}};\\
  \textit{$H(0)$}    & \quad  \textit{if $i=0$ and $L_0$ is \textit{of type II}};\\
  \textit{$A(2, 2b, \pi)$}    & \quad  \textit{if $i=1$}.
      \end{array} \right.
\]
\item[b)] In \textit{Case 2},
\[
K=\left\{
  \begin{array}{l l}
  \textit{$(a)$ where $a \equiv 1$ mod 2}    & \quad  \textit{if $i=0$ and $L_0$ is \textit{of type $I^o$}};\\
  \textit{$A(1, 2b, 1)$}    & \quad  \textit{if $i=0$ and $L_0$ is \textit{of type $I^e$}};\\
  \textit{$A(2\delta, 2b, 1)$}    & \quad  \textit{if $i=0$ and $L_0$ is \textit{of type II}};\\
  \textit{$A(4a, 2\delta, \pi)$}    & \quad  \textit{if $i=1$ and $L_1$ is \textit{free of type I}};\\
  \textit{$H(1)$}    & \quad  \textit{if $i=1$, and $L_1$ is \textit{bound of type I} or \textit{of type II}}.
      \end{array} \right.
\]
Here, $a, b\in A$ and $\delta, \pi$ are explained in Section 2.1. \qed
\end{enumerate}
\end{Thm}

From now on, the pair $(L,h)$ is fixed throughout this paper.

\begin{Rmk}
Working with a basis furnished by Theorem 2.10, we can describe our lattices $A_i$ through $Z_i$ more explicitly.
We use the following conventions.
$\mathcal{L}_i$ denotes $\bigoplus_{j\neq i} \pi^{max\{0, i-j\}}L_j$.
Further, the $\oplus_{\lambda}H_{\lambda}$ will be denoted $\mathcal{H}_i$.
Theorem 2.10 involves a basis for a lattice $K$, which we will write as $\{e_1^{(i)}, e_2^{(i)}\}$ according to the ordering contained therein.
For all cases, we have $A_i=\mathcal{L}_i\oplus L_i$ and $X_i=\mathcal{L}_i\oplus \pi L_i$.

In order to write $W_i$, we should first find the vector $e\in A_i/X_i$ explained in the paragraph right after the definition of $B(L)$ in Section 2.3.
In order to simplify notations, let us work with one example.
Assume that $(e_1, e_2, e_3, e_4)$ is a $B$-basis of $L$ with respect to which $h$ is represented by the matrix
\[\begin{pmatrix} 0&1&0&0\\1&0&0&0\\0&0&1&1\\0&0&1&2 \end{pmatrix}.\]
So $L (=L_0=A_0)$ is of type $I^e$ and our basis is as explained in Theorem 2.10 of the paper.
Now, in order to find $W_0$, we should find the vector $e \in L/\pi L$  explained in section 2.3 (after the definition of $B(L)$).
If $v=(x, y, z, w)$ is a vector in $L/\pi L$, $h(v, v)$ mod $\pi$ $=z^2$.
On the other hand, if $e=(0, 0, 0, 1) \in L/\pi L$, then $(h(v, e))^2$ mod $\pi$ $=z^2$.
Therefore, by uniqueness of the vector $e$, we have that $(0, 0, 0, 1) \in L/\pi L$ is the vector $e$ we are looking for.

Since $W_0$ is the sublattice of $L$ such that $W_0/X_0=W_0/\pi L$ is the subspace of $L/ \pi L$ spanned by the vector $e$,
$W_0$ is spanned by  $(\pi e_1, \pi e_2, \pi e_3, e_4)$.
And it is easy to see that $B_0$ is spanned by $( e_1,  e_2, \pi e_3, e_4)$.

To describe all lattices, it is good to start with the matrix of our fixed hermitian form $h$ with respect to a basis furnished by Theorem 2.10.

\textbf{Case 1, $i$ even :}
For type \textit{I},  $e=(0, \cdots, 0, 1) \in A_i/X_i$.
\begin{center}
    \begin{tabular}{| l | l | l | l |}
    \hline
    Type & $B_i$ & $W_i$ & $Y_i$ \\ \hline
    $I^o$ & $\mathcal{L}_i\oplus \mathcal{H}_i \oplus (\pi) e_1^{(i)}$ & $\mathcal{L}_i\oplus \pi\cdot  \mathcal{H}_i \oplus Be_1^{(i)}$ & $X_i$ \\ \hline
  $I^e$ & $\mathcal{L}_i\oplus \mathcal{H}_i \oplus (\pi) e_1^{(i)}\oplus B e_2^{(i)}$ &
  $\mathcal{L}_i\oplus \pi \cdot \mathcal{H}_i \oplus (\pi)e_1^{(i)} \oplus B e_2^{(i)}$& $W_i$ \\ \hline
$II$ & $A_i$ & $X_i$ & $X_i$\\ \hline
    \end{tabular}
\end{center}
\textit{  }\\

\textbf{Case 1, $i$ odd :}  $B_i=A_i, W_i=X_i$, and $Y_i$ is not defined.
 $Z_i$ is a sublattice of $A_i$ and so we should have congruence conditions for $L_j$.
Namely,
\[Z_i=\bigoplus_{j\notin \{i\}\cup \mathcal{E}}\pi^{max\{0, i-j\}}L_j\oplus \pi L_i\oplus \bigoplus_{j\in \mathcal{E}}\pi^{max\{0, i-j\}}\left(\mathcal{H}_j\oplus Be_2^{(j)}\right)\]
\[\oplus \left\{\sum_{j\in \mathcal{E}}\pi^{max\{0, i-j\}}\cdot a_je_1^{(j)}|\textit{each }a_j\in B, \sum_{j\in \mathcal{E}}a_j\in (\pi)\right\}.\]
Here, $\mathcal{E}=\{j\in \{i-1,  i+1\}| \textit{$L_j$ is of type I}\}$
and the $e_2^{(j)}$ factor should be ignored for those $j \in \mathcal{E}$ such that $L_j$ is of type $I^o$.\\

The following example would be helpful to have a better understanding of the notions of  `bound' and `free'
 and of the notion of type when $i$ is odd.
Let $L=L_1\oplus L_2=A(0, 0, \pi)\oplus (2)$, so that $L_1$ is bound of type $I$ (since $A_1\neq B_1$) and $L_2$ is free of type $I$.\\

\textbf{Case 2, $i$ even :}
 $B_i, W_i,$ and $Y_i$ are exactly as in the table given for Case 1.
The lattice $Z_i$ is a little complicated.
Note that when $L_i$ is of type $I$ or bound of type $II$, the dimension of $Y_i/Z_i$ as a $\kappa$-vector space is 1.
We describe it case by case below.

$\bullet$ Let $\mathcal{E}'=\{j\in \{i-2,  i+2\}| \textit{$L_j$ is of type I}\}$.
If $L_i$ is of type $I$ so that $L_{i-1}$ and $L_{i+1}$ are bound,
\[Z_i=\bigoplus_{j\notin \{i, i\pm 2\}}\pi^{max\{0, i-j\}}L_j\oplus \bigoplus_{j\in \{i\pm 2\}}\pi^{max\{0, i-j\}}\left(\mathcal{H}_j\oplus Be_2^{(j)}\right)
\oplus \pi \mathcal{H}_i\]
\[\oplus \left\{\left(\sum_{j\in \mathcal{E}'}\pi^{max\{0, i-j\}}\cdot a_je_1^{(j)}\right)+\left(\pi\cdot a_ie_1^{(i)}+b\cdot b_ie_2^{(i)}\right)
|\textit{each }a_j\in B, \left(\sum_{j\in \mathcal{E}'}a_j\right)+a_i+b\cdot b_i\in (\pi)\right\},\]
where the $e_2^{(j)}$ (resp. $e_2^{(i)}$) factor should be ignored for those $j \in\{i\pm 2\}$ (resp. $i$) such that $L_j$ (resp. $ L_i$) is not of type $I^e$,
and $b\in B$ is such that $L_i=\pi^{i/2}\left(\mathcal{H}_i\oplus A(1, 2b, 1)\right)$ when $L_i$ is of type $I^e$.

$\bullet$ If $L_i$ is free of type $II$ (so that all of $L_{i\pm 2}$ and $L_{i\pm 1}$ are of type $II$), $Z_i=X_i$.

$\bullet$ If $L_i$ is bound of type $II$, then
\[Z_i=\bigoplus_{j\notin \{i, i\pm 1, i\pm 2\}}\pi^{max\{0, i-j\}}L_j\oplus \pi L_i
\oplus \bigoplus_{j\in \{i\pm 1\}}\pi^{max\{0, i-j\}}\left(\mathcal{H}_j\oplus Be_1^{(j)}\right)
\oplus \bigoplus_{j\in \{i\pm 2\}}\pi^{max\{0, i-j\}}\left(\mathcal{H}_j\oplus Be_2^{(j)}\right)
\]
\[\oplus \left\{\left(\sum_{j\in \mathcal{E}_1}\pi^{max\{0, i-j\}}\cdot a_je_2^{(j)}\right)+
\left(\sum_{j\in \mathcal{E}_2}\pi^{max\{0, i-j\}}\cdot a_je_1^{(j)}\right)
|\textit{each }a_j\in B, \left(\sum_{j\in \mathcal{E}_1\cup \mathcal{E}_2}a_j\right)\in (\pi)\right\},\]
where
$\mathcal{E}_1=\{j\in \{i-1, i+1\}| \textit{$L_j$ is free of type I}\}$ and
$\mathcal{E}_2=\{j\in \{i-2, i+2\}| \textit{$L_j$ is of type I}\}$.
For example, if $i+1\in \mathcal{E}_1$, then $i+2\notin \mathcal{E}_2$.
And if $i+2\in \mathcal{E}_2$, then $i+1\notin \mathcal{E}_1$.
\\

\textbf{Case 2, $i$ odd :} In this case, $W_i=X_i$ and $Z_i$ is not defined.
\begin{center}
    \begin{tabular}{| l | l | l  |}
    \hline
    Type & $B_i$ &  $Y_i$ \\ \hline
    $\textit{free of type I}$ & $\mathcal{L}_i\oplus \mathcal{H}_i \oplus B e_1^{(i)}\oplus (\pi) e_2^{(i)}$
    & $\mathcal{L}_i\oplus \pi \mathcal{H}_i \oplus B e_1^{(i)}\oplus (\pi) e_2^{(i)}$ \\ \hline
  $\textit{bound of type I}$ & \textit{see below} &
  \textit{see below} \\ \hline
$\textit{type II}$ & $A_i$ & $X_i$\\ \hline
    \end{tabular}
\end{center}
\textit{  }\\
When $L_i$ is bound of type $I$, the dimension of $A_i/B_i$ as $\kappa$-spaces is 1.
\[B_i=\bigoplus_{j\notin \{i\}\cup \mathcal{E}}\pi^{max\{0, i-j\}}L_j\oplus  L_i\oplus \bigoplus_{j\in \mathcal{E}}\pi^{max\{0, i-j\}}\left(\mathcal{H}_j\oplus Be_2^{(j)}\right)\]
\[\oplus \left\{\sum_{j\in \mathcal{E}}\pi^{max\{0, i-j\}}\cdot a_je_1^{(j)}|\textit{each }a_j\in B, \sum_{j\in \mathcal{E}}a_j\in (\pi)\right\},\]
\[Y_i=\bigoplus_{j\notin \{i\}\cup \mathcal{E}}\pi^{max\{0, i-j\}}L_j\oplus  \pi L_i\oplus \bigoplus_{j\in \mathcal{E}}\pi^{max\{0, i-j\}}\left(\mathcal{H}_j\oplus Be_2^{(j)}\right)\]
\[\oplus \left\{\sum_{j\in \mathcal{E}}\pi^{max\{0, i-j\}}\cdot a_je_1^{(j)}|\textit{each }a_j\in B, \sum_{j\in \mathcal{E}}a_j\in (\pi)\right\}.\]
Here, $\mathcal{E}=\{j\in \{i-1, i+1\}| \textit{$L_j$ is of type I}\}$.
\end{Rmk}

\section{The construction of the smooth model}

Let $\underline{G}^{\prime}$ be the naive integral model of the unitary group $\mathrm{U}(V, h)$, where $V=L\otimes_AF$, such that
for any commutative $A$-algebra $R$,
$$\underline{G}^{\prime}(R)=\mathrm{Aut}_{B\otimes_AR}(L\otimes_AR, h\otimes_AR).$$
The scheme $\underline{G}^{\prime}$ is then an (possibly non-smooth) affine group scheme over $A$ with the smooth generic fiber $\mathrm{U}(V, h)$.
Then by Proposition 3.7 in \cite{GY}, there exists a unique smooth integral model, denoted by  $\underline{G}$, with the generic fiber $\mathrm{U}(V, h)$,
characterized by
$$\underline{G}(R)=\underline{G}^{\prime}(R)$$
for any \'etale $A$-algebra $R$.
Note that every \'etale $A$-algebra is a finite product of finite unramified extensions of $A$.
This section, Section 4 and Appendix A  are devoted to gaining an explicit knowledge of the smooth integral model
$\underline{G}$ in \textit{Case 1},
 which will be used in Section 5 to compute the local density of $(L, h)$ (again, in \textit{Case 1}).
For a detailed exposition of the relation between the local density of $(L, h)$ and $\underline{G}$, see Section 3 of \cite{GY}.

In this section, we give an explicit construction of the smooth integral model $\underline{G}$
when $E/F$ satisfies \textit{Case 1}.
The construction of $\underline{G}$ is based on that of Section 5 in \cite{GY} and Section 3 in \cite{C1}.
Since the functor $R \mapsto \underline{G}(R)$ restricted to \'etale $A$-algebras $R$ determines $\underline{G}$,
we first list out some properties that are satisfied by each element of $\underline{G}(R)=\underline{G}^{\prime}(R)$.

We choose an element $g\in \underline{G}(R)$ for an  \'etale $A$-algebra $R$.
Then $g$ is an element of $\mathrm{Aut}_{B\otimes_AR}(L\otimes_AR, h\otimes_AR)$.
Here we consider $\mathrm{Aut}_{B\otimes_AR}(L\otimes_AR, h\otimes_AR)$ as a subgroup of $ \mathrm{Res}_{E/F}\mathrm{GL}_E(V)(F\otimes_AR)$.
To ease the notation, we say $g\in \mathrm{Aut}_{B\otimes_AR}(L\otimes_AR, h\otimes_AR)$ stabilizes
a lattice $M\subseteq V$ if $g(M\otimes_AR)=M\otimes_AR$.


\subsection{Main construction}
Let $R$ be an \'etale $A$-algebra.
    In this subsection, as mentioned above, we observe properties of elements of $\mathrm{Aut}_{B\otimes_AR}(L\otimes_AR, h\otimes_AR)$
     and their matrix interpretation.
    We choose a Jordan splitting $L=\bigoplus_iL_i$ and a basis of $L$ as explained in Theorem 2.10 and Remark 2.3.a).
     Let $n_i=\mathrm{rank}_{B}L_i$, and $n=\mathrm{rank}_{B}L=\sum n_i$.
 Assume that $n_i=0$ unless $0\leq i < N$. 
    Let $g$ be an element of $\mathrm{Aut}_{B\otimes_AR}(L\otimes_AR, h\otimes_AR)$.
 We always divide a matrix $g$ of size $n \times n$ into $N^2$ blocks such that the block in position $(i, j)$ is of size $n_i\times n_j$.
 For simplicity, the row and column numbering starts at $0$ rather than $1$.

   \begin{enumerate}
    \item[(1)] First of all, $g$ stabilizes $A_i$ for every integer $i$.
In terms of matrices, this fact means that the $(i,j)$-block
 has entries in $\pi^{max\{0,j-i\}}B\otimes_AR$.
 From now on, we write
\[g= \begin{pmatrix} \pi^{max\{0,j-i\}}g_{i,j} \end{pmatrix}.\]

\item[(2)]
   $g$ stabilizes $A_i, B_i, W_i, X_i$ and induces the identity on $A_i/B_i$ and $W_i/X_i$.
     We also interpret these facts in terms of matrices as described below:
\begin{itemize}
\item[a)]      If $i$ is odd or $L_i$ is \textit{of type II},
then $A_i=B_i$ and  $W_i=X_i$ and so there is no contribution. 
\item[b)]     If $L_i$ is \textit{of type} $\textit{I}^o$, the diagonal $(i,i)$-block $g_{i,i}$ is of the form
     \[\begin{pmatrix} s_i&\pi y_i\\ \pi v_i&1+\pi z_i \end{pmatrix}\in \mathrm{GL}_{n_i}(B\otimes_AR),\]
     where $s_i$ is an $(n_i-1) \times (n_i-1)-$matrix, etc.
\item[c)]
     If $L_i$ is \textit{of type} $\textit{I}^e$, the diagonal $(i,i)$-block $g_{i,i}$ is of the form
     \[\begin{pmatrix} s_i&r_i&\pi t_i\\ \pi y_i&1+\pi x_i&\pi z_i\\ v_i&u_i&1+\pi w_i \end{pmatrix}\in \mathrm{GL}_{n_i}(B\otimes_AR),\]
     where $s_i$ is an $(n_i-2) \times (n_i-2)-$matrix, etc.
\end{itemize}
\end{enumerate}

 \subsection{Construction of \textit{\underline{M}}}



    We define a functor from the category of commutative flat $A$-algebras to the category of monoids as follows. For any commutative flat $A$-algebra $R$, define
    $$\underline{M}(R) \subset \{m \in \mathrm{End}_{B\otimes_AR}(L \otimes_A R)\}$$
    to be the set of $m \in \mathrm{End}_{B\otimes_AR}(L \otimes_A R)$ satisfying the following conditions:
    \begin{itemize}
     \item[(1)]  $m$ stabilizes $A_i\otimes_A R,B_i\otimes_A R,W_i\otimes_A R,X_i\otimes_A R$ for all $i$.
     \item[(2)]  $m$ induces the identity on $A_i\otimes_A R/ B_i\otimes_A R, W_i\otimes_A R/X_i\otimes_A R$ for all $i$.
\end{itemize}

\begin{Rmk}
We give another description for the functor $\underline{M}$ and using this, we show that it is represented by a polynomial ring.
Let us define a functor from the category of commutative flat $A$-algebras to the category of rings as follows:

For any commutative flat $A$-algebra $R$, define
$$\underline{M}^{\prime}(R) \subset \{m \in \mathrm{End}_{B\otimes_AR}(L \otimes_A R) \}$$
to be the set of $m \in \mathrm{End}_{B\otimes_AR}(L \otimes_A R)$ satisfying the following conditions:
\begin{itemize}
 \item[(1)] $m$ stabilizes $A_i\otimes_A R,B_i\otimes_A R,W_i\otimes_A R,X_i\otimes_A R$ for all $i$.
 \item[(2)] $m $ maps $A_i\otimes_A R, W_i\otimes_A R$ into $B_i\otimes_A R, X_i\otimes_A R$, respectively.
\end{itemize}

Then, by  Lemma 3.1 of \cite{C1}, $\underline{M}^{\prime}$ is  represented by a unique flat $A$-algebra $A(\underline{M'})$ which is a polynomial ring over $A$ of $2n^2$ variables.
    Moreover, it is easy to see that $\underline{M}'$ has the structure of a scheme of rings since
    $\underline{M}'(R)$ is  closed under addition and  multiplication.

We consider a scheme $\mathrm{Res}_{B/A}\mathrm{End}_B(L)$ such that the associated set to a commutative flat $A$-algebra $R$
    is $\mathrm{End}_{B\otimes_AR}(L \otimes_A R)$.
    Indeed, $\mathrm{Res}_{B/A}\mathrm{End}_B(L)$ is a group scheme under addition.
    But at this moment, we consider it as a scheme of sets
    so as to embed $\underline{M}$ into this.
    Let us consider both  $\underline{M}$ and $\underline{M}'$
    as functors from the category of commutative flat $A$-algebras to the category of sets.
    Then  they are subfunctors of $ \mathrm{Res}_{B/A}\mathrm{End}_B(L)$. 
Furthermore the functor $\underline{M}$ (viewed as valued in sets) is the same as the functor $1+\underline{M}^{\prime}$,
    where $(1+\underline{M}^{\prime})(R)=\{1+m : m \in \underline{M}^{\prime}(R) \}$.
    Here, the set $\mathrm{End}_{B\otimes_AR}(L \otimes_A R)$ has an obvious additive structure
    and the addition in the description of $(1+\underline{M}^{\prime})(R)$ comes from this.

    Therefore, $\underline{M}$ and $\underline{M}'$ are equivalent, as subfunctors of $ \mathrm{Res}_{B/A}\mathrm{End}_B(L)$.
    This fact induces that
   the functor $\underline{M}$ is also represented by a unique flat $A$-algebra $A[\underline{M}]$ which is a polynomial ring over $A$ of $2n^2$ variables.
    Moreover, it is easy to see that $\underline{M}$ has the structure of a scheme of monoids since $\underline{M}(R)$ is  closed under multiplication.
\end{Rmk}

   We can therefore now talk of $\underline{M}(R)$ for any (not necessarily flat) $A$-algebra $R$.
  However, for a general $R$, the above description for  $\underline{M}(R)$ will  no longer be true.
For such $R$, we use our chosen basis of $L$ to write each element of $\underline{M}(R)$ formally.
 We describe each element of $\underline{M}(R)$  as a formal matrix
    $\begin{pmatrix} \pi^{max\{0,j-i\}}m_{i,j} \end{pmatrix}$.
    Here,  $m_{i,j}$, when $i\neq j$,  is an  $(n_i \times n_j)$-matrix with entries in $B\otimes_AR$ and
    \[
m_{i,i}=\left\{
  \begin{array}{l l}
  \begin{pmatrix} s_i&\pi y_i\\ \pi v_i&1+\pi z_i \end{pmatrix}    & \quad  \textit{if $i$ is even and $L_i$ is \textit{of type $I^o$}};\\
  \begin{pmatrix} s_i&r_i&\pi t_i\\ \pi y_i&1+\pi x_i&\pi z_i\\ v_i&u_i&1+\pi w_i \end{pmatrix}    & \quad  \textit{if $i$ is even and $L_i$ is \textit{of type $I^e$}};\\
  m_{i,i}   & \quad  \textit{otherwise, i.e. if $L_i$ is of type $II$}.
      \end{array} \right.
\]
Here, $s_i$  is an  $(n_i-1 \times n_i-1)$-matrix (resp. $(n_i-2 \times n_i-2)$-matrix) with entries in $B\otimes_AR$
if $L_i$ \textit{of type $I^o$} (resp. \textit{of type $I^e$}) and $y_i, v_i, z_i, r_i, t_i, y_i, x_i, u_i, w_i$ are matrices of suitable sizes with entries in $B\otimes_AR$.
Similarly, if  $L_i$ is \textit{of type II}, then
$m_{i,i}$ is an  $(n_i \times n_j)$-matrix with entries in $B\otimes_AR$.
To simplify notation, each element
$$((m_{i,j})_{i\neq j}, (m_{i,i})_{\textit{$L_i$ of type II}}, (s_i, y_i, v_i, z_i)_{\textit{$L_i$ of type $I^o$}},
(s_i, v_i, z_i, r_i, t_i, y_i, x_i, u_i, w_i)_{\textit{$L_i$ of type $I^e$}})$$ of  $\underline{M}(R)$
is denoted by $(m_{i,j}, s_i \cdots w_i)$.\\

    In the next section, we need a description of an element of $\underline{M}(R)$ and its multiplication for a $\kappa$-algebra $R$.
    In order to  prepare for this, we describe the multiplication explicitly only for a $\kappa$-algebra $R$.
    To multiply $(m_{i,j}, s_i\cdots w_i)$ and $(m_{i,j}', s_i'\cdots w_i')$,
    we form the matrices $m=\begin{pmatrix} \pi^{max\{0,j-i\}}m_{i,j} \end{pmatrix}$ and $m'=\begin{pmatrix} \pi^{max\{0,j-i\}}m_{i,j}' \end{pmatrix}$
    with $s_i\cdots w_i$ and $s_i'\cdots w_i'$
     and write the formal matrix product  $\begin{pmatrix} \pi^{max\{0,j-i\}}m_{i,j} \end{pmatrix}\cdot \begin{pmatrix} \pi^{max\{0,j-i\}}m_{i,j}' \end{pmatrix}=\begin{pmatrix} \pi^{max\{0,j-i\}}\tilde{m}_{i,j}'' \end{pmatrix}$
     with
    \[
\tilde{m}_{i,i}''=\left\{
  \begin{array}{l l}
  \begin{pmatrix} \tilde{s}_i''&\pi \tilde{y}_i''\\ \pi \tilde{v}_i''&1+\pi \tilde{z}_i'' \end{pmatrix}    & \quad  \textit{if $i$ is even and $L_i$ is \textit{of type $I^o$}};\\
  \begin{pmatrix} \tilde{s}_i''&\tilde{r}_i''&\pi \tilde{t}_i''\\ \pi \tilde{y}_i''&1+\pi \tilde{x}_i''&\pi \tilde{z}_i''\\ \tilde{v}_i''&\tilde{u}_i''&1+\pi \tilde{w}_i'' \end{pmatrix}    & \quad  \textit{if $i$ is even and $L_i$ is \textit{of type $I^e$}}.
      \end{array} \right.
\]
Let $(m_{i,j}'', s_i''\cdots w_i'')$ be formed by letting $\pi^2$ be zero
in each entry of $(\tilde{m}_{i,j}'', \tilde{s}_i''\cdots \tilde{w}_i'')$.
Then each matrix of $(m_{i,j}'', s_i''\cdots w_i'')$ has entries in $B\otimes_AR$ and so $(m_{i,j}'', s_i''\cdots w_i'')$ is an element of $\underline{M}(R)$ and is the product of $(m_{i,j}, s_i\cdots w_i)$ and $(m_{i,j}', s_i'\cdots w_i')$.
More precisely,
\begin{enumerate}
\item If $i\neq j$
    or if $i= j$ and $L_i$ is \textit{of type $II$},
    $$m_{i,j}''=\sum_{k=1}^{N}\pi^{(max\{0, k-i\}+max\{0, j-k\}-max\{0, j-i\})}m_{i, k}m_{k, j}';$$

\item When $L_i$ is \textit{of type $I^o$}, we write $m_{i, i-1}m_{i-1, i}'+m_{i, i+1}m_{i+1, i}'=\begin{pmatrix} a_i''&b_i''\\ c_i''&d_i'' \end{pmatrix}$
and $m_{i, i-2}m_{i-2, i}'+m_{i, i+2}m_{i+2, i}'=\begin{pmatrix} \tilde{a}_i''&\tilde{b}_i''\\ \tilde{c}_i''&\tilde{d}_i'' \end{pmatrix}$
where
$a_i''$ and $\tilde{a}_i''$ are  $(n_i-1) \times (n_i-1)$-matrices, etc.
Then
\[\left\{
  \begin{array}{l}
  s_i''=s_is_i'+\pi a_i'';\\
  y_i''=s_iy_i'+y_i+b_i''+\pi (y_iz_i'+ \tilde{b}_i'');\\
  v_i''=v_is_i'+v_i'+c_i''+\pi (z_iv_i'+ \tilde{c}_i'');\\
  z_i''=z_i+z_i'+d_i''+\pi (z_iz_i'+ v_iy_i'+ \tilde{d}_i'').
      \end{array} \right.
\]
\item When $L_i$ is \textit{of type $I^e$}, we write $m_{i, i-1}m_{i-1, i}'+m_{i, i+1}m_{i+1, i}'=
\begin{pmatrix} a_i''&b_i''&c_i''\\ d_i''&e_i''&f_i''\\ g_i''&h_i''&k_i'' \end{pmatrix}$
and $m_{i, i-2}m_{i-2, i}'+m_{i, i+2}m_{i+2, i}'=
\begin{pmatrix} \tilde{a}_i''&\tilde{b}_i''&\tilde{c}_i''\\ \tilde{d}_i''&\tilde{e}_i''&\tilde{f}_i''\\ \tilde{g}_i''&\tilde{h}_i''&\tilde{k}_i'' \end{pmatrix}$
where
$a_i''$ and $\tilde{a}_i''$ are  $(n_i-2) \times (n_i-2)$-matrices, etc.
Then
\[\left\{
  \begin{array}{l}
  s_i''=s_is_i'+\pi (r_iy_i'+t_iv_i'+a_i'');\\
  r_i''=s_ir_i'+r_i+\pi (r_ix_i'+t_iu_i'+b_i'')   ;\\
  t_i''=s_it_i'+r_iz_i'+t_i+c_i''+\pi (t_iw_i'+\tilde{c}_i'');\\
y_i''=y_is_i'+y_i'+z_iv_i'+d_i''+\pi (x_iy_i'+\tilde{d}_i'');\\
x_i''=x_i+x_i'+z_iu_i'+y_ir_i'+e_i''+\pi (x_ix_i'+\tilde{e}_i'');\\
z_i''=z_i+z_i'+f_i''+\pi (y_it_i'+x_iz_i'+z_iw_i'+\tilde{f}_i'');\\
v_i''=v_is_i'+v_i'+\pi (u_iy_i'+w_iv_i'+g_i'');\\
u_i''=u_i+u_i'+v_ir_i'+\pi(u_ix_i'+w_iu_i'+h_i'');\\
w_i''=w_i+w_i'+v_it_i'+u_iz_i'+k_i''+\pi (w_iw_i'+\tilde{k}_i'').
      \end{array} \right.
\]
\end{enumerate}

\textit{  }\\



\begin{Rmk}
Let $d$ be the determinant homomorphism on the algebraic monoid $\mathrm{Res}_{B/A}\mathrm{End}_B(L)$.
We consider the inclusion
$$\iota : \underline{M} \longrightarrow \mathrm{Res}_{B/A}\mathrm{End}_B(L)$$
between functors  of sets on the category of commutative flat $A$-algebras.
Note that this inclusion is a morphism of schemes by \textit{Yoneda's lemma} since $\underline{M}$ is flat over $A$.
 It is not an immersion as schemes since the special fiber of $\underline{M}$ is no longer embedded into
that of $\mathrm{Res}_{B/A}\mathrm{End}_B(L)$.
For a commutative flat $A$-algebra $R$, the multiplication on $\underline{M}(R)$ is induced from that on
$\mathrm{Res}_{B/A}\mathrm{End}_B(L)(R)$ under $\iota$.
Thus the morphism $\iota$ is a morphism of monoid schemes.

We consider $d$ as the restriction of the determinant homomorphism under  $\iota$. 
 Then $\mathrm{Spec} (A[\underline{M}]_d)$ is an open subscheme of $\underline{M}$, where
 $A[\underline{M}]_d$ is the localization of the ring $A[\underline{M}]$ at $d$.
Note that   $\mathrm{Spec} (A[\underline{M}]_d)(R)$, the set of $R$-points of $\mathrm{Spec} (A[\underline{M}]_d)$
for a commutative $A$-algebra $R$, is characterized by
 $$\{m\in\underline{M}(R):
   \textit{there exists $\widetilde{m}' \in \mathrm{End}_{B\otimes_AR}(L\otimes_AR)$ such that $\iota_R(m)\cdot \widetilde{m}'=\widetilde{m}'\cdot \iota_R(m)=1$} \}.$$
Here, $\iota_R : \underline{M}(R) \rightarrow \mathrm{Res}_{B/A}\mathrm{End}_B(L)(R)$ is a morphism of monoids  induced from $\iota$.
It is easy to see that the above set $\mathrm{Spec} (A[\underline{M}]_d)(R)$ is a monoid so that $\mathrm{Spec} (A[\underline{M}]_d)$
is a scheme of monoids.

 We define a functor $\underline{M}^{\ast}$ from the category of commutative $A$-algebras to the category of groups as follows.
For a commutative $A$-algebra $R$,
 set
 $$\underline{M}^{\ast}(R)=\{ m \in \underline{M}(R) :  \textit{there exists $m'\in \underline{M}(R)$ such that $m\cdot m'=m'\cdot m=1$}\}.$$
 We claim that $\underline{M}^{\ast}$ is representable by $\mathrm{Spec} (A[\underline{M}]_d)$.
 It is obvious that $\underline{M}^{\ast}(R) \subseteq \mathrm{Spec} (A[\underline{M}]_d)(R)$ for any commutative $A$-algebra $R$.

In order to show  $\mathrm{Spec} (A[\underline{M}]_d)(R)  \subseteq \underline{M}^{\ast}(R)$,
we first prove that  $\widetilde{m}' ~ (\in \mathrm{End}_{B\otimes_AR}(L\otimes_AR))$ associated to $m\in \underline{M}(R)$ is an element of $\underline{M}(R)$ for every flat $A$-algebra $R$.
 To verify this statement, it suffices to show that $\widetilde{m}'$ satisfies conditions (1) and (2) defining $\underline{M}$.
 This follows from the following fact:
If $L'$ is a sublattice of $L$ and $m$ is an element of $\mathrm{Spec} (A[\underline{M}]_d)(R)$ for a flat $A$-algebra $R$
which stabilizes $L'\otimes_AR$, then $L'\otimes_AR$ is stabilized by $\widetilde{m}'$ as well.
This can be easily proved as in Lemma 3.2 of \cite{C1} and so we skip the proof.
 Thus $\underline{M}^{\ast}(R)$ is the same as $\mathrm{Spec} (A[\underline{M}]_d)(R)$ for a flat $A$-algebra $R$.
 In order to show   $\underline{M}^{\ast}(R)=\mathrm{Spec} (A[\underline{M}]_d)(R)$ for any commutative $A$-algebra $R$,
 we consider the following  well-defined map,  for any flat $A$-algebra $R$:
$$\mathrm{Spec} (A[\underline{M}]_d)(R) \longrightarrow \mathrm{Spec} (A[\underline{M}]_d)(R)\times \mathrm{Spec} (A[\underline{M}]_d)(R),
 m \mapsto (m, \widetilde{m}').$$
Since $\mathrm{Spec} (A[\underline{M}]_d)$ is flat, this map is represented by a morphism of schemes by \textit{Yoneda's lemma}.
On the other hand, since $\mathrm{Spec} (A[\underline{M}]_d)$ is a scheme of monoids,
the following map
\[\mathrm{Spec} (A[\underline{M}]_d)(R)\times \mathrm{Spec} (A[\underline{M}]_d)(R) \longrightarrow \mathrm{Spec} (A[\underline{M}]_d)(R),
(m, m')\mapsto mm'\]
is represented by a morphism of schemes.
We consider the composite of these two morphisms. It is the constant map (at the identity) at least at the level of $R$-points, for a flat $A$-algebra $R$.
To show that the composite is the constant morphism of schemes (at the identity),
it suffices to show that it is uniquely determined at the level of  $R$-points, for a flat $A$-algebra $R$.
 Note that $\mathrm{Spec} (A[\underline{M}]_d)$ is an irreducible smooth affine scheme.
We consider the open subscheme of  $\mathrm{Spec} (A[\underline{M}]_d)$
which is the complement of the closed subscheme of $\mathrm{Spec} (A[\underline{M}]_d)$ determined by the prime ideal $(2)$.
This open subscheme of $\mathrm{Spec} (A[\underline{M}]_d)$ is then non-empty and dense
since $\mathrm{Spec} (A[\underline{M}]_d)$ is reduced and irreducible.
Furthermore, all $R$-points of $\mathrm{Spec} (A[\underline{M}]_d)$, for a flat $A$-algebra $R$, factor through  this open subscheme.
Since a morphism of schemes is continuous, the above composite is uniquely determined
at the level of $R$-points, for a flat $A$-algebra $R$.

Thus, the inverse of $m\in \mathrm{Spec} (A[\underline{M}]_d)(R)$, for any commutative $A$-algebra $R$,
 is also contained in $\mathrm{Spec} (A[\underline{M}]_d)(R) \subseteq \underline{M}(R)$.
  This fact implies $\underline{M}^{\ast}(R) \supseteq \mathrm{Spec} (A[\underline{M}]_d)(R)$.
 Consequently, for any commutative $A$-algebra $R$, we have
 \[\underline{M}^{\ast}(R)=\mathrm{Spec} (A[\underline{M}]_d)(R).\]


Therefore, we conclude that $\underline{M}^{\ast}$ is an open subscheme of $\underline{M}$
(since $\underline{M}^{\ast}=\mathrm{Spec} (A[\underline{M}]_d)$, which is an open subscheme of $\underline{M}$),
with generic fiber $M^{\ast}=\mathrm{Res}_{E/F}\mathrm{GL}_E(V)$,
and that $\underline{M}^{\ast}$ is smooth over $A$. Moreover, $\underline{M}^{\ast}$ is a group scheme since $\underline{M}$ is a scheme in monoids.
\end{Rmk}

\subsection{Construction of \textit{\underline{H}}}

 Recall that the pair $(L, h)$ is fixed throughout this paper
  and the lattices $A_i$, $B_i$, $W_i$, $X_i$ only depend on the hermitian pair $(L, h)$.
  For any flat $A$-algebra $R$, let $\underline{H}(R)$ be the set of hermitian forms $f$ on $L\otimes_{A}R$ (with values in $B\otimes_AR$)
  such that $f$ satisfies the following conditions:
  \begin{enumerate}
 \item[a)]  $f(L\otimes_{A}R,A_i\otimes_{A}R) \subset \pi^iB\otimes_AR$ for all $i$.
 \item[b)] Let $i=2m$.  $\frac{1}{(\pi\cdot\sigma(\pi))^m}f(a_i,a_i)$ mod 2 = $\frac{1}{(\pi\cdot\sigma(\pi))^m} h(a_i, a_i)$ mod 2, where $a_i \in A_i \otimes_{A}R$.
 \item[c)] Let $i=2m$. $\frac{1}{\pi^i}f(a_i,w_i) = \frac{1}{\pi^i}h(a_i, w_i)$ mod $\pi$, where $a_i \in A_i\otimes_{A}R$ and $w_i \in W_i\otimes_{A}R$.
\end{enumerate}
\textit{ }

We interpret the above conditions in terms of matrices.
The matrix forms are taken with respect to the basis of $L$ fixed in Theorem 2.10 and Remark 2.3.a).
A matrix form of the given hermitian form $h$ is described in Remark 3.3.(1) below.
We use $\sigma$ to mean the automorphism of $B\otimes_AR$ given by $b\otimes r \mapsto \sigma(b)\otimes r$.
For a flat $A$-algebra $R$, $\underline{H}(R)$ is
the set of hermitian matrices $$\begin{pmatrix}\pi^{max\{i,j\}}f_{i,j}\end{pmatrix}$$ of size $n\times n$ satisfying the following:
 \begin{enumerate}
\item[(1)]  $f_{i,j}$ is an $(n_i\times n_j)$-matrix with entries in $B\otimes_AR$.
\item[(2)] If $i$ is even and $L_i$ is \textit{of type} $\textit{I}^o$,
then $\pi^i f_{i,i}$ is of the form $$(\pi\cdot \sigma(\pi))^{i/2}\begin{pmatrix} a_i&\pi b_i\\ \sigma(\pi\cdot {}^t b_i) &1 +2c_i \end{pmatrix}.$$
 Here, the diagonal entries of $a_i$ are divisible by $2$, where $a_i$ is an $(n_i-1) \times (n_i-1)$-matrix with entries in $B\otimes_AR$, etc.
\item[(3)]  If $i$ is even and $L_i$ is \textit{of type} $\textit{I}^e$, then $\pi^i f_{i,i}$ is of the form
  $$(\pi\cdot \sigma(\pi))^{i/2}\begin{pmatrix} a_i&b_i&\pi e_i\\ \sigma({}^tb_i) &1+2f_i&1+\pi d_i \\ \sigma(\pi \cdot {}^te_i) &\sigma(1+\pi d_i) &2c_i \end{pmatrix}.$$
 Here,  the diagonal entries of $a_i$ are  divisible by $2$, where $a_i$ is an $(n_i-2) \times (n_i-2)$-matrix with entries in $B\otimes_AR$, etc.
\item[(4)]  Assume that $L_i$ is \textit{of type} $\textit{II}$.
The diagonal entries of $f_{i,i}$ (resp. $\pi f_{i,i}$) are  divisible by $2$ if $i$ is even (resp. odd).
\item[(5)] Since $\begin{pmatrix}\pi^{max\{i,j\}}f_{i,j}\end{pmatrix}$ is a hermitian matrix, its diagonal entries are fixed by the nontrivial Galois action over $E/F$ and hence belong to $R$.
\end{enumerate}

Let us consider the \textit{hermitian functor}  from the category of commutative flat $A$-algebras to the category of sets
such that the associated set to $R$  is the set of hermitian forms $f$ on $L\otimes_{A}R$ (with values in $B\otimes_AR$).
Indeed, this functor is represented by a commutative group scheme since it is closed under addition.
Then $\underline{H}$ is a subfunctor of the \textit{hermitian functor}.
We consider another functor $\underline{H}'$ such that $\underline{H}'(R)=\{f-h : f \in \underline{H}(R) \}$.
Note that $h$ is the fixed hermitian form and the notion of $f-h$ follows from the additive structure of the \textit{hermitian functor}.
For a matrix interpretation of $h$, we refer to Remark 3.3.(1) below.

  Then by Lemma 3.1 of \cite{C1},  $\underline{H}'$ is represented by a flat $A$-scheme
  which is isomorphic to an affine space.
 Since  $\underline{H}$ and $\underline{H}'$ are equivalent as subfunctors of the \textit{hermitian functor},
   the functor $\underline{H}$ is also represented by a flat $A$-scheme
  which is isomorphic to an affine space.

  To compute the dimension of $\underline{H}$,
  we  see that each entry of the upper triangular matrix of an element of  $\underline{H}(R)$, for a flat $A$-algebra $R$,
 gives two variables and each diagonal entry gives one variable.
  Furthermore, each lower triangular entry of the matrix representing an
element of $\underline{H}(R)$ is completely determined by the corresponding upper triangular entry.
  Thus the dimension of $\underline{H}$ is $2\cdot n(n-1)/2+n=n^2$.
  This is also the same as   $2n^2-\mathrm{dim~} \mathrm{U}(V, h)=n^2$.\\

    Now suppose that $R$ is any (not necessarily flat) $A$-algebra. Recall that $\epsilon$ is a unit in $B$
    such that $\sigma(\pi)=\epsilon\pi$ and    $(\epsilon-1)/\pi $ is a unit in $B$.
    We also use $\epsilon$ to mean $\epsilon\otimes 1$ in $B\otimes_AR$.
We again use $\sigma$ to mean the automorphism of $B\otimes_AR$ given by $b\otimes r \mapsto \sigma(b)\otimes r$.
    By choosing a $B$-basis of  $L$ as explained in Theorem 2.10 and Remark 2.3.a), we describe each element of $\underline{H}(R)$ formally as a  matrix
    $\begin{pmatrix}\pi^{max\{i,j\}}f_{i,j}\end{pmatrix}$ with the following:
    \begin{enumerate}
\item When $i\neq j$, $f_{i,j}$  is an  $(n_i \times n_j)$-matrix with entries in $B\otimes_AR$ and
 $\epsilon^{max\{i,j\}}\sigma({}^tf_{i,j})=f_{j,i}$.
\item   Assume that $i=j$ is even. Then
 \[
\pi^if_{i,i}=\left\{
  \begin{array}{l l}
  (\pi\cdot \sigma(\pi))^{i/2}\begin{pmatrix} a_i&\pi b_i\\ \sigma(\pi\cdot {}^t b_i) &1 +2c_i \end{pmatrix}    & \quad  \textit{if  $L_i$ is \textit{of type $I^o$}};\\
  (\pi\cdot \sigma(\pi))^{i/2}\begin{pmatrix} a_i&b_i&\pi e_i\\ \sigma({}^tb_i) &1+2f_i&1+\pi d_i \\ \sigma(\pi \cdot {}^te_i) &\sigma(1+\pi d_i) &2c_i \end{pmatrix}    & \quad  \textit{if  $L_i$ is \textit{of type $I^e$}};\\
  (\pi\cdot \sigma(\pi))^{i/2}a_i  & \quad  \textit{if  $L_i$ is \textit{of type $II$}}.
      \end{array} \right.
\]
Here, $a_i$  is a formal  $(n_i-1 \times n_i-1)$-matrix (resp. $(n_i-2 \times n_i-2)$-matrix or $(n_i \times n_i)$-matrix)
when $L_i$ is \textit{of type $I^o$} (resp. \textit{of type $I^e$} or \textit{of type $II$}).
Non-diagonal entries of $a_i$ are  in $B\otimes_AR$ and
the $j$-th diagonal entry of $a_i$ is of the form $2x_i^j$ with $x_i^j \in R$. In addition, for non-digonal entries of $a_i$, we have the relation  $\sigma({}^ta_i)=a_i$.
And $b_i, d_i, e_i$ are matrices of suitable sizes with entries in $B\otimes_AR$ and $c_i, f_i$ are elements in $R$.
\item   Assume that $i=j$ is odd.
Then  $$\pi^i f_{i,i}=(\pi\cdot \sigma(\pi))^{(i-1)/2}\pi a_i,$$ where    $a_i$ is a formal   $(n_i \times n_i)$-matrix.
Here,  non-diagonal entries of $a_i$ are in $B\otimes_AR$
and the $j$-th diagonal entry of $a_i$ is of the form $\epsilon\pi x_i^j$ with $x_i^j \in R$.
In addition, for non-digonal entries of $a_i$, we have the relation  $\epsilon\cdot\sigma({}^ta_i)=a_i$.
\end{enumerate}
\textit{   }

To simplify notation, each element
$$((f_{i,j})_{i< j}, (a_i, x_i^j)_{\textit{$L_i$ of type II}}, (a_i, x_i^j, b_i, c_i)_{\textit{$L_i$ of type $I^o$}},
(a_i, x_i^j, b_i, c_i, d_i, e_i, f_i)_{\textit{$L_i$ of type $I^e$}})$$ of  $\underline{H}(R)$
is denoted by $(f_{i,j}, a_i \cdots f_i)$.

\begin{Rmk}
\begin{enumerate}\item   Note that the given hermitian form $h$ is an element of $\underline{H}(A)$.
We represent the given hermitian form $h$ by a hermitian matrix $\begin{pmatrix} \pi^{i}\cdot h_i\end{pmatrix}$
whose $(i,i)$-block is $\pi^i\cdot h_i$ for all $i$, and all of whose remaining blocks are $0$.
Then
\begin{enumerate}
\item If $i$ is even and $L_i$ is \textit{of type} $\textit{I}^o$,
then $\pi^i \cdot h_i$ is of the form
$$(\pi\cdot \sigma(\pi))^{i/2}\begin{pmatrix} \begin{pmatrix} 0&1\\1&0\end{pmatrix}& & & \\ &\ddots & & \\ & &\begin{pmatrix} 0&1\\1&0\end{pmatrix}& \\ & & & 1+2\gamma_i \end{pmatrix}.$$
Here, $\gamma_i\in A$.
\item  If $i$ is even and $L_i$ is \textit{of type} $\textit{I}^e$,
then $\pi^i \cdot h_i$ is of the form
$$(\pi\cdot \sigma(\pi))^{i/2}\begin{pmatrix} \begin{pmatrix} 0&1\\1&0\end{pmatrix}& & & \\ &\ddots & & \\ & &\begin{pmatrix} 0&1\\1&0\end{pmatrix}& \\ & & & \begin{pmatrix} 1&1\\1&2\gamma_i\end{pmatrix} \end{pmatrix}.$$
Here, $\gamma_i\in A$.
\item If $i$ is even and $L_i$ is \textit{of type} $\textit{II}$,
then $\pi^i \cdot h_i$ is of the form
$$(\pi\cdot \sigma(\pi))^{i/2}\begin{pmatrix} \begin{pmatrix} 0&1\\1&0\end{pmatrix}& &  \\ &\ddots &  \\ & &\begin{pmatrix} 0&1\\1&0\end{pmatrix} \end{pmatrix}.$$
\item If $i$ is odd, then $\pi^i \cdot h_i$ is of the form
$$(\pi\cdot \sigma(\pi))^{(i-1)/2}\begin{pmatrix} \begin{pmatrix} 0&\pi\\ \sigma(\pi)&0\end{pmatrix}& & & \\ &\ddots & & \\ & &\begin{pmatrix} 0&\pi\\ \sigma(\pi)&0\end{pmatrix}& \\ & & & \begin{pmatrix} 2&\pi\\ \sigma(\pi)&2\gamma_i\end{pmatrix} \end{pmatrix}.$$
Here, $\gamma_i\in A$.\\
\end{enumerate}

\item Let $R$ be a $\kappa$-algebra.
   We also denote by $h$ the element of $\underline{H}(R)$ which is the image of $h\in \underline{H}(A)$
   under the natural map from $\underline{H}(A)$ to $\underline{H}(R)$.
   Recall that we denote each element of $\underline{H}(R)$ by  $(f_{i, j}, a_i\cdots f_i)$.
 Then the tuple $(f_{i, j}, a_i\cdots f_i)$ denoting $h\in \underline{H}(R)$ is defined by the conditions:
  \begin{enumerate}
\item  If $i\neq j$, then  $f_{i,j}=0$.
\item If $i$ is even, then $$a_i=\begin{pmatrix} \begin{pmatrix} 0&1\\1&0\end{pmatrix}& &  \\ &\ddots &  \\ & & \begin{pmatrix} 0&1\\1&0\end{pmatrix}\end{pmatrix}, \textit{thus $x_i^j=0$},$$
    $$b_i=0, d_i=0, e_i=0, f_i=0, c_i=\bar{\gamma}_i.$$
Here, $\bar{\gamma}_i\in \kappa$ is the reduction of $\gamma_i$ mod $2$.
\item If $i$ is odd, then
$$a_i=\begin{pmatrix} \begin{pmatrix} 0&1\\ \bar{\epsilon}&0\end{pmatrix}& & & \\ &\ddots & & \\ & &\begin{pmatrix} 0&1\\ \bar{\epsilon}&0\end{pmatrix}& \\ & & & \begin{pmatrix} \pi\cdot \bar{\epsilon}\bar{\zeta}&1\\ \bar{\epsilon}&\pi\cdot\bar{\epsilon}\bar{\zeta}\bar{\gamma}_i\end{pmatrix} \end{pmatrix}.$$
Here, $\bar{\epsilon}$ is the reduction of $\epsilon$ mod $2$, not mod $\pi$, so that $\bar{\epsilon}$ is an element of $B\otimes_AR$.
In addition, $\bar{\zeta}\in \kappa$ is the reduction of $\zeta$ mod $2$, where  $\zeta\in A$ is the unit satisfying $2=\pi\cdot \sigma(\pi)\cdot \zeta$.
Thus, $x_i^j=0$ for all $1 \leq j \leq n_i-2$ and $x_i^{n_i-1}=\bar{\zeta}$ and $x_i^{n_i}=\bar{\zeta}\bar{\gamma}_i$.
  \end{enumerate}
\end{enumerate}
\end{Rmk}

\subsection{The smooth affine group scheme \textit{\underline{G}}}

\begin{Thm}
    For any flat $A$-algebra $R$, the group $\underline{M}^{\ast}(R)$ acts on  $\underline{H}(R)$ on the right
    by $f\circ m = \sigma({}^tm)\cdot f\cdot m$.
    This action is represented by an action morphism
     \[\underline{H} \times \underline{M}^{\ast} \longrightarrow \underline{H} .\]
\end{Thm}

\begin{proof}
 We start with any $m\in \underline{M}^{\ast}(R)$ and $f\in \underline{H}(R)$.
 In order to show that  $\underline{M}^{\ast}(R)$ acts on the right of $\underline{H}(R)$ by $f\circ m = \sigma({}^tm)\cdot f\cdot m$,
it suffices to show that $f\circ m$ satisfies  conditions a) to c) given in Section 3.3.\\
Since elements of $\underline{M}(R)$ preserve $L\otimes_AR$ and $A_i\otimes_AR$, $f\circ m$ satisfies  condition a).
That $f\circ m$ satisfies condition b) follows from the fact that $m$  stabilizes $A_i$ and $B_i$ and induces the identity on $A_i/B_i$.\\
For  condition c), it suffices to show that $\frac{1}{\pi^i}f(ma_i, mw_i) \equiv \frac{1}{\pi^i}f(a_i, w_i)$ mod $\pi$.
We denote $ma_i=a_i+b_i$ and $mw_i=w_i+x_i$, where $b_i\in B_i\otimes_A R, x_i \in X_i\otimes_A R$.
Hence it suffices to show $\frac{1}{\pi^i}f(a_i+b_i, x_i) + \frac{1}{\pi^i}f(b_i, w_i) \mathrm{~mod~}\pi \equiv 0$.
 Firstly, $\frac{1}{\pi^i}f(a_i+b_i, x_i) \mathrm{~mod~}\pi \equiv 0$ due to the definition of the lattice $X_i$.
Secondly, if $B_i\varsubsetneq A_i$, then $\frac{1}{\pi^i}f(b_i, w_i) \mathrm{~mod~}\pi \equiv 0$ because
$\frac{1}{\pi^i}f(b_i, w_i)=\frac{1}{\pi^i}h(b_i, w_i)$ $\mathrm{~mod~}\pi$ and
$(\frac{1}{(\pi\cdot\sigma(\pi))^m}h( b_i, e))^2  \equiv \frac{1}{(\pi\cdot\sigma(\pi))^m}h( b_i, b_i) \equiv 0 \mathrm{~mod~}\pi$,
where $e$ is the unique vector chosen earlier.
If $B_i=A_i$, then  $W_i=X_i$ and thus $\frac{1}{\pi^i}f(b_i, w_i) \mathrm{~mod~}\pi \equiv 0$.

We now show that this action of the group $\underline{M}^{\ast}(R)$  on the right of $\underline{H}(R)$ is represented by an action morphism of schemes.
We observe that the action map $\underline{H}(R) \times \underline{M}^{\ast}(R) \longrightarrow \underline{H}(R)$,
$(f, m)\mapsto \sigma({}^tm)\cdot f\cdot m$ is given by polynomials over $A$.
Thus it  induces a ring homomorphism over $A$ from the coordinate ring of $\underline{H}$
to the coordinate ring of $\underline{H} \times \underline{M}^{\ast}$, which accordingly induces a morphism from $\underline{H} \times \underline{M}^{\ast}$ to $\underline{H}$  such that
the action map induced by this morphism at the level of $R$-points, for a flat $A$-algebra $R$, is the same as the action  given in the theorem.
\end{proof}

\begin{Rmk}
Let $R$ be a $\kappa$-algebra.
We explain the above action morphism in terms of $R$-points.
Choose an element  $(m_{i,j}, s_i\cdots w_i)$ in $ \underline{M}^{\ast}(R) $ as explained in Section 3.2
and we express this element formally as a matrix $m=\begin{pmatrix}\pi^{max\{0,j-i\}}m_{i,j}\end{pmatrix}$.
We also choose an element $(f_{i,j}, a_i \cdots f_i)$ of $\underline{H}(R)$ and express this element formally as a matrix
 $f=\begin{pmatrix}\pi^{max\{i,j\}}f_{i,j}\end{pmatrix}$ as explained in Section 3.3.

We then compute the formal matrix product $\sigma({}^tm)\cdot f\cdot m$  and denote it by the formal matrix $\begin{pmatrix}\pi^{max\{i,j\}}\tilde{f}_{i,j}'\end{pmatrix}$ with  $(\tilde{f}_{i,j}', \tilde{a}_i' \cdots \tilde{f}_i')$.
 Here, the description of the formal matrix $\begin{pmatrix}\pi^{max\{i,j\}}\tilde{f}_{i,j}'\end{pmatrix}$ with
 $(\tilde{f}_{i,j}', \tilde{a}_i' \cdots \tilde{f}_i')$ is as explained in Section 3.3.

 We now let $\pi^2$ be zero in each entry of formal matrices
 $(\tilde{f}_{i,j}')_{i< j},  (\tilde{b}_i')_{\textit{$L_i$ of type $I^o$}},  (\tilde{b}_i', \tilde{d}_i', \tilde{e}_i')_{\textit{$L_i$ of type $I^e$}}$
   and  in each non-diagonal entry of a formal matrix  $(\tilde{a}_i')$.
 Then these entries are elements in $B\otimes_AR$.
 We also let $\pi^2$ be zero in $(\tilde{x}_i^j)', (\tilde{c}_i')_{\textit{$L_i$ of type $I^o$}},
 (\tilde{f}_i', \tilde{c}_i')_{\textit{$L_i$ of type $I^e$}}$. Note that $(\tilde{x}_i^j)'$ is a  diagonal entry of a formal matrix $\tilde{a}_i'$.
 Then these entries are elements in $R$.

 Let $(f_{i,j}', a_i' \cdots f_i')$ be the reduction of $(\tilde{f}_{i,j}', \tilde{a}_i' \cdots \tilde{f}_i')$ as explained above,
 i.e. by letting $\pi^2$ be zero in the entries of formal matrices as described above.
 Then $(f_{i,j}', a_i' \cdots f_i')$ is an element of $\underline{H}(R)$
and the composition $(f_{i,j}, a_i \cdots f_i)\circ (m_{i,j}, s_i\cdots w_i)$ is $(f_{i,j}', a_i' \cdots f_i')$.

We can also write $(f_{i,j}', a_i' \cdots f_i')$ explicitly in terms of $(f_{i,j}, a_i \cdots f_i)$ and $(m_{i,j}, s_i\cdots w_i)$
like  the product of $(m_{i,j}, s_i\cdots w_i)$ and $(m_{i,j}', s_i'\cdots w_i')$ explained in Section 3.2.
However, this is complicated and we do not use it in this generality. 
On the other hand, we explicitly calculate $(f_{i,j}, a_i \cdots f_i)\circ (m_{i,j}, s_i\cdots w_i)$ when
$(f_{i,j}, a_i \cdots f_i)$ is the given hermitian form $h$ and $(m_{i,j}, s_i\cdots w_i)$ satisfies certain  conditions on each block.
This explicit calculation  will be done in Appendix A.
\end{Rmk}

 \begin{Thm}
  Let $\rho$ be the morphism $\underline{M}^{\ast} \rightarrow \underline{H}$ defined by $\rho(m)=h \circ m$,
  which is induced by the  action morphism  of Theorem 3.4.
  Then $\rho$ is smooth of relative dimension dim $\mathrm{U}(V, h)$.
  \end{Thm}

  \begin{proof}
  The theorem follows from Lemma 5.5.1 of \cite{GY} and the following lemma.
  \end{proof}

      \begin{Lem}
      The morphism $\rho \otimes \kappa : \underline{M}^{\ast}\otimes \kappa \rightarrow \underline{H}\otimes \kappa$
      is smooth of relative dimension $\mathrm{dim~} \mathrm{U}(V, h)$.
      \end{Lem}

\begin{proof}
 The proof is based on Lemma 5.5.2 in \cite{GY}.
    It is enough to check the statement over the algebraic closure $\bar{\kappa}$ of $\kappa$.
    By \cite{H}, III.10.4, it suffices to show that, for any $m \in \underline{M}^{\ast}(\bar{\kappa})$,
    the induced map on the Zariski tangent space $\rho_{\ast, m}:T_m \rightarrow T_{\rho(m)}$ is surjective.

   We define the two functors from the category of commutative flat $A$-algebras to the category of abelian groups as follows:
 \[T_1(R)=\{m-1 : m\in\underline{M}(R)\},\]
     \[T_2(R)=\{f-h : f\in\underline{H}(R)\}.\]

   The functor $T_1$ (resp. $T_2$) is representable by a flat $A$-algebra which is a polynomial ring over $A$ of $2n^2$ (resp. $n^2$) variables
   by Lemma 3.1 of \cite{C1}.
    Moreover, each of them is represented by a commutative group scheme since they are closed under addition.
        In fact, $T_1$ is the same as the functor $\underline{M}^{\prime}$ in Remark 3.1
        and $T_2$ is the same as the functor $\underline{H}'$ in Section 3.3.

   We still need to introduce another functor on flat $A$-algebras.
   Define $T_3(R)$ to be the set of all $(n \times n)$-matrices $y$ over $B\otimes_AR$ satisfying the following conditions:
   \begin{enumerate}
  \item[a)] The $(i,j)$-block  of $y$ has entries in $\pi^{max\{i,j\}}B\otimes_AR$ so that
   $$y=\begin{pmatrix} \pi^{max\{i,j\}}y_{i,j}\end{pmatrix}.$$
      Here, the size of $y_{i,j}$ is $n_i\times n_j$.
\item[b)] If $i$ is even and $L_i$ is \textit{of type} $\textit{I}^o$, $y_{i,i}$ is of the form
   \[\begin{pmatrix} s_i&\pi y_i\\ \pi v_i&\pi z_i \end{pmatrix}\in \mathrm{M}_{n_i}(B\otimes_AR)\]
     where $s_i$ is an $(n_i-1) \times (n_i-1)$ matrix, etc.
\item[c)] If $i$ is even and $L_i$ is \textit{of type} $\textit{I}^e$, $y_{i,i}$ is of the form
   \[\begin{pmatrix} s_i&r_i&\pi t_i\\ y_i&x_i&\pi z_i\\ \pi v_i&\pi u_i&\pi w_i \end{pmatrix}\in \mathrm{M}_{n_i}(B\otimes_AR)\]
     where $s_i$ is an $(n_i-2) \times (n_i-2)$-matrix, etc.
     \end{enumerate}
     The functor $T_3$ is represented by a flat $A$-scheme which is isomorphic to an affine space by Lemma 3.1 of \cite{C1}.
    Moreover   it is represented by a commutative group scheme since it is closed under addition.
So far, we have defined three functors $T_1, T_2, T_3$ and these are represented by schemes. Therefore, we can talk about their $\bar{\kappa}$-points.

We now compute the map $\rho_{\ast, m}$ explicitly.
We first describe an element of the tangent space $T_m$.
Since $\underline{M}^{\ast}$ is an open subscheme of $\underline{M}$,
the tangent space $T_m$ may and shall be identified with the set of elements of $\underline{M}(\bar{\kappa}[\epsilon]/(\epsilon^2))$ whose reduction to
$\underline{M}(\bar{\kappa})$ induced by the obvious map $\bar{\kappa}[\epsilon]/(\epsilon^2)\rightarrow \bar{\kappa}$ is $m$,
by considering $m$ as an element of $\underline{M}(\bar{\kappa})$.
Recall from Remark 3.1 that we  defined the functor $\underline{M}^{\prime}$ such that $(1+\underline{M}^{\prime})(R)=\underline{M}(R)$
inside $\mathrm{End}_{B\otimes_AR}(L \otimes_A R)$ for a flat $A$-algebra $R$.
Thus there is an isomorphism of schemes (as set valued functors)
$$1+ :  \underline{M}^{\prime} \longrightarrow \underline{M}.$$
Let $m'$ be an element of $\underline{M}'(\bar{\kappa})$ which maps to $m$ under the morphism $1+$ at the level of $\bar{\kappa}$-points.
Then each element of the tangent space of $\underline{M}'$ at $m'$ is of the form $m'+\epsilon X\in \underline{M}'(\bar{\kappa}[\epsilon]/(\epsilon^2))$ for $X\in \underline{M}'(\bar{\kappa})$.
We denote by $m+\epsilon X$  the image of $m'+\epsilon X$ under the morphism $1+$ at the level of $\bar{\kappa}[\epsilon]/(\epsilon^2)$-points.
Thus we can express an element of $T_m$  formally as $m+\epsilon X$ where $X\in \underline{M}'(\bar{\kappa})$.
Similarly, an element of $T_{\rho(m)}$ can be expressed formally as $\rho(m)+\epsilon Y$ where $Y\in \underline{H}'(\bar{\kappa})$,
by using  an isomorphism of schemes (as set valued functors)
\[h+ : \underline{H}^{\prime} \longrightarrow \underline{H}.\]
Here, $\underline{H}'$ is defined in Section 3.3.

Before observing the image of $m+\epsilon X$ under the morphism $\rho$ at the level of $\bar{\kappa}[\epsilon]/(\epsilon^2)$-points,
 we lift $m+\epsilon X$ to an element of $\underline{M}(R[\epsilon]/(\epsilon^2))$ as follows, where $R$ is a local $A$-algebra whose residue field is $\bar{\kappa}$.
Let $\tilde{m}'\in \underline{M}'(R)$ (resp. $\tilde{X}\in \underline{M}'(R)$) be a  lift of $m'$ (resp. $X$)
so that $\tilde{m}'+\epsilon \tilde{X} \in \underline{M}'(R[\epsilon]/(\epsilon^2))$ is a lift of $m'+\epsilon X\in \underline{M}'(\bar{\kappa}[\epsilon]/(\epsilon^2))$.
Let $\tilde{m}\in \underline{M}(R)$ be the image of $\tilde{m}'$ under the morphism $1+$.
Then $\tilde{m}+\epsilon \tilde{X}$ is an element of $\underline{M}(R[\epsilon]/(\epsilon^2))$ whose reduction
to $\underline{M}(\bar{\kappa}[\epsilon]/(\epsilon^2))$ induced by the  map $R[\epsilon]/(\epsilon^2) \rightarrow \bar{\kappa}[\epsilon]/(\epsilon^2)$
is $m+\epsilon X$.
Here, the addition in $\tilde{m}+\epsilon \tilde{X}$ is the addition inside $\mathrm{End}_{B\otimes_AR[\epsilon]/(\epsilon^2)}(L \otimes_A R[\epsilon]/(\epsilon^2))$
since $R[\epsilon]/(\epsilon^2)$ is flat over $A$ (cf. Remark 3.1).
This is illustrated in the following commutative diagram:
\[\xymatrixcolsep{5pc}\xymatrix{
\underline{M}'(R[\epsilon]/(\epsilon^2)) \ar[d] \ar[r]^{1+} &\underline{M}(R[\epsilon]/(\epsilon^2))\ar[d]\\
\underline{M}'(\bar{\kappa}[\epsilon]/(\epsilon^2)) \ar[r]^{1+} &\underline{M}(\bar{\kappa}[\epsilon]/(\epsilon^2)),}
\xymatrix{
\tilde{m}'+\epsilon \tilde{X} \ar@{|->}[d] \ar@{|->}[r] &\tilde{m}+\epsilon \tilde{X}\ar@{|->}[d]\\
m'+\epsilon X \ar@{|->}[r] &m+\epsilon X.}
\]

Note that the proof of Theorem 3.4 also tells the existence of the  morphism $\underline{H}\times \underline{M}\rightarrow \underline{H}$,
$(f, m)\mapsto f\circ m = \sigma({}^tm)\cdot f\cdot m$,
where $f\in \underline{H}(R)$ and $m\in \underline{M}(R)$ for a flat $A$-algebra $R$.
This morphism  induces the morphism
$\underline{M}\rightarrow \underline{H}$ with $m\mapsto h\circ m$ whose reduction to $\underline{M}^{\ast}$ is the same as $\rho$.
Thus the above morphism  $\underline{M}\rightarrow \underline{H}$ can also be denoted by $\rho$.
We can now talk about the image of $\tilde{m}+\epsilon \tilde{X}$ under the morphism $\rho$ at the level of $R[\epsilon]/(\epsilon^2)$-points.
Since $R[\epsilon]/(\epsilon^2)$ is a flat $A$-algebra, the image of $\tilde{m}+\epsilon \tilde{X}$ comes from a usual matrix product
\begin{equation}
\sigma(\tilde{m}+\epsilon \tilde{X})^t\cdot h \cdot (\tilde{m}+\epsilon \tilde{X})=\sigma(\tilde{m})^t\cdot h \cdot \tilde{m}
+\epsilon(\sigma(\tilde{m})^t\cdot h \cdot \tilde{X}+\sigma(\tilde{X})^t\cdot h \cdot \tilde{m}).
\end{equation}
Thus the image of $m+\epsilon X$ under the morphism $\rho$ at the level of $\bar{\kappa}[\epsilon]/(\epsilon^2)$-points is the reduction of
$\sigma(\tilde{m})^t\cdot h \cdot \tilde{m} 
+\epsilon(\sigma(\tilde{m})^t\cdot h \cdot \tilde{X}+\sigma(\tilde{X})^t\cdot h \cdot \tilde{m})$
to $\underline{H}(\bar{\kappa}[\epsilon]/(\epsilon^2))$.
It is obvious that $\rho(m) ~(\in \underline{H}(\bar{\kappa}))$ is the reduction of $\sigma(\tilde{m})^t\cdot h \cdot \tilde{m} ~(\in \underline{H}(R))$ since $\tilde{m}$ is a lift of $m$ and $\rho$ is a morphism of schemes.
To observe the reduction of
$\sigma(\tilde{m})^t\cdot h \cdot \tilde{X}+\sigma(\tilde{X})^t\cdot h \cdot \tilde{m} ~(\in \underline{H}'(R))$
to $\underline{H}'(\bar{\kappa})$,
we consider a morphism  $\underline{M}\times \underline{H}'\rightarrow \underline{H}'$ such that
$(\tilde{m}, \tilde{X})$ maps to $\sigma(\tilde{m})^t\cdot h \cdot \tilde{X}+\sigma(\tilde{X})^t\cdot h \cdot \tilde{m}$,
where  $(\tilde{m}, \tilde{X})\in \underline{M}(R)\times \underline{H}'(R)$ for a flat $A$-algebra $R$.
To show that this map is well-defined, we need to show that
$\sigma(\tilde{m})^t\cdot h \cdot \tilde{X}+\sigma(\tilde{X})^t\cdot h \cdot \tilde{m}$ is an element of $\underline{H}'(R)$.
This can be easily shown by considering the morphism of tangent spaces induced from $\rho$ at $\tilde{m}\in \underline{M}(R)$
(cf. Equation (3.1)).
Since this morphism is representable, we can denote by
$\sigma(m)^t\cdot h \cdot X+\sigma(X)^t\cdot h \cdot m ~(\in \underline{H}'(\bar{\kappa}))$  the reduction of
$\sigma(\tilde{m})^t\cdot h \cdot \tilde{X}+\sigma(\tilde{X})^t\cdot h \cdot \tilde{m} ~(\in \underline{H}'(R))$
to $\underline{H}'(\bar{\kappa})$.
Then the image of $m+\epsilon X$  is a formal sum $\rho(m)+\epsilon(\sigma(m)^t\cdot h \cdot X+\sigma(X)^t\cdot h \cdot m)
~(\in \underline{H}(\bar{\kappa}[\epsilon]/(\epsilon^2)))$.\\

Thus if we identify $T_m$ with $T_1(\bar{\kappa})$ and $T_{\rho(m)}$ with $T_2(\bar{\kappa})$, then
 $$\rho_{\ast, m} : T_m \longrightarrow T_{\rho(m)}, X\mapsto \sigma(m)^t\cdot h \cdot X+\sigma(X)^t\cdot h \cdot m.$$

We explain how to compute  $X\mapsto  \sigma(m)^t\cdot h \cdot X+\sigma(X)^t\cdot h \cdot m$ explicitly.
Recall that for a $\kappa$-algebra $R$, we denote an element $m$  of $\underline{M}(R)$ by $(m_{i,j}, s_i\cdots w_i)$ with a formal matrix interpretation
$m=(\pi^{max\{0, j-i\}}m_{i,j})$ (cf.  Section 3.2)
and  we denote an element $f$ of $\underline{H}(R)$  by $(f_{i,j}, a_i\cdots f_i)$   with a formal matrix interpretation
$f=(\pi^{max\{i,j\}}f_{i,j})$ (cf. Section 3.3).
Similarly, we can also denote an element $X$ of $T_1(\bar{\kappa})$ by $(m_{i,j}', s_i'\cdots w_i')$ with a formal matrix interpretation $X=(\pi^{max\{0, j-i\}}m_{i,j}')$ and an element $Z$ of $T_2(\bar{\kappa})$ by $(f_{i,j}', a_i'\cdots f_i')$ with a formal matrix interpretation  $Z=(\pi^{max\{i,j\}}f_{i,j}')$.
Then we formally compute $X \mapsto \sigma(m^t)\cdot h\cdot X + \sigma(X^t)\cdot h\cdot m$
and consider the reduction of the formal matrix $ \sigma(m^t)\cdot h\cdot X + \sigma(X^t)\cdot h\cdot m$
 in a manner similar to that of the reduction   explained in Remark 3.5.
We denote  this reduction by  $(f_{i,j}'', a_i''\cdots f_i'')$ with a formal matrix interpretation $(\pi^{max\{i,j\}}f_{i,j}'')$.
This $(f_{i,j}'', a_i''\cdots f_i'')$ may and shall be identifed with an element of $T_2(\bar{\kappa})$ in the manner just described.
Then $\rho_{\ast, m}(X)$ is the element $Z=(f_{i,j}'', a_i''\cdots f_i'')$ of $T_2(\bar{\kappa})$.\\

   To prove the surjectivity of $\rho_{\ast, m}:T_1(\bar{\kappa}) \rightarrow T_2(\bar{\kappa})$, it suffices to show the following three statements:
       \begin{itemize}
   \item[(1)] $X \mapsto h\cdot X $ defines a bijection $T_1(\bar{\kappa}) \rightarrow T_3(\bar{\kappa})$;
   \item[(2)] for any $m \in \underline{M}^{\ast}(\bar{\kappa})$, $Y \mapsto \sigma({}^t m) \cdot Y$ defines a bijection from $T_3(\bar{\kappa})$ to itself;
   \item[(3)] $Y \mapsto \sigma({}^t Y) + Y$ defines a surjection $T_3(\bar{\kappa}) \rightarrow T_2(\bar{\kappa})$.
\end{itemize}
Here, all the above maps are  interpreted as in Remark 3.5 (if they are well-defined).
Then $\rho_{\ast, m}$ is the composite of these three.
   (3) is direct from the construction of $T_3(\bar{\kappa})$. Hence we provide the proof of (1) and (2).\\

  For (1), suppose that the two functors $T_1(R)\longrightarrow T_3(R), X\mapsto h\cdot X (\in \mathrm{M}_{n\times n}(B\otimes_AR))$ and
  $T_3(R) \longrightarrow T_1(R), Y \mapsto h^{-1}\cdot Y (\in \mathrm{M}_{n\times n}(B\otimes_AR))$
   are  well-defined for all flat $A$-algebras $R$.
  In other words, suppose that $h\cdot X \in T_3(R)$ and $h^{-1}\cdot Y\in T_1(R)$.
  These functors are then represented by  morphisms of schemes by an argument similar to
that used in the proof of Theorem 3.4, so we skip it.
Thus they give maps at the level of $\kappa$-algebra points.
Furthermore, the composition of these two maps at the level of $\kappa$-algebra points is the identity.
To show this, it suffices to prove that the composition of  two morphisms given by the actions of $h$ and $h^{-1}$ is uniquely determined
at the level of $R$-points, for a flat $A$-algebra $R$.
This is proved in Remark 3.2.


We now show  that these two functors are well-defined for a flat $A$-algebra $R$.
We represent $h$ by a hermitian block matrix $\begin{pmatrix} \pi^{i}\cdot h_i\end{pmatrix}$
   with a matrix $(\pi^{i}\cdot h_i)$ for the $(i,i)$-block and $0$ for the remaining blocks as in Remark 3.3.(1).

   For the first functor, it suffices to show that  $h\cdot X$ satisfies the three conditions defining the functor $T_3$.
    Here, $X\in T_1(R)$ for a flat $A$-algebra $R$.
  We express $$X=\begin{pmatrix} \pi^{max\{0,j-i\}}x_{i,j} \end{pmatrix}.$$
  Then
   $$h\cdot X = \begin{pmatrix} \pi^{max(i,j)}y_{i,j}\end{pmatrix}.$$
  Here, $y_{i,i}=h_i\cdot x_{i,i}$.
  Therefore, it suffices to show that $y_{i,i}=h_i\cdot x_{i,i}$ satisfies conditions b) and c) in the description of $T_3(R)$ when $L_i$ is \textit{of type I}.

  If $L_i$ is \textit{of type $I^o$}, then we express $x_{i,i}$ as a matrix $\begin{pmatrix} s_i&\pi y_i\\ \pi v_i&\pi z_i \end{pmatrix}$.
 The matrix form of $h_i$ is $\epsilon^{i/2}\begin{pmatrix} \begin{pmatrix} 0&1\\1&0\end{pmatrix}& & & \\ &\ddots & & \\ & &\begin{pmatrix} 0&1\\1&0\end{pmatrix}& \\ & & & 1+2\gamma_i \end{pmatrix}$ as in Remark 3.3.(1).  Here, $\epsilon$ is a unit in $B$ such that $\sigma(\pi)=\epsilon\pi$, as explained in Section 2.1.
 To simplify our notation, write $h_i=\epsilon^{i/2}\begin{pmatrix} I_i&0\\ 0&1+2\gamma_i \end{pmatrix}$.
  Then we can see that
  $$h_i\cdot x_{i,i}=\epsilon^{i/2}\begin{pmatrix} I_i&0\\ 0&1+2\gamma_i \end{pmatrix}\cdot\begin{pmatrix} s_i&\pi y_i\\ \pi v_i&\pi z_i \end{pmatrix}
  =\epsilon^{i/2}\begin{pmatrix} I_is_i&\pi I_iy_i\\ \pi (1+2\gamma_i)v_i&\pi (1+2\gamma_i)z_i \end{pmatrix}.$$
  Thus, $h_i\cdot x_{i,i}$ satisfies the congruence condition given in b) of the description of $T_3(R)$.

  If $L_i$ is \textit{of type $I^e$}, then we express $x_{i,i}$ as a matrix $\begin{pmatrix} s_i&r_i&\pi t_i\\\pi y_i&\pi x_i&\pi z_i\\  v_i& u_i&\pi w_i \end{pmatrix}$.
 The matrix form of $h_i$ is given as in Remark 3.3.(1) and again,
 in order to simplify our notation, write $h_i=\epsilon^{i/2}\begin{pmatrix} I_i&0&0\\ 0&1&1\\0&1&2\gamma_i \end{pmatrix}$.
  Then we can see that
  $$h_i\cdot x_{i,i}=\epsilon^{i/2}\begin{pmatrix} I_i&0&0\\ 0&1&1\\0&1&2\gamma_i \end{pmatrix}\cdot\begin{pmatrix} s_i&r_i&\pi t_i\\\pi y_i&\pi x_i&\pi z_i\\  v_i& u_i&\pi w_i \end{pmatrix}
  =\epsilon^{i/2}\begin{pmatrix} I_is_i&I_ir_i&\pi I_it_i\\\pi y_i+v_i&\pi x_i+u_i&\pi (z_i+w_i)\\  \pi y_i+2\gamma_iv_i& \pi x_i+2\gamma_iu_i&\pi z_i+2\gamma_i\pi w_i \end{pmatrix}.$$
  Thus, $h_i\cdot x_{i,i}$ satisfies the congruence condition given in c) of the description of $T_3(R)$
  and our functor is well-defined.

   For the second functor, we express $Y=\begin{pmatrix} \pi^{max(i,j)}y_{i,j}\end{pmatrix}$ and
   $h^{-1}=\begin{pmatrix} \pi^{-i}\cdot h_i^{-1}\end{pmatrix}$.
   Then we have the following:
   $$h^{-1}\cdot Y = \begin{pmatrix} \pi^{max\{0,j-i\}}x_{i,j} \end{pmatrix}.$$
  Here, $x_{i,i}=h_i^{-1}\cdot y_{i,i}$.

Then it suffices to show that $h^{-1}\cdot Y = \begin{pmatrix} \pi^{max\{0,j-i\}}x_{i,j} \end{pmatrix}$ satisfies the conditions
defining $T_1(R)$ for a flat $A$-algebra $R$.
Indeed, we do not describe the conditions defining $T_1(R)$ explicitly in this paper.
However, these conditions can be read off from the conditions defining $\underline{M}(R)$ because of the definition of the functor $T_1$.
The matrix form of an element of $\underline{M}(R)$ is described in Section 3.2 and based on this,
it suffices to observe the diagonal blocks $x_{i,i}=h_i^{-1}\cdot y_{i,i}$ when $L_i$ is \textit{of type I}.

  If $L_i$ is \textit{of type $I^o$}, then we express $y_{i,i}$ as a matrix $\begin{pmatrix} s_i&\pi y_i\\ \pi v_i&\pi z_i \end{pmatrix}$.
 The matrix form of $h_i^{-1}$ is  $h_i^{-1}=\epsilon^{-i/2}\begin{pmatrix} I_i&0\\ 0&1+2\gamma_i' \end{pmatrix}$ for a certain $\gamma_i'\in A$.
  Then we can see that
  $$h_i^{-1}\cdot y_{i,i}=\epsilon^{-i/2}\begin{pmatrix} I_i&0\\ 0&1+2\gamma_i' \end{pmatrix}\cdot\begin{pmatrix} s_i&\pi y_i\\ \pi v_i&\pi z_i \end{pmatrix}
  =\epsilon^{i/2}\begin{pmatrix} I_is_i&\pi I_iy_i\\ \pi (1+2\gamma_i')v_i&\pi (1+2\gamma_i')z_i \end{pmatrix}.$$
  Thus, $h_i^{-1}\cdot y_{i,i}$ satisfies the relevant congruence condition in the definition of $T_1(R)$.

  If $L_i$ is \textit{of type $I^e$}, then we express $y_{i,i}$ as a matrix $\begin{pmatrix} s_i&r_i&\pi t_i\\y_i&x_i&\pi z_i\\ \pi  v_i&\pi  u_i&\pi w_i \end{pmatrix}$.
 The matrix form of $h_i^{-1}$ is $h_i^{-1}=\epsilon^{-i/2}\begin{pmatrix} I_i&0&0\\ 0&2\epsilon'\gamma_i&-\epsilon'\\0&-\epsilon'&\epsilon' \end{pmatrix}$.
 Here, $\epsilon'=(2\gamma_i-1)^{-1}$ is a unit in $A$.
  Then we can see that $h_i^{-1}\cdot y_{i,i}$ is
  $$\epsilon^{-i/2}\begin{pmatrix} I_i&0&0\\ 0&2\epsilon'\gamma_i&-\epsilon'\\0&-\epsilon'&\epsilon' \end{pmatrix}\cdot
  \begin{pmatrix} s_i&r_i&\pi t_i\\y_i&x_i&\pi z_i\\ \pi  v_i&\pi  u_i&\pi w_i \end{pmatrix}
  =\epsilon^{-i/2}\begin{pmatrix} I_is_i&I_ir_i&\pi I_it_i\\ \epsilon'(2\gamma_iy_i-\pi v_i)&\epsilon'(2\gamma_ix_i-\pi u_i)&\pi \epsilon'(2\gamma_iz_i-w_i)\\
  \epsilon'(-y_i+\pi v_i)&\epsilon'(-x_i+\pi u_i)&\pi \epsilon'(-z_i+w_i) \end{pmatrix}.$$
  Thus, $h_i^{-1}\cdot y_{i,i}$ satisfies the relevant congruence condition in the definition of $T_1(R)$
  and our functor is well-defined.\\

For (2), suppose that the functor $$\underline{M}^{\ast}(R)\times T_3(R)\longrightarrow T_3(R), ~~~~~ (m, Y)\mapsto \sigma({}^tm)\cdot Y, $$
 for a flat $A$-algebra $R$, is well-defined. In other words, we suppose that $\sigma({}^tm)\cdot Y\in T_3(R)$.
  This functor is then represented by  a morphism of schemes,
   a fact whose proof is similar to the argument used in the proof of Theorem 3.4, so we skip it.
Thus it gives the map at the level of $\bar{\kappa}$-points
  $$\underline{M}^{\ast}(\bar{\kappa})\times T_3(\bar{\kappa})\longrightarrow T_3(\bar{\kappa}), ~~~~~ (m, Y)\mapsto \sigma({}^tm)\cdot Y. $$
  This map implies that our map in (2) is well-defined.
On the other hand, the inverse  of our map in (2) is $Y \mapsto \sigma({}^t m)^{-1} \cdot Y$ and
this map is well-defined as well since $m^{-1}$ is also an element of $\underline{M}^{\ast}(\bar{\kappa})$.
Therefore, the map in (2) is a bijection.

We now show that the above functor is well-defined.
For a flat $A$-algebra, we choose an element $m\in \underline{M}^{\ast}(R)$ and $Y\in T_3(R)$
and    we again express $m=\begin{pmatrix} \pi^{max\{0,j-i\}}m_{i,j} \end{pmatrix}$ and $Y=\begin{pmatrix} \pi^{max(i,j)}y_{i,j}\end{pmatrix}$.
Then $\sigma({}^t m) \cdot Y $ obviously satisfies the condition a) in the definition of $T_3(R)$ and
it suffices to show that $\sigma({}^t m_{i,i}) \cdot y_{i,i}$ satisfies  conditions b) and c) when $L_i$ is \textit{of type I}.

  If $L_i$ is \textit{of type $I^o$}, then we express $m_{i,i}$ as a matrix $\begin{pmatrix} s_i&\pi y_i\\ \pi v_i&1+\pi z_i \end{pmatrix}$
  and    $y_{i,i}$ as a matrix $\begin{pmatrix} a_i&\pi b_i\\ \pi c_i&\pi d_i \end{pmatrix}$.
Then
\[\sigma({}^t m_{i,i}) \cdot y_{i,i}=\begin{pmatrix} \sigma({}^ts_i)&\sigma(\pi\cdot{}^t v_i)\\ \sigma(\pi\cdot{}^t y_i)&1+\sigma(\pi z_i) \end{pmatrix}\cdot
\begin{pmatrix} a_i&\pi b_i\\ \pi c_i&\pi d_i \end{pmatrix}.\]
Then we can easily see that this matrix satisfies the congruence condition b) in the definition of $T_3(R)$.

  If $L_i$ is \textit{of type $I^e$}, then we express $m_{i,i}$ as a matrix
  $\begin{pmatrix} s_i&r_i&\pi t_i\\ \pi y_i&1+\pi x_i&\pi z_i\\ v_i&u_i&1+\pi w_i \end{pmatrix} $
  and    $y_{i,i}$ as a matrix $\begin{pmatrix} a_i&b_i&\pi c_i\\ d_i&e_i&\pi f_i\\ \pi g_i&\pi h_i&\pi k_i \end{pmatrix}$.
Then
\[\sigma({}^t m_{i,i}) \cdot y_{i,i}=\begin{pmatrix} \sigma({}^ts_i)&\sigma(\pi\cdot{}^ty_i)&\sigma({}^tv_i)\\
\sigma({}^tr_i)&1+\sigma(\pi x_i)&\sigma(u_i)\\ \sigma(\pi\cdot {}^tt_i)&\sigma(\pi z_i)&1+\sigma(\pi w_i) \end{pmatrix}\cdot
\begin{pmatrix} a_i&b_i&\pi c_i\\ d_i&e_i&\pi f_i\\ \pi g_i&\pi h_i&\pi k_i \end{pmatrix}.\]
Then we can easily see that this matrix satisfies the congruence condition c) in the definition of $T_3(R)$.



\end{proof}

 Let $\underline{G}$ be the stabilizer of $h$ in $\underline{M}^{\ast}$. It is an affine group subscheme of $\underline{M}^{\ast}$, defined over $A$.
 Thus we have the following theorem.
 \begin{Thm}
 The group scheme $\underline{G}$ is smooth, and $\underline{G}(R)=\mathrm{Aut}_{B\otimes_AR}(L\otimes_A R,h\otimes_A R)$ for any \'{e}tale $A$-algebra $R$.
 \end{Thm}
\begin{proof}
Since $\underline{G}$ is the fiber of $h$ along the smooth morphism
$\rho : \underline{M}^{\ast} \rightarrow \underline{H}$,  $\rho(m)=h \circ m$, $\underline{G}$ is smooth.
Here, we use the fact that smoothness is stable under base change.

For the identity, we recall that each element of $\mathrm{Aut}_{B\otimes_AR}(L\otimes_A R,h\otimes_A R)$, for an \'{e}tale $A$-algebra $R$,
 satisfies all congruence conditions defining $\underline{M}$,
which is explained in Section 3.1.
Since $\underline{G}(R)$ is the group of $R$-points of $\underline{M}^{\ast}$ stabilizing the given hermitian form $h$,
we have the identity $\underline{G}(R)=\mathrm{Aut}_{B\otimes_AR}(L\otimes_A R,h\otimes_A R)$ for any \'{e}tale $A$-algebra $R$.
\end{proof}

Note that in the theorem, the equality holds only for  an \'{e}tale $A$-algebra $R$
since we obtain conditions defining $\underline{M}$ by observing properties of elements of
$\mathrm{Aut}_{B\otimes_AR}(L\otimes_A R,h\otimes_A R)$ for an \'{e}tale $A$-algebra $R$ (cf. Section 3.1).
For example, let $(L, h)$ be the hermitian lattice of rank 1 as given in Appendix B.
For simplicity, let $\pi+\sigma(\pi)=\pi^2=2$.
As a set, $\mathrm{Aut}_{B\otimes_AR}(L\otimes_A R,h\otimes_A R)$ is the same as
$\{(a, b):a,b\in R \textit{ and } a^2+2ab+2b^2=1\}$ for a flat $A$-algebra $R$.
Thus we cannot guarantee that $a-1$ is contained in the ideal $(2)$, which should be necessary
in order that $(a, b)$ is an element of $\underline{G}(R)$.

\section{The special fiber of the smooth integral model}

In this section, we will determine the structure of the special fiber $\tilde{G}$ of $\underline{G}$
by determining the maximal reductive quotient and the component group when $E/F$ satisfies \textit{Case 1},
 by adapting the approach of Section 4 of \cite{C1}.
From this section to the end, the identity matrix is denoted by id.

    \subsection{The reductive quotient of the special fiber}
Recall that $Y_i$ is the sublattice of $B_i$ such that $Y_i/\pi A_i$ is the radical of the alternating bilinear form
$\frac{1}{(\pi\cdot\sigma(\pi))^{\frac{i}{2}}}h$ mod $\pi$ on $B_i/\pi A_i$ (when $i$ is even) and
that $Z_i$ is the sublattice of $A_i$ such that $Z_i/\pi A_i$ is the radical of the quadratic form $\frac{1}{2^m}q$ mod 2 on $A_i/\pi A_i$,
where $\frac{1}{2^m}q(x)=\frac{1}{2^m}h(x,x)$ (when $i=2m-1$ is odd).
 \begin{Lem}
 Let $i$ be odd. Consider the lattice $\pi A_{i-1}+A_{i+1}=\{x+y:x\in \pi A_{i-1}, y\in A_{i+1}\}$.
 Then $\pi A_{i-1}+A_{i+1}=X_i$.

 \end{Lem}

\begin{proof}
Let $L=\bigoplus_i L_i$ be a Jordan splitting. We describe $\pi A_{i-1}, A_{i+1}, X_i$ below:
\[\pi A_{i-1}=\pi^iL_0 \oplus \pi^{i-1}L_1 \oplus \cdots  \oplus \pi L_{i-1} \oplus \pi L_i  \oplus \pi L_{i+1} \oplus \cdots ;\]
\[A_{i+1}=    \pi^{i+1}L_0 \oplus \pi^{i}L_1 \oplus \cdots  \oplus \pi^2 L_{i-1} \oplus \pi L_{i}  \oplus  L_{i+1} \oplus \cdots ;\]
\[X_i=\pi^iL_0 \oplus \pi^{i-1}L_1 \oplus \cdots  \oplus \pi L_{i-1} \oplus \pi L_i  \oplus L_{i+1} \oplus \cdots.\]
Our claim follows directly from the above descriptions.
\end{proof}

\begin{Lem}
Each element of $\underline{M}(R)$, for a flat $A$-algebra $R$, preserves $Y_i\otimes_A R$ (for $i$ even) and $Z_i\otimes_A R$ (for $i$ odd).
\end{Lem}
\begin{proof}
The claim for $Y_i$ follows from the fact that $Y_i=X_i$ or $Y_i=W_i$ according to the type of $L_i$ as described in Remark 2.11.

To prove the claim for $Z_i$,  use Lemma 4.1 to express a given arbitrary element of $Z_i\otimes_AR$ as $x + y$, where
$x\in \pi A_{i-1}\otimes_AR$ and  $y\in A_{i+1}\otimes_AR$.
Let $g\in \underline{M}(R)$. Then $g(x+y)=g(x)+g(y)=(x+x')+(y+y')$, where $x'\in \pi B_{i-1}\otimes_AR, y'\in B_{i+1}\otimes_AR$ since $g$ induces the identity on
$A_{i-1}\otimes_AR/B_{i-1}\otimes_AR, A_{i+1}\otimes_AR/B_{i+1}\otimes_AR$.
Since $\pi A_{i-1}\otimes_AR, A_{i+1}\otimes_AR \subset W_i\otimes_AR$ and so
$\pi B_{i-1}\otimes_AR, B_{i+1}\otimes_AR \subset Z_i\otimes_AR$, we have that $g(x+y)=(x+y)+x'+y'\in Z_i\otimes_AR$.
\end{proof}

\begin{Thm}
 Assume that $i$ is even.
Let $h_i$ denote the nonsingular alternating bilinear form $\frac{1}{(\pi\cdot\sigma(\pi))^{\frac{i}{2}}}h$ mod $\pi$ on $B_i/Y_i$.
Then there exists a unique morphism of algebraic groups
$$\varphi_i:\tilde{G}\longrightarrow \mathrm{Sp}(B_i/Y_i, h_i)$$
defined over $\kappa$ such that for all \'{e}tale local $A$-algebras $R$ with residue field $\kappa_R$ and
every $\tilde{m} \in \underline{G}(R)$ with reduction $m\in \tilde{G}(\kappa_R)$,
$\varphi_i(m)\in \mathrm{GL}(B_i\otimes_AR/Y_i\otimes_AR)$ is induced by the action of $\tilde{m}$ on $L\otimes_AR$
(which preserves $B_i\otimes_AR$ and $Y_i\otimes_AR$ by Lemma 4.2).
Note that the dimension of $B_i/Y_i$, as a $\kappa$-vector space, is as follows:
\[
\left\{
  \begin{array}{l l}
  n_i & \quad  \textit{if $L_i$ is \textit{of type II}};\\
  n_i-1 & \quad \textit{if $L_i$ is \textit{of type} $\textit{I}^o$};\\
  n_i-2 & \quad \textit{if $L_i$ is \textit{of type} $\textit{I}^e$}.
    \end{array} \right.
\]
\end{Thm}

\begin{proof}
Let $R$ be an \'{e}tale local $A$-algebra with $\kappa_R$ as its residue field.
Note that such an $R$ is finite over $A$ since
  any \'etale local algebra $R$ over a henselian local ring is finite   by Proposition 4 of Section 2.3 in \cite{BLR}
  and since $A$ is henselian.
For such a finite field extension $\kappa_R$ of $\kappa$, $R$ is uniquely determined up to isomorphism.
Since $\underline{G}$ is smooth over $A$, the map $\underline{G}(R)\rightarrow \tilde{G}(\kappa_R)$ is surjective by \textit{Hensel's lemma}.

 Now, we choose an element $m\in \tilde{G}(\kappa_R)$ and a lift $\tilde{m} \in \underline{G}(R)$.
Since the action of $\tilde{m}$ on $L\otimes_AR$ preserves  $B_i\otimes_AR$ and $Y_i\otimes_AR$,
$\tilde{m}$ determines an element of $\mathrm{GL}(B_i\otimes_AR/Y_i\otimes_AR)$.
Furthermore, it is easy to show that this element determined by $\tilde{m}$ fixes $h_i\otimes \kappa_R$
on $B_i/Y_i\otimes_{\kappa}\kappa_R ~(=B_i\otimes_AR/Y_i\otimes_AR)$.
Thus $\tilde{m}$ determines an element of $\mathrm{Sp}(B_i/Y_i, h_i)(\kappa_R)$ and so
 we have a map  from $\tilde{G}(\kappa_R)$ to $\mathrm{Sp}(B_i/Y_i, h_i)(\kappa_R)$.
Indeed,  this map is well-defined, i.e. independent of a lift  $\tilde{m}$ of $m$
as will be explained later after describing a matrix interpretation  of this map.
 In order to show that this map is well-defined and representable, we interpret it in terms of matrices.
 Recall that $\underline{G}$ is a closed subgroup scheme of $\underline{M}^{\ast}$ and $\tilde{G}$ is a closed subgroup scheme of $\tilde{M}$,
 where $\tilde{M}$ is the special fiber of $\underline{M}^{\ast}$.
 Thus we may consider an element of  $\tilde{G}(\kappa_R)$ as an element of $\tilde{M}(\kappa_R)$.
 Based on Section 3.2, an element $m$ of $\tilde{G}(\kappa_R)$ may be written as, say,
  $(m_{i,j}, s_i \cdots w_i)$ and it has the following formal matrix description:
$$m= \begin{pmatrix} \pi^{max\{0,j-i\}}m_{i,j} \end{pmatrix}.$$
Here, if $i$ is even and $L_i$ is \textit{of type} $\textit{I}^o$ or \textit{of type} $\textit{I}^e$, then
 $$m_{i,i}=\begin{pmatrix} s_i&\pi y_i\\ \pi v_i&1+\pi z_i \end{pmatrix} \textit{or}
\begin{pmatrix} s_i&r_i&\pi t_i\\ \pi y_i&1+\pi x_i&\pi z_i\\ v_i&u_i&1+\pi w_i \end{pmatrix},$$
 respectively,
where
 $s_i\in M_{(n_i-1)\times (n_i-1)}(B\otimes_A\kappa_R)$ (resp.  $s_i\in M_{(n_i-2)\times (n_i-2)}(B\otimes_A\kappa_R)$), etc., and $s_i$ is invertible.
For the remaining $m_{i,j}$'s except for the cases explained above, $m_{i,j}\in M_{n_i\times n_j}(B\otimes_A\kappa_R)$ and $m_{i,i}$ is invertible.
Note that the description of the multiplication in $\tilde{M}(\kappa_R)$ given in Section 3.2 forces  $s_i$ and $m_{i,i}$ to be invertible.

We can write  $m_{i, i}=m_{i, i}^1+\pi\cdot m_{i, i}^2$ when $L_i$ is \textit{of type II} and
for each block of $m_{i,i}$ when $L_i$ is \textit{of type I},  $s_i=s_i^1+\pi\cdot s_i^2$ and so on.
Here,
$m_{i, i}^1, m_{i, i}^2\in M_{n_i\times n_i}(\kappa_R) \subset M_{n_i\times n_i}(B\otimes_A\kappa_R)$ when $L_i$ is \textit{of type II}
 and so on, and $\pi$ stands for $\pi\otimes 1\in B\otimes_A\kappa_R$. 
Then $m$ maps to $ m_{i, i}^1$ if $L_i$ is \textit{of type II} and
$ s_i^1$ if $L_i$ is \textit{of type I}.
Since this map is independent of the choice of $ m_{i, i}^2, s_i^2$ and so on, it is independent of the choice of $\tilde{m}$, i.e. this map is well-defined.

We note that this map is given by polynomials over $A$ of degree at most 1 as well as a group homomorphism. Thus
the above matrix interpretation induces a Hopf algebra homomorphism over $A$ 
from the coordinate ring of $\mathrm{Sp}(B_i/Y_i, h_i)$ to the coordinate ring of $\tilde{G}$,
which accordingly induces an algebraic group homomorphism $\varphi_i : \tilde{G} \rightarrow \mathrm{Sp}(B_i/Y_i, h_i)$ such that
the group homomorphism induced by $\varphi_i$ at the level of $\kappa_R$-points is the same as the map explained above.

Since $\tilde{G}$ is smooth over $\kappa$ and $\kappa$ is perfect, the set of $\kappa_R$-points of $\tilde{G}$ for all finite field extensions $\kappa_R/\kappa$
is dense in $\tilde{G}$ by Corollary 13 of Section 2.2 in \cite{BLR}.
Therefore, $\varphi_i$ is uniquely determined by the map constructed above at the level of $\kappa_R$-points.
\end{proof}

\begin{Thm}
We next assume that $i=2m-1$ is odd.
Let $\bar{q}_i$ denote the nonsingular quadratic form $\frac{1}{2^m}q$ mod 2 on $A_i/Z_i$.
Then there exists a unique morphism of algebraic groups
$$\varphi_i:\tilde{G}\longrightarrow \mathrm{O}(A_i/Z_i, \bar{q}_i)_{\mathrm{red}}$$
defined over $\kappa$, where $\mathrm{O}(A_i/Z_i, \bar{q}_i)_{\mathrm{red}}$ is the reduced subgroup scheme of $\mathrm{O}(A_i/Z_i, \bar{q}_i)$,
such that for all \'{e}tale local $A$-algebras $R$ with residue field $\kappa_R$ and
every $\tilde{m} \in \underline{G}(R)$ with reduction $m\in \tilde{G}(\kappa_R)$,
$\varphi_i(m)\in \mathrm{GL}(A_i\otimes_AR/Z_i\otimes_AR)$ is induced by the action of $\tilde{m}$ on $L\otimes_AR$
(which preserves $A_i\otimes_AR$ and $Z_i\otimes_AR$ by Lemma 4.2).
\end{Thm}
\begin{proof}
The proof of this theorem is similar to that of the above theorem which deals with the case where $i$ is  even.
Thus we only provide the image of an element $m$ of $\tilde{G}(\kappa_R)$ in
$\mathrm{O}(A_i/Z_i, \bar{q}_i)_{\mathrm{red}}(\kappa_R)$, where $R$ is an \'{e}tale local $A$-algebra with $\kappa_R$ as its residue field.
In this case, an element $m$ of $\tilde{G}(\kappa_R)$ maps to $ m_{i, i}^1$ (if $L_i$ is \textit{free}) or
$ \begin{pmatrix} m_{i, i}^1&0\\ \delta_{i-1}e_{i-1}\cdot m_{i-1, i}^1+\delta_{i+1}e_{i+1}\cdot m_{i+1, i}^1&1 \end{pmatrix}$ (if $L_i$ is \textit{bound}).
Here, $
\delta_{j} = \left\{
  \begin{array}{l l}
  1    & \quad  \textit{if $L_j$ is \textit{of type I}};\\
  0    &   \quad  \textit{if $L_j$ is \textit{of type II}},
    \end{array} \right.$
    and $e_{j}=(0,\cdots, 0, 1)$ (resp. $e_j=(0,\cdots, 0, 1, 0)$) of size $1\times n_{j}$
if $L_{j}$ is \textit{of type} $\textit{I}^o$ (resp. \textit{of type} $\textit{I}^e$).
\end{proof}

 Notice that  if the dimension of $A_i/Z_i$ is even and positive,
 then $\mathrm{O}(A_i/Z_i, \bar{q}_i)_{\mathrm{red}} (= \mathrm{O}(A_i/Z_i, \bar{q}_i))$  is disconnected.
  If the dimension of $A_i/Z_i$ is odd, then
    $\mathrm{O}(A_i/Z_i, \bar{q}_i)_{\mathrm{red}} (= \mathrm{SO}(A_i/Z_i, \bar{q}_i))$  is connected.
The dimension of $A_i/Z_i$, as a $\kappa$-vector space, is as follows:
\[
\left\{
  \begin{array}{l l}
  n_i & \quad  \textit{if $L_i$ is \textit{free}};\\
  n_i+1 & \quad \textit{if $L_i$ is \textit{bound}}.
    \end{array} \right.
\]
Note that the integer $n_i$, with $i$ odd, is always even.

\begin{Thm}
The morphism $\varphi$ defined by
$$\varphi=\prod_i \varphi_i : \tilde{G} ~ \longrightarrow  ~\prod_{i:even} \mathrm{Sp}(B_i/Y_i, h_i)\times \prod_{i:odd} \mathrm{O}(A_i/Z_i, \bar{q}_i)_{\mathrm{red}}$$
is surjective.
\end{Thm}

 \begin{proof}
 Let us first prove the theorem under the assumption that
\begin{equation}
\textit{dim $\tilde{G}$ = dim $\mathrm{Ker~}\varphi$ + $\sum_{i:\mathrm{even}}$ (dim $\mathrm{Sp}(B_i/Y_i, h_i)$)
  + $\sum_{i:\mathrm{odd}}$ (dim $\mathrm{O}(A_i/Z_i, \bar{q}_i)_{\mathrm{red}}$).}
  \end{equation}
  This equation will be proved in Appendix A.
Thus $\mathrm{Im~}\varphi$ contains the identity component of
$\prod_{i:even} \mathrm{Sp}(B_i/Y_i, h_i)\times \prod_{i:odd} \mathrm{O}(A_i/Z_i, \bar{q}_i)_{\mathrm{red}}$.
Here $\mathrm{Ker~}\varphi$ denotes the kernel of $\varphi$ and $\mathrm{Im~}\varphi$ denotes the image of $\varphi.$
Note that it is well known that the image of a homomorphism of algebraic groups is a closed subgroup.

Recall from Section 3.2 that a matrix form of an element of $\tilde{G}(R)$ for a $\kappa$-algebra $R$ is written $(m_{i,j}, s_i \cdots w_i)$
with the formal matrix interpretation $$m= \begin{pmatrix} \pi^{max\{0,j-i\}}m_{i,j} \end{pmatrix}.$$
We represent the given hermitian form  $h$ by a hermitian matrix $\begin{pmatrix} \pi^{i}\cdot h_i\end{pmatrix}$
   with $\pi^{i}\cdot h_i$ for the $(i,i)$-block and $0$ for the remaining blocks, as in Remark 3.3.(1).

 Let $\mathcal{H}$ be the set of odd integers $i$ such that $\mathrm{O}(A_i/Z_i, \bar{q}_i)_{\mathrm{red}}$ is disconnected.
Notice that $\mathrm{O}(A_i/Z_i, \bar{q}_i)_{\mathrm{red}}$ is disconnected exactly when $L_i$ with $i$ odd is \textit{free}.
We  first prove that $ \varphi_i$, for such an odd integer $i$, is surjective.
We prove this by a series of reductions, after which we will be able to assume that $L$ is of rank two.

For such an odd integer $i$  with a $free$ lattice $L_i$, we define the closed  scheme $H_i$ of $\tilde{G}$ by the equations $m_{j,k}=0$ if $ j\neq k$, and $m_{j,j}=\mathrm{id}$ if $j \neq i$.
An element of $H_i(R)$ for a $\kappa$-algebra $R$ can be represented by a matrix of the form
$$\begin{pmatrix} id&0& & \ldots& & &0\\ 0&\ddots&& & & &\\ & &id& & & & \\  \vdots & & &m_{i,i} & & &\vdots
\\ & & & & id & & \\  & & & & &\ddots &0 \\ 0& & &\ldots & &0 &id \end{pmatrix}.$$
Obviously, $H_i$ has a group scheme structure.
We claim that $\varphi_i$ is surjective from  $H_i$ to
$\mathrm{O}(A_i/Z_i, \bar{q}_i)_{\mathrm{red}}$ (recall that $Z_i=X_i$ since $L_i$ is \textit{free}). 
Note that equations defining $H_i$ are induced by the formal matrix equation $$\sigma({}^tm_{i,i})(\pi^{i}\cdot h_i)m_{i,i}=\pi^{i}\cdot h_i$$
 which is interpreted as in Remark 3.5. We emphasize that, in this formal matrix equation, we work with $m_{i,i}$, not $m$, because of the description of $H_i$.
Note that none of the congruence conditions mentioned in Section 3.1 involves any entry from $m_{i,i}$.

On the other hand, let us consider the hermitian lattice $L_i$ independently as a $\pi^i$-modular lattice.
Since there is only one non-trivial Jordan component for this lattice and $i$ is odd, the smooth integral model associated to $L_i$ is determined by  the following formal matrix equation which is interpreted as in Remark 3.5:
$$\sigma({}^tm)(\pi^{i}\cdot h_i)m=\pi^{i}\cdot h_i,$$
where $m$ is an $(n_i\times n_i)$-matrix and is not subject to any congruence condition.

We consider the map from $H_i$ to the special fiber  of the smooth integral model
associated to the hermitian lattice $L_i$ such that $m_{i,i}$ maps to $m$.
Since $m_{i,i}$ and $m$ are subject to the same set of equations,
this map is an isomorphism as algebraic groups.
In addition, this map induces compatibility between
the morphism $\varphi_i$ from  $H_i$ to $\mathrm{O}(A_i/Z_i, \bar{q}_i)_{\mathrm{red}}$
and the morphism from the special fiber of the smooth integral model associated to  $L_i$ to $\mathrm{O}(A_i/Z_i, \bar{q}_i)_{\mathrm{red}}$.
Thus, in order to show that $\varphi_i$ is surjective from $H_i$ to
$\mathrm{O}(A_i/Z_i, \bar{q}_i)_{\mathrm{red}}$, we may and do assume that $L=L_i$ and in this case $Z_i=X_i=\pi L_i$.
For simplicity, we can also assume that $i=1$.

Because of Equation (4.1) made at the beginning of the proof, the dimension of the image of $\varphi_i$, as a $\kappa$-algebraic group, is the same as that of
$\mathrm{O}(A_i/Z_i, \bar{q}_i)_{\mathrm{red}} (=\mathrm{O}(L_i/\pi L_i, \bar{q}_i))$.
Thus, the image of $\varphi_i$ contains the identity component of $\mathrm{O}(L_i/\pi L_i, \bar{q}_i)$, namely $\mathrm{SO}(L_i/\pi L_i, \bar{q}_i)$.
Since $\mathrm{O}(L_i/\pi L_i, \bar{q}_i)$ has  two connected components, we only need to show the surjectivity of $\varphi_i$
at the level of $\kappa$-points and it suffices to show that the image of $\varphi_i(\kappa)$ contains at least one element which is not contained in $\mathrm{SO}(L_i/\pi L_i, \bar{q}_i)(\kappa)$, where $\mathrm{SO}(L_i/\pi L_i, \bar{q}_i)(\kappa)$ is the group of $\kappa$-points of the algebraic group $\mathrm{SO}(L_i/\pi L_i, \bar{q}_i)$.

Recall that $L_i=\bigoplus_{\lambda}H_{\lambda}\oplus A(2, 2b, \pi)$ for a certain $b\in A$, cf. Theorem 2.10.
We  consider the orthogonal group associated to the quadratic $\kappa$-space $A(2, 2b, \pi)/\pi A(2, 2b, \pi)$ of dimension $2$.
Then this group is embedded into $\mathrm{O}(L_i/\pi L_i, \bar{q}_i)(\kappa)$ as a closed subgroup
and we denote the embedded group by $\mathrm{O}(A(2, 2b, \pi)/\pi A(2, 2b, \pi), \bar{q}_i)(\kappa)$.

We express an element $m_{i,i}\in H_i(R)$, for a $\kappa$-algebra $R$, as $\begin{pmatrix} x&y\\ z&w \end{pmatrix}$
such that $x=x^1+\pi x^2$ and so on, where 
$x^1, x^2 \in M_{(n_i-2)\times(n_i-2)}(R)\subset M_{(n_i-2)\times(n_i-2)}(R\otimes_AB)$ and $\pi$ stands for $1\otimes \pi\in R\otimes_AB$.
Consider the closed subscheme of $H_i$ defined by the equations $x=id, y=0, z=0$.
An argument similar to one used above to reduce to the case where $L = L_i$ shows that this subscheme is isomorphic to the special fiber
of the smooth integral model associated to the hermitian lattice $A(2, 2b, \pi)$ of rank $2$.
Then under the map $\varphi_i(\kappa)$, an element of this subgroup maps to an element of
$\mathrm{O}(A(2, 2b, \pi)/\pi A(2, 2b, \pi), \bar{q}_i)(\kappa)$ of the form $\begin{pmatrix} id&0\\ 0&w^1 \end{pmatrix}$.
Note that $\mathrm{O}(A(2, 2b, \pi)/\pi A(2, 2b, \pi), \bar{q}_i)(\kappa)$ is not contained in $\mathrm{SO}(L_i/\pi L_i, \bar{q}_i)(\kappa)$.
Thus it suffices to show that the restriction of $\varphi_i(\kappa)$ to the above subgroup of $H_i(\kappa)$,
which is given by letting  $x=id, y=0, z=0$, is surjective onto
$\mathrm{O}(A(2, 2b, \pi)/\pi A(2, 2b, \pi), \bar{q}_i)(\kappa)$ and we may and do assume that $L=L_i=A(2, 2b, \pi)$ is of rank $2$.

Let $m_{i,i}=\begin{pmatrix} r&s\\ t&v \end{pmatrix}$ be an element of $H_i(\kappa)$ such that $r=r_1+\pi r_2$ and so on,
where $r_1, r_2 \in R\subset R\otimes_AB$ and $\pi$ stands for $1\otimes \pi\in R\otimes_AB$.
Recall that $\pi=1+\sqrt{1+2u}$ for a certain unit $u\in A$ so that
$\pi+\sigma(\pi)=2$, $\sigma(\pi)=\epsilon \pi$ with $\epsilon\equiv 1$ mod $\pi$,  and $\pi^2\equiv \left(\sigma(\pi)\right)^2 \equiv \pi\cdot \sigma(\pi) \equiv 2u$ mod $2\pi$ as mentioned in Section 2.1.
Let $\bar{u}\in \kappa$ be the reduction of $u$ modulo $\pi$.
 Then the equations defining  $H_i(\kappa)$  are
$$r_1^2+r_1t_1+bt_1^2=1, r_1v_1+t_1s_1=1, r_1s_1+bt_1v_1+\bar{u}(r_2v_1+r_1v_2+t_2s_1+t_1s_2)=0,   s_1^2+s_1v_1+bv_1^2=b.$$
Under the map $\varphi_i(\kappa)$, $m_{i,i}$ maps to  $\begin{pmatrix} r_1&s_1\\ t_1&v_1 \end{pmatrix}$.
Note that the quadratic form $\bar{q}_i$ restricted to $A(2, 2b, \pi)/\pi A(2, 2b, \pi)$ is given by the matrix
$\begin{pmatrix} 1&1\\ 0&b \end{pmatrix}$.

We now choose an element of $H_i(\kappa)$ by setting
\[r_1=s_1=v_1=1, t_1=0, 1+\bar{u}(r_2+v_2+t_2)=0.  \]
Then under the morphism $\varphi_i(\kappa)$,
this element maps to  $\begin{pmatrix} 1&1\\ 0&1 \end{pmatrix} \in \mathrm{O}(A(2, 2b, \pi)/\pi A(2, 2b, \pi), \bar{q}_i)(\kappa)$
whose Dickson invariant is nontrivial so that it is not contained   in $\mathrm{SO}(A(2, 2b, \pi)/\pi A(2, 2b, \pi), \bar{q}_i)(\kappa)$.

Therefore, $\varphi_i(\kappa)$ induces a surjection from $H_i(\kappa)$ to $\mathrm{O}(A(2, 2b, \pi)/\pi A(2, 2b, \pi), \bar{q}_i)(\kappa)$
for $i\in \mathcal{H}$.\\



We now prove that $\varphi=\prod_i \varphi_i$ is surjective.
We consider the morphism
\[\prod_{i\in\mathcal{H}} H_i \longrightarrow \tilde{G},\]
$(h_i)_{i\in\mathcal{H}}\mapsto\prod_{i\in\mathcal{H}}h_i$.
By observing a formal matrix form of an element of $H_i(R)$ for a $\kappa$-algebra $R$ as given above,
it is easy to see the following two facts:
Firstly,   $H_i$ and $H_j$ commute with each other in the sense that $h_i\cdot h_j=h_j\cdot h_i$ for all $i \neq j$, where $h_i\in H_i(R)$ and $ h_j\in H_j(R)$
for a $\kappa$-algebra $R$.
Based on this, the above morphism becomes a group homomorphism.
Secondly,
$H_i\cap H_j=0$ for all $i\neq j$.
This fact implies that the morphism $H_i\times H_j \longrightarrow \tilde{G},
(h_i, h_j)\mapsto h_i\cdot h_j$ is injective and so $H_i\times H_j$ is a closed subgroup scheme of $\tilde{G}$.
A matrix form of an element of $H_i(R)$ also implies that  $(H_i\times H_j)\cap H_k=0$
 for all pairwise different three integers $i, j, k$  
and so the morphism $(H_i\times H_j)\times H_k \longrightarrow \tilde{G},
(h_i, h_j, h_k)\mapsto h_i\cdot h_j\cdot h_k$ is injective. Thus $H_i\times H_j\times H_k$ is a closed subgroup scheme of $\tilde{G}$.
Therefore, by repeating this argument,
 the product $\prod_{i\in\mathcal{H}} H_i$ is embedded into $\tilde{G}$ as a closed subgroup scheme.
Since $\varphi_i|_{H_j}$ is trivial for  $i\neq j$ with $i,j\in \mathcal{H}$, the morphism
$$\prod_{i\in\mathcal{H}}\varphi_i : \prod_{i\in\mathcal{H}}H_i \rightarrow \prod_{i\in\mathcal{H}} \mathrm{O}(A_i/Z_i, \bar{q}_i)_{\mathrm{red}}$$
is surjective.
Therefore, $\varphi$ is surjective.
Now it suffices to prove Equation (4.1) made at the beginning of the proof, which is the next lemma.
\end{proof}

\begin{Lem}

$\mathrm{Ker~}\varphi $ is  smooth and unipotent  of dimension $l$.
In addition, the number of connected components of $\mathrm{Ker~}\varphi $ is $2^\beta$.
 Here,
 \begin{itemize}
\item   $l$ is such that
\[\textit{$l$ +
  $\sum_{i:\mathrm{even}}$ (dim $~\mathrm{Sp}(B_i/Y_i, h_i)$)
  + $\sum_{i:\mathrm{odd}}$ (dim $~\mathrm{O}(A_i/Z_i, \bar{q}_i)_{\mathrm{red}}$) = dim $\tilde{G}$.}\]
 \item  $\beta$ is the number of even integers $j$ such that $L_j$ is \textit{of type I} and $L_{j+2}$ is \textit{of type II}.
\end{itemize}
\qed
\end{Lem}
Recall that the zero lattice is \textit{of type II}.
The proof is postponed to Appendix A.

\begin{Rmk}
We summarize the description of Im $\varphi_i$ as follows.
     \[
      \begin{array}{c|c}
      \mathrm{Type~of~lattice~}  L_i ~\mathrm{and}~ i & \mathrm{Im~}  \varphi_i \\
      \hline
      \textit{II},\ \  \textit{i : even} & \mathrm{Sp}(n_i, h_i)\\
      \textit{I}^o,\ \  \textit{i : even} & \mathrm{Sp}(n_i-1, h_i)\\
      \textit{I}^e,\ \  \textit{i : even} & \mathrm{Sp}(n_i-2, h_i)\\
      \textit{free},\ \  \textit{i : odd} & \mathrm{O}(n_i, \bar{q}_i)\\
      \textit{bound},\ \  \textit{i : odd} & \mathrm{SO}(n_i+1, \bar{q}_i)
      \end{array}
    \]
    Let $i$ be odd and $L_i$ be \textit{free}.
Then $A_i/Z_i=L_i/\pi L_i$ is a $\kappa$-vector space with even dimension. Thus there are two different orthogonal groups
$\mathrm{O}(A_i/Z_i, \bar{q}_i) ~(=\mathrm{O}(n_i, \bar{q}_i))$,
according to split or non-split.

By Theorem 2.10, we have that $L_i=\bigoplus_{\lambda}H_{\lambda}\oplus A(2, 2b_i, \pi)$ for certain $b_i\in A$.
Thus the orthogonal group $\mathrm{O}(A_i/Z_i, \bar{q}_i) ~(=\mathrm{O}(n_i, \bar{q}_i))$ is split
if and only if the quadratic space  $A(2, 2b_i, \pi)/\pi A(2, 2b_i, \pi)$ is isotropic.
Recall that $\pi+\sigma(\pi)=2$ and $\pi=1+\sqrt{1+2u}$ for a certain unit $u\in A$.
Using this, the quadratic form on $A(2, 2b_i, \pi)/\pi A(2, 2b_i, \pi)$  is $q(x, y)=x^2+xy+\bar{b}_iy^2$,
where $\bar{b}_i$ is the reduction of $b_i$ in $\kappa$.

We consider the identity $q(x, y)=x^2+xy+\bar{b}_iy^2=0$.
If $y=0$, then $x=0$. Assume that $y\neq 0$.
Then we have that $\bar{b}_i=(x/y)^2+x/y$.

Thus we can see that there exists a solution of the equation $z^2+z=\bar{b}_i$ over $\kappa$ if and only if $q(x, y)$ is isotropic if and only if
$\mathrm{O}(A_i/Z_i, \bar{q}_i) ~(=\mathrm{O}(n_i, \bar{q}_i))$ is split.
\end{Rmk}

\subsection{The construction of component groups}

The purpose of this subsection is to define a surjective morphism from $\tilde{G}$ to $(\mathbb{Z}/2\mathbb{Z})^{\beta}$,
where $\beta$ is the number of even integers $j$ such that $L_j$ is \textit{of type I} and $L_{j+2}$ is \textit{of type II} as defined in Lemma 4.6.
   \begin{Def}
    We set $L^0=L$ and inductively define, for positive integers $i$,
    \[L^i:=\{x\in L^{i-1} | h(x, L^{i-1})\subset (\pi^i)\}.\]
    When $i=2m$ is even,
      $$L^{2m}=\pi^m(L_0\oplus L_1)\oplus\pi^{m-1}(L_2\oplus L_3)\oplus \cdots \oplus \pi(L_{2m-2}\oplus L_{2m-1})\oplus \bigoplus_{i\geq 2m}L_i.$$
    \end{Def}
     We choose a Jordan splitting for the hermitian lattice $(L^{2m}, \frac{1}{(\pi\cdot\sigma(\pi))^{m}}h)$ as follows:
     $$L^{2m}=\bigoplus_{i \geq 0} M_i,$$ where
     $$M_0=\pi^mL_0\oplus\pi^{m-1}L_2\oplus \cdots \oplus \pi L_{2m-2}\oplus L_{2m},$$
     $$M_1=\pi^mL_1\oplus\pi^{m-1}L_3\oplus \cdots \oplus \pi L_{2m-1}\oplus L_{2m+1}$$
     $$\mathrm{and}~ M_k=L_{2m+k} \mathrm{~if~} k\geq 2.$$
Here, $M_i$ is $\pi^i$-modular.
We caution that the hermitian form we use on $L^{2m}$ is not $h$, but its rescaled version $\frac{1}{(\pi\cdot\sigma(\pi))^{m}}h$.
Thus   $M_i$ is $\pi^i$-modular, not $\pi^{2m+i}$-modular.
\begin{Def}
We define $C(L)$ to be the sublattice of $L$ such that $$C(L)=\{x\in L \mid h(x,y) \in (\pi) \ \ \mathrm{for}\ \ \mathrm{all}\ \ y \in B(L)\}.$$
\end{Def}

    We  choose any even integer $j$ such that $L_{j}$ is \textit{of type} $\textit{I}$ and $L_{j+2}$ is \textit{of type} $\textit{II}$
    (possibly zero, by our convention), and consider the Jordan splitting $\bigoplus_{i \geq 0} M_i$ of $L^j$ defined above.
    We stress that $M_0$ is nonzero and \textit{of type} $\textit{I}$, since it contains $L_{j}$ as a direct summand
    so that $n(M_0)=s(M_0)$ (cf. Definition 2.1.c)), and $M_2=L_{j+2}$ is \textit{of type} $\textit{II}$.
     Choose a basis $(\langle e_i\rangle, e)$ (resp. $(\langle e_i\rangle, a, e)$) for $M_0$
     so that $M_0=\bigoplus_{\lambda}H_{\lambda}\oplus K$ 
      when the rank of $M_0$ is odd (resp. even).
      Here, we follow the notation from Theorem 2.10.
     Then $B(L^{j})$ is spanned by $$(\langle e_i\rangle, \pi e) ~(resp.~ (\langle e_i\rangle, \pi a, e)) \textit{~and~}  M_1 \oplus (\bigoplus_{i\geq 2} M_i)$$
     and  $C(L^{j})$ is spanned by
     $$(\langle \pi e_i\rangle, e) ~(resp.~ (\langle \pi e_i\rangle, \pi a, e)) \textit{~and~}  M_1 \oplus (\bigoplus_{i\geq 2} M_i).$$\\

We now construct a morphism $\psi_j : \tilde{G} \rightarrow \mathbb{Z}/2\mathbb{Z}$ as follows
(There are 2 cases depending on whether $M_0$ is \textit{of type} $\textit{I}^e$ or \textit{of type} $\textit{I}^o$.):

1. Firstly, we assume that $M_0$ is \textit{of type} $\textit{I}^e$.
We choose a Jordan splitting for the hermitian lattice $(C(L^j), \frac{1}{(\pi\cdot\sigma(\pi))^{m}}h)$ as follows:
$$C(L^j)=\bigoplus_{i \geq 1} M_i^{\prime},$$
where $$M_1^{\prime}=(\pi)a\oplus Be\oplus M_1, ~~~ M_2^{\prime}=(\oplus_i(\pi)e_i)\oplus M_2, ~~~ \mathrm{and}~ M_k^{\prime}=M_k \mathrm{~if~} k\geq 3.$$
Here, $M_i^{\prime}$ is $\pi^i$-modular and $(\pi)$ is the ideal of $B$ generated by a uniformizer $\pi$.
Notice that $M_2^{\prime}$ is \textit{of type II}, since both $\oplus_i(\pi)e_i$ and $M_2$ are \textit{of type II},
so that $M_1^{\prime}$ is   \textit{free}.

If $m$ is an element of the group of $R$-points of the naive integral model
associated to the hermitian lattice $L$, for a flat $A$-algebra $R$, then $m$ stabilizes the hermitian lattice $(C(L^j)\otimes_AR, \frac{1}{(\pi\cdot\sigma(\pi))^{m}}h\otimes 1)$ as well.
If we use this fact in the case of  an $F$-alegbra $R$, where $F$ is the quotient field of $A$, then
we obtain a morphism of algebraic groups
from the unitary group associated to the hermitian space $L\otimes_AF$ to the unitary group associated to
the hermitian space $(C(L^j)\otimes_AF, \frac{1}{(\pi\cdot\sigma(\pi))^{m}}h)$ by \textit{Yoneda's lemma}.
Furthermore, if we use the above fact in the case of an \'{e}tale $A$-algebra $R$, then the  morphism between unitary groups is extended
to give a map from the group of $R$-points of the naive integral model associated to  the hermitian lattice $L$
to  that of the hermitian lattice $(C(L^j), \frac{1}{(\pi\cdot\sigma(\pi))^{m}}h)$.
Note that the naive integral model and the associated smooth integral model have the same generic fiber
and are the same at the level of \'etale $A$-points.
Thus by Proposition 2.3 of \cite{Yu}, the  morphism between unitary groups is uniquely extended to a morphism of group schemes
from the smooth integral model associated to $L$
to the smooth integral model associated to $(C(L^j), \frac{1}{(\pi\cdot\sigma(\pi))^{m}}h)$ such that
the map induced from it at the level of \'etale $A$-points is the same as that described above.
Let $G_j$ denote the special fiber of the latter smooth  integral model.
We now have a morphism from $\tilde{G}$ to $G_j$. 
Moreover, since $M_1^{\prime}$ is \textit{free} and nonzero, we have a morphism from $G_j$ to the even orthogonal group associated to $M_1^{\prime}$
 as explained in Section 4.1.
Thus, the Dickson invariant of this orthogonal group induces the morphism
$$\psi_j : \tilde{G} \longrightarrow \mathbb{Z}/2\mathbb{Z}.$$
\textit{ }

2. We next assume that $M_0$ is \textit{of type} $\textit{I}^o$.
We choose a Jordan splitting for the hermitian lattice $(C(L^j), \frac{1}{(\pi\cdot\sigma(\pi))^{m}}h)$ as follows:
$$C(L^j)=\bigoplus_{i \geq 0} M_i^{\prime},$$
where $$M_0^{\prime}= Be, ~~~M_1^{\prime}=M_1, ~~~ M_2^{\prime}=(\oplus_i(\pi)e_i)\oplus M_2, ~~~ \mathrm{and}~ M_k^{\prime}=M_k \mathrm{~if~} k\geq 3.$$
Here, $M_i^{\prime}$ is $\pi^i$-modular and $(\pi)$ is the ideal of $B$ generated by a uniformizer $\pi$.
Notice that the rank of the $\pi^0$-modular lattice $M_0^{\prime}$ is 1 and the lattice $M_2^{\prime}$ is \textit{of type II}.
If $G_j$ denotes the special fiber of the smooth integral model associated to the hermitian lattice $(C(L^j), \frac{1}{(\pi\cdot\sigma(\pi))^{m}}h)$,
then we  have a morphism from $\tilde{G}$ to $G_j$ as in the above case $1$.

We now consider the new hermitian lattice $M_0^{\prime}\oplus C(L^j)$.
Then for a flat $A$-algebra $R$, there is a natural  embedding from the group of $R$-points of the naive integral model
associated to the hermitian lattice $(C(L^j), \frac{1}{(\pi\cdot\sigma(\pi))^{m}}h)$
to that of  the hermitian lattice $M_0^{\prime}\oplus C(L^j)$ such that
$m$ maps to $\begin{pmatrix} 1&0 \\ 0&m \end{pmatrix}$,
where $m$ is an element of the former group.
As in the previous case 1, the above fact induces a closed immersion of algebraic groups from the unitary group associated to
the hermitian space $(C(L^j)\otimes_AF, \frac{1}{(\pi\cdot\sigma(\pi))^{m}}h)$
to the unitary group associated to the hermitian space $(M_0^{\prime}\oplus C(L^j))\otimes_AF$
and its extension at the level of \'{e}tale $A$-algebra points between the associated naive integral models.
Thus by Proposition 2.3 of \cite{Yu}, the  morphism between unitary groups is uniquely extended to a
  morphism of group schemes from the smooth integral model associated to the hermitian lattice $(C(L^j), \frac{1}{(\pi\cdot\sigma(\pi))^{m}}h)$
to the smooth integral model associated to the hermitian lattice $M_0^{\prime}\oplus C(L^j)$ such that
the map induced from it at the level of \'etale $A$-points is the same as that described above.
In Remark 4.10, we describe this morphism explicitly in terms of matrices.


Thus we have a morphism from the special fiber $G_j$ of the smooth integral model associated to $C(L^j)$  to the special fiber
$G_j'$ of the smooth integral model associated to  $M_0^{\prime}\oplus C(L^j)$.
Note that $(M_0^{\prime}\oplus M_0^{\prime})\oplus \bigoplus_{i \geq 1} M_i^{\prime}$ is a Jordan splitting of the hermitian lattice $M_0^{\prime}\oplus C(L^j)$.
Let $G_j''$ be the special fiber of the smooth integral model associated to
$C((M_0^{\prime}\oplus M_0^{\prime})\oplus \bigoplus_{i \geq 1} M_i^{\prime})$.
Since the $\pi^0$-modular lattice $M_0^{\prime}\oplus M_0^{\prime}$ is \textit{of type} $\textit{I}^e$,
we have a morphism $G_j'\rightarrow \mathbb{Z}/2\mathbb{Z}$ obtained by factoring through $G_j''$
and the corresponding even orthogonal group with the Dickson invariant as constructed in the  case 1.
$\psi_j$ is defined to be the composite
$$\psi_j : \tilde{G} \rightarrow G_j \rightarrow G_j' \rightarrow  \mathbb{Z}/2\mathbb{Z}.$$


\begin{Rmk}
In this remark, we describe the morphism from the smooth integral model $\underline{G}_j$ associated to the hermitian lattice
$(C(L^j), \frac{1}{(\pi\cdot\sigma(\pi))^{m}}h)$
to the smooth integral model $\underline{G}_j'$
associated to the hermitian lattice $M_0^{\prime}\oplus C(L^j)$ as given in case 2 above, in terms of matrices.
Let $R$ be a flat $A$-algebra.
We choose an element in $\underline{G}_j(R)$ 
and express it as a  matrix $m= \begin{pmatrix} \pi^{max\{0,j-i\}}m_{i,j} \end{pmatrix}$.
Then $m_{0,0}=\begin{pmatrix} 1+\pi z_0 \end{pmatrix}$ since $M_0'$ is \textit{of type I} with rank 1 so that we may and do write $m$ as
$m= \begin{pmatrix} 1+\pi z_0 &m_1\\m_2&m_3 \end{pmatrix}$.
We consider a morphism from $\underline{G}_j$ to $\mathrm{Aut}_{B}(M_0^{\prime}\oplus C(L^j))$
such that $m$ maps to $\begin{pmatrix} 1&0 \\ 0&m \end{pmatrix}=\begin{pmatrix} 1&0 &0\\ 0&1+\pi z_0 &m_1\\0& m_2&m_3\end{pmatrix}$,
where the set of $R$-points of the group scheme $\mathrm{Aut}_{B}(M_0^{\prime}\oplus C(L^j))$ is
the automorphism group of $(M_0^{\prime}\oplus C(L^j)) \otimes_A R$ by ignoring the hermitian form.
Then the image of this morphism is represented by an affine group scheme which is isomorphic to $\underline{G}_j$.
Note that  $\begin{pmatrix} 1&0 \\ 0&m \end{pmatrix}=\begin{pmatrix} 1&0 &0\\ 0&1+\pi z_0 &m_1\\0& m_2&m_3\end{pmatrix}$
preserves the hermitian form attached to the lattice $M_0^{\prime}\oplus C(L^j)$.

We claim that $\begin{pmatrix} 1&0 \\ 0&m \end{pmatrix}$ is contained in $\underline{G}_j'(R)$.
If this is true, then the above matrix description defines a morphism from $\underline{G}_j$ to $\underline{G}_j'$ by \textit{Yoneda's lemma}
since $\underline{G}_j$ is flat.
Furthermore, this matrix description is the same as that of naive integral models explained in the above case 2
when $R$ is an $F$-algebra or an \'etale $A$-algebra,
since the naive integral model and the associated smooth integral model have the same generic fiber
and are the same at the level of \'etale $A$-points.
Since the desired morphism is completely determined at the level of $F$-algebra points and  \'etale $A$-algebra points
by Proposition 2.3 of \cite{Yu},
the morphism from  $\underline{G}_j$ to $\underline{G}_j'$ obtained by the above matrix description is
the morphism we want to describe.

We rewrite the hermitian lattice $M_0^{\prime}\oplus C(L^j)$ as $(M_0^{\prime}\oplus M_0^{\prime})\oplus (\bigoplus_{i \geq 1} M_i^{\prime})$.
Let $(e_1, e_2)$ be a basis for $(M_0^{\prime}\oplus M_0^{\prime})$
so that the corresponding Gram matrix of $(M_0^{\prime}\oplus M_0^{\prime})$  is   $\begin{pmatrix} a&0 \\ 0&a \end{pmatrix}$, where $a \equiv 1$ mod 2.
Then the hermitian lattice $(M_0^{\prime}\oplus M_0^{\prime})$ has Gram matrix $\begin{pmatrix} a&a \\ a&2a \end{pmatrix}$
with respect to the basis $(e_1, e_1+e_2)$.
$(M_0^{\prime}\oplus M_0^{\prime})$ is \textit{unimodular of type $I^e$} with rank 2.
With this basis, $\begin{pmatrix} 1&0 &0\\ 0&1+\pi z_0 &m_1\\0& m_2&m_3\end{pmatrix}$ becomes
$\begin{pmatrix} 1&-\pi z_0 &-m_1\\ 0&1+\pi z_0 &m_1\\0& m_2&m_3\end{pmatrix}$.

On the other hand, an element of $\underline{G}_j'(R)$,
with respect to a basis for $M_0^{\prime}\oplus C(L^j)$ obtained by putting together
the basis $(e_1, e_1+e_2)$ for $(M_0^{\prime}\oplus M_0^{\prime})$ and a basis for $C(L^j)$, is given by an expression
$\begin{pmatrix} 1+\pi x_0'&-\pi z_0' & m_1'  \\ u_0'&1+\pi w_0' &m_1'' \\ m_2' &m_2''  & m_3''\end{pmatrix}$, cf. Section 3.1.
Then we can easily see that the congruence conditions on $m_1, m_2, m_3$ are the same as those of $m_1'', m_2'', m_3''$, respectively,
and that the congruence conditions on $m_1'$ is the same as those of $m_1''$.
Thus $\begin{pmatrix} 1&-\pi z_0 &-m_1\\ 0&1+\pi z_0 &m_1\\0& m_2&m_3\end{pmatrix}$ is an element of $\underline{M}_j^{\ast}(R)$,
where $\underline{M}_j^{\ast}$ is the group scheme in Section 3.2 associated to $M_0^{\prime}\oplus C(L^j)$
so that $\underline{G}_j'$ is defined as the closed subgroup scheme of $\underline{M}_j^{\ast}$ stabilizing the hermitian form
on $M_0^{\prime}\oplus C(L^j)$.


In conclusion, $\begin{pmatrix} 1&-\pi z_0 &-m_1\\ 0&1+\pi z_0 &m_1\\0& m_2&m_3\end{pmatrix}\in \underline{M}_j^{\ast}(R)$
preserves the hermitian form on $M_0^{\prime}\oplus C(L^j)$.
Therefore, it is an element of $\underline{G}_j'(R)$.

To summarize, if $R$ is a non-flat $A$-algebra, then
we can write an element of $\underline{G}_j(R)$ formally as $m= \begin{pmatrix} 1+\pi z_0 &m_1\\m_2&m_3 \end{pmatrix}$.
Then the image of $m$ in $\underline{G}_j'(R)$ is $\begin{pmatrix} 1&-\pi z_0 &-m_1\\ 0&1+\pi z_0 &m_1\\0& m_2&m_3\end{pmatrix}$
with respect to a basis as explained above.
\end{Rmk}
\textit{ }

3. Combining all cases, we have the morphism $$\psi=\prod_j \psi_j : \tilde{G} \longrightarrow (\mathbb{Z}/2\mathbb{Z})^{\beta},$$
where $\beta$ is the number of even integers $j$ such that $L_j$ is \textit{of type I} and  $L_{j+2}$ is \textit{of type II} (possibly zero, by our convention).

We now have the following result.
\begin{Thm}
 The morphism $$\psi=\prod_j \psi_j : \tilde{G} \longrightarrow (\mathbb{Z}/2\mathbb{Z})^{\beta}$$ is surjective.

Moreover, the morphism $$\varphi \times \psi : \tilde{G} \longrightarrow \prod_{i:even} \mathrm{Sp}(B_i/Y_i, h_i)\times \prod_{i:odd} \mathrm{O}(A_i/Z_i, \bar{q}_i)_{\mathrm{red}} \times (\mathbb{Z}/2\mathbb{Z})^{\beta}$$
is also surjective.

\end{Thm}

\begin{proof}
We first show that $\psi_j$ is surjective.
Recall that for such an even integer $j$, $L_j$ is \textit{of type I} and $L_{j+2}$ is \textit{of type II} (possibly zero by our convention).
We define the closed subgroup scheme $F_j$ of $\tilde{G}$ defined by the following equations: 
\begin{itemize}
\item $m_{i,k}=0$ \textit{if $i\neq k$};
\item $m_{i,i}=\mathrm{id}$ \textit{if $i\neq j$};
\item and for $m_{j,j}$,
\[\left \{
  \begin{array}{l l}
  s_j=\mathrm{id~}, y_j=0, v_j=0 & \quad  \textit{if $L_i$ is \textit{of type} $\textit{I}^o$};\\
  s_j=\mathrm{id~}, r_j=t_j=y_j=v_j=u_j=w_j=0 & \quad \textit{if $L_i$ is \textit{of type} $\textit{I}^e$}.\\
    \end{array} \right.\]
\end{itemize}
A formal matrix form of an element of $F_j(R)$ for a $\kappa$-algebra $R$ is then
\[\begin{pmatrix} id&0& & \ldots& & &0\\ 0&\ddots&& & & &\\ & &id& & & & \\  \vdots & & &m_{j,j} & & &\vdots
\\ & & & & id & & \\  & & & & &\ddots &0 \\ 0& & &\ldots & &0 &id \end{pmatrix}\]
such that
\[m_{j,j}=\left\{
\begin{array}{l l}
\begin{pmatrix}id&0\\0&1+\pi z_j \end{pmatrix} & \quad \textit{if $L_j$ is of type $I^o$};\\
\begin{pmatrix}id&0&0\\0&1+\pi x_j&\pi z_j\\0&0&1 \end{pmatrix} & \quad \textit{if $L_j$ is of type $I^e$}.
\end{array}\right.\]
In Lemma A.9, we will show  that   $F_j$ is isomorphic to $ \textbf{A}^{1} \times \mathbb{Z}/2\mathbb{Z}$ as a $\kappa$-variety
so that it has exactly two connected components,
by enumerating equations defining $F_j$ as a closed subvariety of an affine space of dimension $2$ (resp. $4$)
if $L_j$ is \textit{of type $\textit{I}^o$} (resp. \textit{of type $\textit{I}^e$}).
Here,  $\textbf{A}^{1}$ is an affine space of dimension $1$.
These equations are necessary in this theorem and thus we state them in Equation (4.2) below.
We  refer to Lemma A.9 for the proof.
Let $\alpha$ be the unit in $B$ such that  $\epsilon=1+\alpha\pi$ as explained in Section 2.1,
and $\bar{\alpha}$ be the image of $\alpha$ in $\kappa$.
We write $x_j=x_j^1+\pi x_j^2$ and $z_j=z_j^1+\pi z_j^2$,
where $x_j^1, x_j^2, z_j^1, z_j^2 \in R \subset R\otimes_AB$ and $\pi$ stands for $1\otimes \pi\in R\otimes_AB$.
Then the equations defining $F_j$ as a closed subvariety of an affine space of dimension $2$ (resp. $4$) are
\begin{equation}
\left\{
\begin{array}{l l}
(z_j^1/\bar{\alpha})+ (z_j^1/\bar{\alpha})^2=0 & \quad \textit{if $L_j$ is of type $I^o$};\\
x_j^1=z_j^1, (z_j^1/\bar{\alpha})+ (z_j^1/\bar{\alpha})^2=0, z_{j}^2+x_{j}^2+x_{j}^1z_{j}^1=0 & \quad \textit{if $L_j$ is of type $I^e$}.
\end{array}\right.
\end{equation}

The proof of the surjectivity of $\psi_j$ is given below.
The main idea is to show that $\psi_j|_{F_j}$ is surjective.
There are 4 cases according to the types of $M_0$ and $L_j$.
Recall that $\bigoplus_{i \geq 0} M_i$ is a Jordan splitting of  a rescaled hermitian lattice $(L^{j}, \frac{1}{(\pi\cdot\sigma(\pi))^{j/2}}h)$
and that $M_0=\pi^{j/2}L_0\oplus\pi^{j/2-1}L_2\oplus \cdots \oplus \pi L_{j-2}\oplus L_{j}$.

\begin{enumerate}
\item
Assume that both $M_0$ and $L_j$ are \textit{of type $I^e$}.
In this case and the next case, we will describe $\psi_j|_{F_j} : F_j \rightarrow \mathbb{Z}/2\mathbb{Z}$ explicitly in terms of a formal matrix.
To do that, we will first describe a morphism from $F_j$ to the special fiber of the smooth integral model associated to $L^j$
and  then to $G_j$. Recall that $G_j$ is  the special fiber of the smooth integral model associated to $C(L^j)=\bigoplus_{i \geq 1} M_i^{\prime}$.
Then we will describe a morphism from $F_j$ to the even orthogonal group associated to $M_1'$ and compute the Dickson invariant of the image of an element of $F_j$ in this orthogonal group.


We write $M_0=N_0\oplus L_j$, where $N_0$ is unimodular with even rank.
Thus $N_0$ is either \textit{of type II} or \textit{of type $I^e$}.
Assume that $N_0$ is \textit{of type $I^e$}.
Then we can write $N_0=(\oplus_{\lambda'}H_{\lambda'})\oplus A(1, 2b, 1)$ and $L_j=(\oplus_{\lambda''}H_{\lambda''})\oplus A(1, 2b', 1)$ by Theorem 2.10,
where  $H_{\lambda'}=H(0)=H_{\lambda''}$ and $b, b'\in A$.
Thus we write  $M_0=(\oplus_{\lambda}H_{\lambda})\oplus A(1, 2b, 1)\oplus A(1, 2b', 1)$, where $H_{\lambda}=H(0)$.
For this choice of a basis of $L^j=\bigoplus_{i \geq 0} M_i$,
the image of a fixed element of $F_j$ in the special fiber of the smooth integral model associated to $L^j$ is
$$\begin{pmatrix} id&0 &0\\ 0 &\begin{pmatrix} 1+\pi x_j & \pi z_j\\ 0 & 1 \end{pmatrix}  &0 \\ 0& 0 &id \end{pmatrix}.$$
Here, $id$ in the $(1,1)$-block corresponds to the direct summand $(\oplus_{\lambda}H_{\lambda})\oplus A(1, 2b, 1)$ of $M_0$ and
the diagonal block $\begin{pmatrix} 1+\pi x_j & \pi z_j\\ 0 & 1 \end{pmatrix} $ corresponds to the direct summand $A(1, 2b', 1)$ of $M_0$.

Let $(e_1, e_2, e_3, e_4)$ be a basis for the direct summand $A(1, 2b, 1)\oplus A(1, 2b', 1)$ of $M_0$.
Since this is  \textit{unimodular of type $I^e$}, we can choose another basis based on Theorem 2.10.
Namely, if we choose $(-2be_1+e_2, (2b'-1)e_1+e_3-e_4, e_3, e_2+e_4)$ as another basis, then
$A(1, 2b, 1)\oplus A(1, 2b', 1)$ becomes $A(2b(2b-1), 2b'(2b'-1), -(2b-1)(2b'-1))\oplus A(1, 2(b+b'), 1)$.
Since $A(2b(2b-1), 2b'(2b'-1), -(2b-1)(2b'-1))$ is \textit{unimodular of type II}, it is isomorphic to $H(0)$ by Theorem 2.10.
Thus we can write  $M_0=(\oplus_{\lambda}H_{\lambda})\oplus H(0)\oplus A(1, 2(b+b'), 1)$.
For this basis, the image of a fixed element of $F_j$ in the special fiber of the smooth integral model associated to $L^j$ is
$$\begin{pmatrix} \ast&\ast' &0\\ \ast'' &\begin{pmatrix} 1+\pi x_j & \pi z_j\\ 0 & 1 \end{pmatrix}  &0 \\ 0& 0 &id \end{pmatrix}.$$
Here, the diagonal block $\begin{pmatrix} 1+\pi x_j & \pi z_j\\ 0 & 1 \end{pmatrix} $ corresponds to $A(1, 2(b+b'), 1)$ with a basis $(e_3, e_2+e_4)$
and the diagonal block $\ast$ corresponds to the direct summand $(\oplus_{\lambda}H_{\lambda})\oplus H(0)$ of $M_0$.

Then the direct summand $M_1'$ of $C(L^j)=\oplus_{i\geq 1}M_i'$ is $(\pi)e_3\oplus B(e_2+e_4)\oplus M_1$.
The image of a fixed element of $F_j$ in the special fiber of the smooth integral model associated to $C(L^j)$ is then
$$\begin{pmatrix} \begin{pmatrix} 1+\pi x_j &  z_j\\ 0 & 1 \end{pmatrix}&0 &\ast'\\ 0 & id &\ast'' \\ \ast'''& \ast'''' &\ast \end{pmatrix}.$$
Here, the diagonal block $\begin{pmatrix} 1+\pi x_j &  z_j\\ 0 & 1 \end{pmatrix} $ corresponds to $(\pi)e_3\oplus B(e_2+e_4)$
and the diagonal block $id$ corresponds to the direct summand $M_1$ of $M_1'$.

Now, the image of a fixed element of $F_j$ in the orthogonal group associated to $M_1'/\pi M_1'$ is
$$\begin{pmatrix} \begin{pmatrix} 1 &  z_j^1\\ 0 & 1 \end{pmatrix}&0 \\ 0 & id  \end{pmatrix}.$$
Note that   $z_j^1$ is in $R$ such that  $z_j=z_j^1+\pi z_j^2$ as explained in the paragraph before Equation (4.2).
 The Dickson invariant of $\begin{pmatrix} \begin{pmatrix} 1 &  z_j^1\\ 0 & 1 \end{pmatrix}&0 \\ 0 & id  \end{pmatrix}$
is the same as that of  $\begin{pmatrix} 1 &  z_j^1\\ 0 & 1 \end{pmatrix}$.
Here we consider $\begin{pmatrix} 1 &  z_j^1\\ 0 & 1 \end{pmatrix}$ as an element of the orthogonal group associated to
$((\pi)e_3\oplus B(e_2+e_4))/\pi((\pi)e_3\oplus B(e_2+e_4))$.
In order to compute the Dickson invariant, we use the scheme-theoretic description of the Dickson invariant explained in Remark 4.4 of \cite{C1}.
The Dickson invariant of an orthogonal group of the quadratic space with dimension 2 is explicitly given at the end of the proof of Lemma 4.5 in \cite{C1}.
Based on this, the Dickson invariant of $\begin{pmatrix} 1 &  z_j^1\\ 0 & 1 \end{pmatrix}$ is $z_j^1/\bar{\alpha}$.
Note that $z_j^1/\bar{\alpha}$ is indeed an element of $\mathbb{Z}/2\mathbb{Z}$ by Equation (4.2).

In conclusion, $z_j^1/\bar{\alpha}$ is the image of a fixed element of $F_j$ under the map $\psi_j$.
Since $z_j^1/\bar{\alpha}$ can be either $0$ or $1$,
 $\psi_j|_{F_j}$ is surjective onto $\mathbb{Z}/2\mathbb{Z}$ and thus $\psi_j$ is surjective.\\


If $N_0$ is \textit{of type II}, then the proof of the surjectivity of $\psi_j$ is similar to that of the above case and so we skip it.\\

\item Assume that $M_0$  is \textit{of type $I^e$} and   $L_j$ is  \textit{of type $I^o$}.
We write $M_0=N_0\oplus L_j$, where $N_0$ is unimodular with odd rank so that it is  \textit{of type $I^o$}.
Then we can write $N_0=(\oplus_{\lambda'}H_{\lambda'})\oplus (a)$ and $L_j=(\oplus_{\lambda''}H_{\lambda''})\oplus (a')$ by Theorem 2.10,
where $H_{\lambda'}=H(0)=H_{\lambda''}$ and $a, a'\in A$ such that $a, a' \equiv 1$ mod 2.
Thus we write  $M_0=(\oplus_{\lambda}H_{\lambda})\oplus (a)\oplus (a')$, where $H_{\lambda}=H(0)$.
For this choice of a basis of $L^j=\bigoplus_{i \geq 0} M_i$,
the image of a fixed element of $F_j$ in the special fiber of the smooth integral model associated to $L^j$ is
$$\begin{pmatrix} id&0 &0\\ 0 &\begin{pmatrix} 1+ \pi z_j  \end{pmatrix}  &0 \\ 0& 0 &id \end{pmatrix}.$$
Here, $id$ in the $(1,1)$-block corresponds to the direct summand $(\oplus_{\lambda}H_{\lambda})\oplus (a)$ of $M_0$ and
the diagonal block $\begin{pmatrix} 1+ \pi z_j \end{pmatrix} $ corresponds to the direct summand $(a')$ of $M_0$.

Let $(e_1, e_2)$ be a basis for the direct summand $(a)\oplus (a')$ of $M_0$.
Since this is  \textit{unimodular of type $I^e$},  we can choose another basis $(e_1, e_1+e_2)$
such that the associated Gram matrix is $A(a, a+a', a)$, where $a+a'\in (2)$.
For this basis, the image of a fixed element of $F_j$ in the special fiber of the smooth integral model associated to $L^j$ is
$$\begin{pmatrix} id&0 &0\\ 0 &\begin{pmatrix} 1 & -\pi z_j\\ 0 & 1+\pi z_j \end{pmatrix}  &0 \\ 0& 0 &id \end{pmatrix}.$$
Here, the diagonal block $\begin{pmatrix}  1 & -\pi z_j\\ 0 & 1+\pi z_j \end{pmatrix} $ corresponds to $A(a, a+a', a)$ with a basis $(e_1, e_1+e_2)$
and $id$ in the $(1,1)$-block corresponds to the direct summand $(\oplus_{\lambda}H_{\lambda})\oplus (a)$ of $M_0$.

Then the direct summand $M_1'$ of $C(L^j)=\oplus_{i\geq 1}M_i'$ is $(\pi)e_1\oplus B(e_1+e_2)\oplus M_1$.
The image of a fixed element of $F_j$ in the special fiber of the smooth integral model associated to $C(L^j)$ is then
$$\begin{pmatrix} \begin{pmatrix} 1 & -z_j\\ 0 & 1+\pi z_j \end{pmatrix}&0 &0\\ 0 & id &0 \\ 0& 0 & id \end{pmatrix}.$$
Here, the diagonal block $\begin{pmatrix} 1 & -z_j\\ 0 & 1+\pi z_j \end{pmatrix} $ corresponds to $(\pi)e_1\oplus B(e_1+e_2)$
and $id$ in the $(2 \times 2)$-block corresponds to the direct summand $M_1$ of $M_1'$.

Now, the image of a fixed element of $F_j$ in the orthogonal group associated to $M_1'/\pi M_1'$ is
$$\begin{pmatrix} \begin{pmatrix} 1 &  z_j^1\\ 0 & 1 \end{pmatrix}&0 \\ 0 & id  \end{pmatrix}.$$
Note that   $z_j^1$ is in $R$ such that  $z_j=z_j^1+\pi z_j^2$ as explained in the paragraph before Equation (4.2).
 The Dickson invariant of $\begin{pmatrix} \begin{pmatrix} 1 &  z_j^1\\ 0 & 1 \end{pmatrix}&0 \\ 0 & id  \end{pmatrix}$
is the same as that of  $\begin{pmatrix} 1 &  z_j^1\\ 0 & 1 \end{pmatrix}$.
Here, we consider $\begin{pmatrix} 1 &  z_j^1\\ 0 & 1 \end{pmatrix}$ as an element of the orthogonal group associated to
$((\pi)e_1\oplus B(e_1+e_2))/\pi((\pi)e_1\oplus B(e_1+e_2))$.
Then as explained in the above case (1),
the Dickson invariant of $\begin{pmatrix} 1 &  z_j^1\\ 0 & 1 \end{pmatrix}$ is $z_j^1/\bar{\alpha}$.
Note that $z_j^1/\bar{\alpha}$ is indeed an element of $\mathbb{Z}/2\mathbb{Z}$ by Equation (4.2).

In conclusion, $z_j^1/\bar{\alpha}$ is the image of a fixed element of $F_j$ under the map $\psi_j$.
Since $z_j^1/\bar{\alpha}$ can be either $0$ or $1$,
 $\psi_j|_{F_j}$ is surjective onto $\mathbb{Z}/2\mathbb{Z}$ and thus $\psi_j$ is surjective.\\


\item Assume that both $M_0$ and $L_j$ are \textit{of type $I^o$}.
In this case, we will describe $\psi_j|_{F_j} : F_j \rightarrow \mathbb{Z}/2\mathbb{Z}$ explicitly in terms of a formal matrix.
To do that, we will first describe a morphism from $F_j$ to the special fiber of the smooth integral model associated to $L^j$
and  then to $G_j$. Recall that $G_j$ is  the special fiber of the smooth integral model associated to $C(L^j)=\bigoplus_{i \geq 0} M_i^{\prime}$.
Then we will describe a morphism from $F_j$ to the special fiber of the smooth integral model associated to $M_0'\oplus C(L^j)$
and to the special fiber of the smooth integral model associated to $C(M_0'\oplus C(L^j))$.
Finally,  we will describe a morphism from $F_j$ to a certain even orthogonal group associated to $C(M_0'\oplus C(L^j))$
 and compute the Dickson invariant of the image of an element of $F_j$ in this orthogonal group.

We write $M_0=N_0\oplus L_j$, where $N_0$ is unimodular with even rank.
Thus $N_0$ is either \textit{of type II} or \textit{of type $I^e$}.
Assume that $N_0$ is \textit{of type $I^e$}.
Then we can write $N_0=(\oplus_{\lambda'}H_{\lambda'})\oplus A(1, 2b, 1)$ and $L_j=(\oplus_{\lambda''}H_{\lambda''})\oplus (a)$ by Theorem 2.10,
where  $H_{\lambda'}=H(0)=H_{\lambda''}$, $b\in A$, and  $a (\in A) \equiv 1$ mod 2.
Thus we write  $M_0=(\oplus_{\lambda}H_{\lambda})\oplus A(1, 2b, 1)\oplus (a)$, where $H_{\lambda}=H(0)$.
For this choice of a basis of $L^j=\bigoplus_{i \geq 0} M_i$,
the image of a fixed element of $F_j$ in the special fiber of the smooth integral model associated to $L^j$ is
$$\begin{pmatrix} id&0 &0\\ 0 &\begin{pmatrix} 1+\pi z_j  \end{pmatrix}  &0 \\ 0& 0 &id \end{pmatrix}.$$
Here, $id$ in the $(1,1)$-block corresponds to the direct summand $(\oplus_{\lambda}H_{\lambda})\oplus A(1, 2b, 1)$ of $M_0$ and
the diagonal block $\begin{pmatrix} 1+\pi z_j\end{pmatrix} $ corresponds to the direct summand $(a)$ of $M_0$.

Let $(e_1, e_2, e_3)$ be a basis for the direct summand $A(1, 2b, 1)\oplus (a)$ of $M_0$.
Since this is  \textit{unimodular of type $I^o$}, we can choose another basis based on Theorem 2.10.
Namely, if we choose $(-2be_1+e_2, -ae_1+e_3,  e_2+e_3)$ as another basis, then
$A(1, 2b, 1)\oplus (a)$ becomes $A(2b(2b-1), a(a+1), a(2b-1))\oplus (a+2b)$.
Since $A(2b(2b-1), a(a+1), a(2b-1))$ is \textit{unimodular of type II}, it is isomorphic to $H(0)$ by Theorem 2.10.
Thus we can write  $M_0=(\oplus_{\lambda}H_{\lambda})\oplus H(0)\oplus (a+2b)$.
For this basis, the image of a fixed element of $F_j$ in the special fiber of the smooth integral model associated to $L^j$ is
$$\begin{pmatrix} \ast&\ast' &0\\ \ast'' &\begin{pmatrix} 1+\frac{a}{a+2b}\pi z_j \end{pmatrix}  &0 \\ 0& 0 &id \end{pmatrix}.$$
Here, the diagonal block $\begin{pmatrix} 1+\frac{a}{a+2b}\pi z_j \end{pmatrix} $ corresponds to $(a+2b)$ with a basis $e_2+e_3$
and the diagonal block $\ast$ corresponds to the direct summand $(\oplus_{\lambda}H_{\lambda})\oplus H(0)$ of $M_0$.

Then the direct summand $M_0'$ of $C(L^j)=\oplus_{i\geq 0}M_i'$ is $B(e_2+e_3)$ of rank 1.       
The image of a fixed element of  $F_j$ in the special fiber of the smooth integral model associated to $C(L^j)$ is then
$$\begin{pmatrix} \begin{pmatrix} 1+\frac{a}{a+2b}\pi z_j \end{pmatrix}&0 &\ast'\\ 0 & id &\ast'' \\ \ast'''& \ast'''' &\ast \end{pmatrix}.$$
Here, the diagonal block $\begin{pmatrix} 1+\frac{a}{a+2b}\pi z_j \end{pmatrix} $ corresponds to $M_0'=B(e_2+e_3)$ with a Gram matrix $(a+2b)$
and the diagonal block $id$ corresponds to $M_1'=M_1$.

We now describe the image of the above in  the special fiber of the smooth integral model associated to
$M_0'\oplus C(L^j)=(M_0^{\prime}\oplus M_0^{\prime})\oplus (\bigoplus_{i \geq 1} M_i^{\prime})$. 
If $(e_1', e_2')$ is a basis for $(M_0^{\prime}\oplus M_0^{\prime})$, then
we choose another basis $(e_1', e_1'+e_2')$  for $(M_0^{\prime}\oplus M_0^{\prime})$.
For this basis, based on the description of the morphism from   the smooth integral model associated to
$C(L^j)$ to  the smooth integral model associated to
$M_0'\oplus C(L^j)$ explained in Remark 4.10,
the image of a fixed element of $F_j$ in the special fiber of the smooth integral model associated to $M_0'\oplus C(L^j)$ is
$$\begin{pmatrix}  1& -\frac{a}{a+2b}\pi z_j& 0 &\ast' \\
0& 1+\frac{a}{a+2b}\pi z_j &0 &\ast'\\0& 0 & id &\ast'' \\0& \ast'''& \ast'''' &\ast \end{pmatrix}.$$
Here, the diagonal block $\begin{pmatrix} 1& -\frac{a}{a+2b}\pi z_j\\0& 1+\frac{a}{a+2b}\pi z_j \end{pmatrix} $ corresponds to $(M_0^{\prime}\oplus M_0^{\prime})$ with a basis $(e_1', e_1'+e_2')$
and the diagonal block $id$ corresponds to $M_1'=M_1$.

We now follow Step (1) with $M_0'\oplus C(L^j)=(M_0^{\prime}\oplus M_0^{\prime})\oplus (\bigoplus_{i \geq 1} M_i^{\prime})$.
Namely, $C(M_0'\oplus C(L^j))=(\pi)e_1'\oplus B(e_1'+e_2')\oplus (\bigoplus_{i \geq 1} M_i^{\prime})=((\pi) e_1'\oplus B(e_1'+e_2')\oplus M_1')\oplus (\bigoplus_{i \geq 2} M_i^{\prime})$.
Here, $((\pi) e_1'\oplus B(e_1'+e_2')\oplus M_1')$ is $\pi^1$-modular and $M_i^{\prime}$ is $\pi^i$-modular with $i\geq 2$.
Then the image of a fixed element of $F_j$ in the special fiber of the smooth integral model associated to $C(M_0'\oplus C(L^j))$ is
$$\begin{pmatrix}  1& -\frac{a}{a+2b} z_j& 0 &\ast' \\
0& 1+\frac{a}{a+2b}\pi z_j &0 &\ast'\\0& 0 & id &\ast'' \\0& \ast'''& \ast'''' &\ast \end{pmatrix}.$$
Here, $\begin{pmatrix}  1& -\frac{a}{a+2b} z_j& 0  \\
0& 1+\frac{a}{a+2b}\pi z_j &0\\0& 0 & id  \end{pmatrix}$ corresponds to $(\pi e_1'\oplus B(e_1'+e_2')\oplus M_1')$.

Now, the image of a fixed element of $F_j$ in the orthogonal group associated to $(\pi e_1\oplus B(e_1+e_2)\oplus M_1')/\pi (\pi e_1\oplus B(e_1+e_2)\oplus M_1')$ is
$$\begin{pmatrix} \begin{pmatrix} 1 &  z_j^1\\ 0 & 1 \end{pmatrix}&0 \\ 0 & id  \end{pmatrix}$$
since mod 2 reduction of $\frac{a}{a+2b}$ is 1.
Note that   $z_j^1$ is in $R$ such that  $z_j=z_j^1+\pi z_j^2$ as explained in the paragraph before Equation (4.2).
Then, as explained in Step (1), the Dickson invariant of this is $z_j^1/\bar{\alpha}$.
Note that $z_j^1/\bar{\alpha}$ is indeed an element of $\mathbb{Z}/2\mathbb{Z}$ by Equation (4.2).

In conclusion, $z_j^1/\bar{\alpha}$ is the image of a fixed element of $F_j$ under the map $\psi_j$.
Since $z_j^1/\bar{\alpha}$ can be either $0$ or $1$,
 $\psi_j|_{F_j}$ is surjective onto $\mathbb{Z}/2\mathbb{Z}$ and thus $\psi_j$ is surjective.\\


If $N_0$ is \textit{of type II}, then the proof of the surjectivity of $\psi_j$ is similar to that of the above case and so we skip it.\\

\item Assume that  $M_0$ is \textit{of type $I^o$} and $L_j$ is \textit{of type $I^e$}.
We write $M_0=N_0\oplus L_j$, where $N_0$ is unimodular with odd rank so that it is  \textit{of type $I^o$}.
Then we can write $N_0=(\oplus_{\lambda'}H_{\lambda'})\oplus (a)$ and $L_j=(\oplus_{\lambda''}H_{\lambda''})\oplus A(1, 2b, 1)$ by Theorem 2.10,
where  $H_{\lambda'}=H(0)=H_{\lambda''}$ and $a, b \in A$ such that $a \equiv 1$ mod 2.
We  write  $M_0=(\oplus_{\lambda}H_{\lambda})\oplus (a)\oplus A(1, 2b, 1)$, where $H_{\lambda}=H(0)$.
For this choice of a basis of $L^j=\bigoplus_{i \geq 0} M_i$,
the image of a fixed element of $F_j$ in the special fiber of the smooth integral model associated to $L^j$ is
$$\begin{pmatrix} id&0 &0\\ 0 &\begin{pmatrix} 1+\pi x_j & \pi z_j\\ 0 & 1 \end{pmatrix}  &0 \\ 0& 0 &id \end{pmatrix}.$$
Here, $id$ in the $(1,1)$-block corresponds to the direct summand $(\oplus_{\lambda}H_{\lambda})\oplus (a)$ of $M_0$ and
the diagonal block $\begin{pmatrix} 1+\pi x_j & \pi z_j\\ 0 & 1 \end{pmatrix} $ corresponds to the direct summand $A(1, 2b, 1)$ of $M_0$.

Let $(e_1, e_2, e_3)$ be a basis for the direct summand $(a)\oplus A(1, 2b, 1)$ of $M_0$.
Since this is  \textit{unimodular of type $I^o$}, we can choose another basis based on Theorem 2.10.
Namely, if we choose $(-2be_2+e_3, e_1-ae_2,  e_1+e_3)$ as another basis, then
$(a)\oplus A(1, 2b, 1)$ becomes $A(2b(2b-1), a(a+1), a(2b-1))\oplus (a+2b)$.
Since $A(2b(2b-1), a(a+1), a(2b-1))$ is \textit{unimodular of type II}, it is isomorphic to $H(0)$ by Theorem 2.10.
Thus we can write  $M_0=(\oplus_{\lambda}H_{\lambda})\oplus H(0)\oplus (a+2b)$.
For this basis, the image of a fixed element of $F_j$ in the special fiber of the smooth integral model associated to $L^j$ is
$$\begin{pmatrix} \ast&\ast' &0\\ \ast'' &\begin{pmatrix} 1+\frac{1}{a+2b}\pi z_j \end{pmatrix}  &0 \\ 0& 0 &id \end{pmatrix}.$$
Here, the diagonal block $\begin{pmatrix} 1+\frac{1}{a+2b}\pi z_j \end{pmatrix} $ corresponds to $(a+2b)$ with a basis $(e_1+e_3)$
and the diagonal block $\ast$ corresponds to the direct summand $(\oplus_{\lambda}H_{\lambda})\oplus H(0)$ of $M_0$.

Note that the reduction of $\frac{1}{a+2b}$ mod 2 is 1.
The rest of the proof is similar to that of Step (3) and so we skip it.\\

\end{enumerate}

So far, we have proved that $\psi_j$ is surjective.
 We now show that $\psi=\prod_{j}\psi_j$ is surjective.
 The proof is similar to the proof showing
that $\prod_{i\in\mathcal{H}} H_i \rightarrow \tilde{G}$ is a closed immersion in the last paragraph of the proof of Theorem 4.5.

We consider the morphism
\[F=\prod_{j} F_j \longrightarrow \tilde{G},\]
$(f_j)\mapsto\prod_{j}f_j$.
By observing a matrix form of an element of $F_j(R)$ for a $\kappa$-algebra $R$ as given at the beginning of the proof,
it is easy to see the following two facts:
Firstly,   $F_j$ and $F_{j'}$ commute with each other in the sense that $f_j\cdot f_{j'}=f_{j'}\cdot f_j$ for all even integers $j \neq j'$,
where $f_j\in F_j(R)$ and $ f_{j'}\in F_{j'}(R)$ for a $\kappa$-algebra $R$.
Note that $L_j$ and $L_{j'}$ (resp. $L_{j+2}$ and $L_{j'+2}$) are \textit{of type I} (resp. \textit{of type II}).
Based on this, the above morphism becomes a group homomorphism.
Secondly,
$F_j\cap F_{j'}=0$ for all $j\neq j'$.
This fact implies that the morphism $F_j\times F_{j'} \longrightarrow \tilde{G},
(f_j, f_{j'})\mapsto f_j\cdot f_{j'}$ is injective and so $F_j\times F_{j'}$ is a closed subgroup scheme of $\tilde{G}$.
A matrix form of an element of $F_j(R)$ also implies that  $(F_j\times F_{j'})\cap F_{j''}=0$ for all pairwise different three integers $j, j', j''$
and so the morphism $(F_j\times F_{j'})\times F_{j''} \longrightarrow \tilde{G},
(f_j, f_{j'}, f_{j''})\mapsto f_j\cdot f_{j'}\cdot f_{j''}$ is injective. Thus $F_j\times F_{j'}\times F_{j''}$ is a closed subgroup scheme of $\tilde{G}$.
Therefore, by repeating this argument,
 the product $F=\prod_{j} F_j$ is embedded into $\tilde{G}$ as a closed subgroup scheme.


In addition, we claim that $\psi_j|_{F_{j'}}$ is trivial for all $j<j'$.
The proof of our claim relies on the matrix interpretation of $\psi_j$.
We first notice that $j'-j\geq 4$ since $L_j$ is \textit{of type I} and $L_{j+2}$ is \textit{of type II}.
To obtain the morphism $\psi_j$, we observe the lattice $C(L^j)=\bigoplus_{i \geq 1} M_i^{\prime}$ (resp. $C(L^j)=\bigoplus_{i \geq 0} M_i^{\prime}$)
if $M_0$ is \textit{of type $I^e$} (resp. \textit{of type $I^o$}).
In either case, $L_{j'}$ is a direct summand of $M_{j'-j}$
and the morphism $\psi_j$ is attached to the Dickson invariant of the orthogonal group associated to $M_1'$.
We should mention that if $M_0$ is \textit{of type $I^o$} then we need a new hermitian lattice $M_0'\oplus C(L^j)$.
In this case, the morphism $\psi_j$ is also attached to the Dickson invariant of the orthogonal group associated to $M_1'$
as a direct summand of $M_0'\oplus C(L^j)$.
On the other hand, recall that $G_j$ is the special fiber of the smooth integral model associated to $C(L^j)$. 
Then as a formal matrix, $F_{j'}$ maps to the block of $G_j$ associated to $M_{j'-j}$.
Therefore, since $j'-j$ is at least $4$, the image of $F_{j'}$ under $\psi_j$ is zero
by observing the description of the orthogonal group associated to $M_1'$ based on Section 4.1.

We finally claim that  the morphism $\psi$ induces a surjective morphism from $F$ to $(\mathbb{Z}/2\mathbb{Z})^{\beta}$ defined over $\kappa$.
To show this, we express $F$ as $F=F_{j_1}\times \cdots \times F_{j_{\beta}}$ and $(\mathbb{Z}/2\mathbb{Z})^{\beta}$ as
$(\mathbb{Z}/2\mathbb{Z})^{\beta}=(\mathbb{Z}/2\mathbb{Z})_{j_1}\times \cdots \times (\mathbb{Z}/2\mathbb{Z})_{j_{\beta}}$,
where $j_i<j_{i'}$ if $i<i'$.
Choose an arbitrary element $(z_{j_1}, \cdots, z_{j_{\beta}})$ of $(\mathbb{Z}/2\mathbb{Z})_{j_1}\times \cdots \times (\mathbb{Z}/2\mathbb{Z})_{j_{\beta}}$
where each $z_{j_i}$ is an element of $(\mathbb{Z}/2\mathbb{Z})_{j_i}$.
We first choose $f_{j_1}\in F_{j_1}$ such that $\psi_{j_1}(f_{j_1})=z_{j_1}$.
Then choose $f_{j_2}\in F_{j_2}$ such that $\psi_{j_2}(f_{j_1}\cdot f_{j_2})=z_{j_2}$.
In this way, we choose $f_{j_t}\in F_{j_t}$ such that $\psi_{j_t}(f_{j_1}\cdot \cdots\cdot f_{j_t})=z_{j_t}$.
 Note that $\psi_{j_t}(f_{j_{t'}})=0$ for all $t<t'$.
Therefore, $\psi(f_{j_1} \cdots f_{j_{\beta}})=\prod_t\psi_{j_t}(f_{j_1} \cdots f_{j_{\beta}})=(z_{j_1}, \cdots, z_{j_{\beta}})$ and
this shows the surjectivity  of the morphism $\psi$.\\



For  the surjectivity of $\varphi \times \psi$, we recall the following criterion by Proposition 22.3 in \cite{KMRT}:
 the surjectivity of $\varphi \times \psi$ as algebraic groups
 is equivalent to the surjectivity of $\varphi \times \psi$ at the level of $\bar{\kappa}$-points since $\prod_{i:even} \mathrm{Sp}(B_i/Y_i, h_i)\times \prod_{i:odd} \mathrm{O}(A_i/Z_i, \bar{q}_i)_{\mathrm{red}} \times (\mathbb{Z}/2\mathbb{Z})^{\beta}$ is smooth.

Choose an element $(x, y)$ in the group of $\bar{\kappa}$-points of
$\prod_{i:even} \mathrm{Sp}(B_i/Y_i, h_i)\times \prod_{i:odd} \mathrm{O}(A_i/Z_i, \bar{q}_i)_{\mathrm{red}} \times (\mathbb{Z}/2\mathbb{Z})^{\beta}$
such that $x\in (\prod_{i:even} \mathrm{Sp}(B_i/Y_i, h_i)\times \prod_{i:odd} \mathrm{O}(A_i/Z_i, \bar{q}_i)_{\mathrm{red}})(\bar{\kappa})$
and $y\in (\mathbb{Z}/2\mathbb{Z})^{\beta}(\bar{\kappa})$.
Then there is an element $a\in \tilde{G}(\bar{\kappa})$ such that $\varphi(a)=x$ since $\varphi$ is surjective by Theorem 4.5.
We choose an element  $b\in F(\bar{\kappa})$ such that $\psi(ab)=y$.
On the other hand, $\varphi$ vanishes on  $F$ since the morphism $\varphi_i$ vanishes on  $F_j$ for all $i, j$.
Thus  $\varphi(b)=0$
and $(\varphi \times \psi)(ab)=(x, y)$.
This completes the proof.
\end{proof}

     \subsection{The maximal reductive quotient}



     We finally have the structure theorem for the algebraic group $\tilde{G}$.

     \begin{Thm}
     The morphism $$\varphi  \times \psi : \tilde{G} \longrightarrow   \prod_{i:even} \mathrm{Sp}(B_i/Y_i, h_i)\times \prod_{i:odd} \mathrm{O}(A_i/Z_i, \bar{q}_i)_{\mathrm{red}} \times (\mathbb{Z}/2\mathbb{Z})^{\beta}$$
      is surjective and the kernel is unipotent and connected.
     Consequently, $$\prod_{i:even} \mathrm{Sp}(B_i/Y_i, h_i)\times \prod_{i:odd} \mathrm{O}(A_i/Z_i, \bar{q}_i)_{\mathrm{red}} \times (\mathbb{Z}/2\mathbb{Z})^{\beta}$$ is the maximal reductive quotient.
  Here, $\mathrm{Sp}(B_i/Y_i, h_i)$ and $\mathrm{O}(A_i/Z_i, \bar{q}_i)_{\mathrm{red}}$ are explained in Section 4.1 (especially Remark 4.7)
 and $\beta$ is defined in Lemma 4.6.
     \end{Thm}

     \begin{proof}
     We only need to prove that the kernel is unipotent and connected.
     The kernel of $\varphi$ is a closed subgroup scheme of the unipotent group $\tilde{M}^+$ which is defined in Lemma A.2 of Appendix A and so
     it suffices to show that the kernel of $\varphi \times \psi$ is connected.
     Equivalently, it suffices to show that the kernel of the restricted morphism $\psi|_{\mathrm{Ker~}\varphi}$ is connected.
     From Lemma 4.6,  the number of connected components of $\mathrm{Ker~}\varphi $ is $2^\beta$.
 Since $\varphi|_{F}=0$ so that $F=\prod_{j} F_j\subset \mathrm{Ker~}\varphi$, the restricted morphism $\psi|_{\mathrm{Ker~}\varphi}$ is surjective onto $(\mathbb{Z}/2\mathbb{Z})^{\beta}$.
 We complete the proof by counting the number of connected components.
 \end{proof}

\section{Comparison of volume forms and final formulas}

This section is based on Section 7 of \cite{GY} and Section 5 of \cite{C1}.
Let $H$ be the $F$-vector space of hermitian forms on $V=L\otimes_AF$.
Let $M'=\mathrm{End}_{B}(L)$  
and let $H'=\{\textit{f : f is a hermitian form on $L$}\}$. 
Regarding  $\mathrm{End}_EV$ and $H$ as varieties over $F$,
let $\omega_M$ and $\omega_H$ be nonzero,  translation-invariant forms on   $\mathrm{End}_EV$ and $H$, respectively, 
with normalization
$$\int_{M'}|\omega_M|=1 \mathrm{~and~}  \int_{H'}|\omega_H|=1.$$
Let $M^{\ast}=\mathrm{Res}_{E/F}\mathrm{GL}_E(V)$.
Define a map $\rho : M^{\ast} \rightarrow H$ by $\rho(m)=h\circ m$.
Here $h\circ m$ is the hermitian form $(v, w)\mapsto h(mv, mw)$.
Then the inverse of $h$ under $\rho$ is $G$, which is the unitary group associated to the hermitian space $(V, h)$.
It is also the generic fiber of $\underline{G}'$.
Put $\omega^{\mathrm{ld}}=\omega_M/\rho^{\ast}\omega_H$.
For a detailed explanation of what $\omega_M/\rho^{\ast}\omega_H$ means, we refer to Section 3.2 of \cite{GY}.


We choose two forms $\omega^{\prime}_M$ and $\omega^{\prime}_H$ as generators for the spaces of the top degree forms on $\underline{M}'$,
 which is identified with the Lie algebra of $\underline{M}^{\ast}$,
and $\underline{H}'$,  which is identified with the tangent space to $\underline{H}$ at $h$, respectively.
Here $\underline{M}'$ is defined in Remark 3.1 and $\underline{H}'$ is defined
in the paragraph following the matrix description of an element of $\underline{H}(R)$ for a flat $A$-algebra $R$ in Section 3.3.
They are nonzero  translation-invariant forms  on  $\mathrm{End}_EV$ and $H$, respectively,
with normalization
$$\int_{\underline{M}(A)}|\omega^{\prime}_M|=1 \mathrm{~and~}  \int_{\underline{H}(A)}|\omega^{\prime}_H|=1. $$

By Theorem 3.6, we have an exact sequence of locally free sheaves on $\underline{M}^{\ast}$:
\[ 0\longrightarrow \rho^{\ast}\Omega_{\underline{H}/A} \longrightarrow \Omega_{\underline{M}^{\ast}/A}
\longrightarrow \Omega_{\underline{M}^{\ast}/\underline{H}} \longrightarrow 0. \]
Put $\omega^{\mathrm{can}}=\omega^{\prime}_M/\rho^{\ast}\omega^{\prime}_H$.
For a detailed explanation of what $\omega^{\prime}_M/\rho^{\ast}\omega^{\prime}_H$ means, we refer to Section 3.2 of \cite{GY}.
It follows that 
$\omega^{\mathrm{can}}$ is a differential of top degree on $\underline{G}$, which is  invariant under the generic fiber of $\underline{G}$,
and which has nonzero reduction on the special fiber.

\begin{Lem}
  $$|\omega_M|=|2|^{N_M}|\omega_M^{\prime}|, \ \ \ \  N_M=\sum_{\textit{$i$:even and $L_i$:type I}}(2n_i-1)+\sum_{i<j}(j-i)\cdot n_i\cdot n_j,$$
  $$|\omega_H|=|2|^{N_H}|\omega_H^{\prime}|,  \ \ \ \ N_H=\sum_{\textit{$i$:even and $L_i$:type I}}(n_i-1)+\sum_{i<j}j\cdot n_i\cdot n_j+\sum_{\textit{i:even}} \frac{i+2}{2}\cdot n_i
  +\sum_{\textit{i:odd}} \frac{i+1}{2}\cdot n_i+\sum_i d_i,$$
  $$|\omega^{\mathrm{ld}}|=|2|^{N_M-N_H}|\omega^{\mathrm{can}}|.$$
  Here, $d_i=i\cdot n_i\cdot (n_i-1)/2$.
\end{Lem}
\begin{proof}
Note that both $\omega_M$ and $\omega^{\prime}_M$ are volume forms on $\mathrm{End}_EV$ with different normalizations,
so that they differ by a scalar.
 The `difference' between the Haar measures associated to these volume forms can be detected at the level of $F$-points of $\mathrm{End}_EV$,  since $\mathrm{End}_EV$ is an affine space.

Since  $\underline{M}(A)=1+\underline{M}'(A)$, where $\underline{M}'$ is defined in Remark 3.1,
 we  have the identity $\int_{\underline{M}'(A)}|\omega^{\prime}_M|=1$.
Note that $\underline{M}'(A)$ is a finitely generated free $A$-submodule of $M'$ whose rank is the same as that of $M'$.
Thus $N_M$ is the `difference' between these two modules $M'$ and $\underline{M}'(A)$.
More precisely, $N_M$ is the length of the finitely generated torsion $A$-module $M'/\underline{M}'(A)$.
Note that $2$ is a uniformizer of $A$.

Similarly, $N_H$ is the length of the finitely generated torsion $A$-module $H'/\underline{H}'(A)$.
Here, $\underline{H}'$ is defined
in the paragraph following the matrix description of an element of $\underline{H}(R)$ for a flat $A$-algebra $R$ in Section 3.3.

Then the above formula for $N_M$ (resp. $N_H$) can be read off from the matrix interpretation for $\underline{M}(A)$ (resp. $\underline{H}(A)$)
given in Sections 3.1 and 3.2 (resp. Section 3.3).
\end{proof}

Let $f$ be the cardinality of $\kappa$.
The local density is defined as
\[\beta_L= \frac{1}{[G:G^{\circ}]}\cdot \lim_{N\rightarrow \infty} f^{-N~dim G}\#\underline{G}'(A/\pi^N A). \]
Here, $\underline{G}'$ is the naive integral model described at the beginning of Section 3 and $G$ is the generic fiber of $\underline{G}'$
and $G^{\circ}$ is the identity component of $G$.
In our case, $G$ is the unitary group $\mathrm{U}(V, h)$, where $V=L\otimes_AF$.
Since $\mathrm{U}(V, h)$ is connected, $G^{\circ}$ is the same as $G$ so that $[G:G^{\circ}]=1$.

Then based on Lemma 3.4 and Section 3.9 of \cite{GY}, we finally have the following local density formula.

\begin{Thm}
Let $f$ be the cardinality of $\kappa$.
The local density of ($L,h$) is
$$\beta_L=f^N \cdot f^{-\mathrm{dim~} \mathrm{U}(V, h)} \#\tilde{G}(\kappa),$$

where $$N=N_H-N_M=\sum_{i<j}i\cdot n_i\cdot n_j+\sum_{\textit{i:even}} \frac{i+2}{2}\cdot n_i
  +\sum_{\textit{i:odd}} \frac{i+1}{2}\cdot n_i+\sum_i d_i-\sum_{\textit{$i$:even and $L_i$:type I}}n_i.$$
Here, $\#\tilde{G}(\kappa)$ can be computed explicitly based on Remark 5.3.(1) below and Theorem 4.12. \qed
\end{Thm}

So as not to inconvenience the reader we  repeat the following remark from Remark 5.3 in \cite{C1}.

\begin{Rmk}(Remark 5.3 in \cite{C1})

\begin{enumerate}
\item In the above local density formula, $\#\tilde{G}(\kappa)$ is computed as follows.
We denote by $R_u\tilde{G}$  the unipotent radical of $\tilde{G}$ so that the  maximal reductive quotient of $\tilde{G}$ is $\tilde{G}/R_u\tilde{G}$.
That is, there is the following exact sequence of group schemes over $\kappa$:
\[1 \longrightarrow R_u\tilde{G} \longrightarrow \tilde{G} \longrightarrow \tilde{G}/R_u\tilde{G} \longrightarrow 1. \]
Furthermore, the following sequence of groups
\[1 \longrightarrow R_u\tilde{G}(\kappa) \longrightarrow \tilde{G}(\kappa) \longrightarrow (\tilde{G}/R_u\tilde{G})(\kappa) \longrightarrow 1 \]
is also exact by Lemma A.1.
Using Lemma A.1,  one can see  that $\# R_u\tilde{G}(\kappa)=f^m$, where $m$ is the dimension of $R_u\tilde{G}$.
Notice that the dimension of $R_u\tilde{G}$ can be computed explicitly based on Theorem 4.12,
since the dimension of $\tilde{G}$ is $n^2$ with $n=\mathrm{rank}_{B}L$.
In addition, the orders of  orthogonal and symplectic groups defined over a finite field are well known. Thus, one can compute $\# (\tilde{G}/R_u\tilde{G})(\kappa)$ explicitly based on Theorem 4.12.
Finally, the order of the group $\tilde{G}(\kappa)$ is identified as follows:
\[\#\tilde{G}(\kappa)=\# R_u\tilde{G}(\kappa)\cdot \# (\tilde{G}/R_u\tilde{G})(\kappa).\]

\item As  in Remark 7.4 of \cite{GY}, although we have assumed that $n_i=0$ for $i<0$,
it is easy to check that the formula in the preceding theorem remains true without this assumption.
\end{enumerate}
\end{Rmk}

\appendix
\section{The proof of  Lemma 4.6} \label{App:AppendixA}


The proof of Lemma 4.6 is based on  Proposition 6.3.1 in \cite{GY}. 
We first state a theorem of Lazard which is repeatedly used in this paper.
Let $U$ be a group scheme of finite type over $\kappa$ which is isomorphic to an affine space as an algebraic variety.
Then $U$ is connected smooth unipotent group
(cf.  IV, $\S$ 4, Theorem 4.1 and IV, $\S$ 2, Corollary 3.9 in \cite{DG}).

For preparation, we  state several lemmas.

\begin{Lem} (Lemma 6.3.3. in \cite{GY})

Let $1 \rightarrow X\rightarrow Y\rightarrow Z\rightarrow 1$ be an exact sequence of group schemes that are locally of finite type over $\kappa$,
where  $\kappa$ is a perfect field. Suppose that $X$ is smooth, connected, and unipotent.
Then $1 \rightarrow X(R)\rightarrow Y(R)\rightarrow Z(R)\rightarrow 1$ is exact for any $\kappa$-algebra $R$. \qed
\end{Lem}

Let $\tilde{M}$ be the special fiber of $\underline{M}^{\ast}$ and let $R$ be a $\kappa$-algebra.
Recall that we have described an element   and
the multiplication of elements of $\underline{M}(R)$  in Section 3.2.
Based on these,
an element of $\tilde{M}(R)$ is
$$m= \begin{pmatrix} \pi^{max\{0,j-i\}}m_{i,j} \end{pmatrix}.$$
Here, if $i$ is even and $L_i$ is \textit{of type} $\textit{I}^o$ (resp. \textit{of type} $\textit{I}^e$), then
 $$m_{i,i}=\begin{pmatrix} s_i&\pi y_i\\ \pi v_i&1+\pi z_i \end{pmatrix} \textit{(resp.
$\begin{pmatrix} s_i&r_i&\pi t_i\\ \pi y_i&1+\pi x_i&\pi z_i\\ v_i&u_i&1+\pi w_i \end{pmatrix}$)},$$
where
 $s_i\in M_{(n_i-1)\times (n_i-1)}(B\otimes_AR)$ (resp.  $s_i\in M_{(n_i-2)\times (n_i-2)}(B\otimes_AR)$), etc.,
 and $s_i$ mod $\pi\otimes 1$ is invertible.
For the remaining $m_{i,j}$'s except for the cases explained above, $m_{i,j}\in M_{n_i\times n_j}(B\otimes_AR)$
and $m_{i,i}$  mod $\pi\otimes 1$ is invertible.

Let
  \[
\tilde{M}_i=  \left \{
  \begin{array}{l l}
  \mathrm{GL}_{\kappa}(B_i/Y_i) & \quad  \textit{if $i$ is even};\\
  \mathrm{GL}_{\kappa}(A_i/X_i) & \quad \textit{if $i$ is odd}.
    \end{array} \right.
\]
Let $s_i=m_{i,i}$ if $L_i$ is \textit{of type II} in the above description of $\tilde{M}(R)$.
Then $s_i$ mod $\pi\otimes 1$ is an element of $\tilde{M}_i(R)$.
 Therefore, we have a surjective morphism of algebraic groups
 $$r : \tilde{M} \longrightarrow \prod\tilde{M}_i,$$
 defined over $\kappa$. We now have the following easy lemma:

 \begin{Lem}
The kernel of $r$ is the unipotent radical $\tilde{M}^+$ of $\tilde{M}$, and $\prod\tilde{M}_i$ is the maximal reductive quotient of $\tilde{M}$.\\
\end{Lem}
\begin{proof}
Since $\prod\tilde{M}_i$ is a reductive group, we only have to show that the kernel of $r$ is a connected smooth unipotent group.
Let $R$ be a $\kappa$-algebra.
By the description of the morphism $r$ in terms of matrices explained above, an element of the kernel of $r$ is
$$m= \begin{pmatrix} \pi^{max\{0,j-i\}}m_{i,j} \end{pmatrix}$$
satisfying the following:\\
If $i$ is even and $L_i$ is \textit{of type} $\textit{I}^o$ (resp. \textit{of type} $\textit{I}^e$), then
 $$m_{i,i}=\begin{pmatrix} id+\pi s_i'&\pi y_i\\ \pi v_i&1+\pi z_i \end{pmatrix} \textit{(resp.
$\begin{pmatrix} id+\pi s_i'&r_i&\pi t_i\\ \pi y_i&1+\pi x_i&\pi z_i\\ v_i&u_i&1+\pi w_i \end{pmatrix}$)},$$
where
 $id+\pi\otimes 1 \cdot s_i'\in M_{(n_i-1)\times (n_i-1)}(B\otimes_AR)$ (resp.  $id+\pi\otimes 1\cdot  s_i'\in M_{(n_i-2)\times (n_i-2)}(B\otimes_AR)$), etc.,
 such that  $s_i'$ has entries in $R\subset B\otimes_AR$.
For the remaining $m_{i,j}$'s except for the cases explained above, $m_{i,j}\in M_{n_i\times n_j}(B\otimes_AR)$ and
$m_{i,i}=id+\pi\otimes 1\cdot  m_{i,i}'$ such that $m_{i,i}'$  has entries in $R\subset B\otimes_AR$.
Note that there are no equations among variables given above.
Thus the kernel of $r$ is isomorphic to an affine space  as an algebraic variety over $\kappa$.
Therefore, it is a connected smooth unipotent group by a theorem of Lazard which is stated at the beginning of Appendix A.
\end{proof}
Recall that we have defined the morphism $\varphi$ in Section 4.1.
 The  morphism $\varphi$ extends to     an obvious morphism
 $$\tilde{\varphi} : \tilde{M} \longrightarrow  \prod_{i:even}\mathrm{GL}_{\kappa}(B_i/Y_i)\times \prod_{i:odd}\mathrm{GL}_{\kappa}(A_i/Z_i) $$
 such that $\tilde{\varphi}|_{\tilde{G}}=\varphi $.
 Note that $Y_i\otimes_AR$ and $Z_i\otimes_AR$ are preserved by an element of $\underline{M}(R)$ for a flat $A$-algebra $R$ (cf. Lemma 4.2).
 By using this,
the construction of $\tilde{\varphi}$ is similar to Theorems 4.3 and 4.4  and thus we skip it.
Let $R$ be a $\kappa$-algebra.
Based on the description of the morphism $\varphi_i$ explained in Section 4.1,
$\mathrm{Ker~}\tilde{\varphi}(R)$ is the subgroup of  $\tilde{M}(R)$ defined by the following conditions:

\begin{itemize}
\item[a)]  If $i$ is even and $L_i$ is \textit{of type I},   $s_i=\mathrm{id}$ mod $\pi \otimes 1$.
\item[b)]  If $i$ is even and $L_i$ is \textit{of type II},  $m_{i,i}=\mathrm{id}$ mod $\pi \otimes 1$.
\item[c)]  If $i$ is odd,  $m_{i,i}=\mathrm{id}$  mod $\pi \otimes 1$ and $\delta_{i-1}e_{i-1}\cdot m_{i-1, i}+\delta_{i+1}e_{i+1}\cdot m_{i+1, i}=0$  mod $\pi \otimes 1$.
Here, $
\delta_{j} = \left\{
  \begin{array}{l l}
  1    & \quad  \textit{if $L_j$ is \textit{of type I}};\\
  0    &   \quad  \textit{if $L_j$ is \textit{of type II}},
    \end{array} \right.
$ \\
and $e_{j}=(0,\cdots, 0, 1)$ (resp. $e_j=(0,\cdots, 0, 1, 0)$) of size $1\times n_{j}$
if $L_{j}$ is \textit{of type} $\textit{I}^o$ (resp. \textit{of type} $\textit{I}^e$).\\
\end{itemize}

 It is obvious that $\mathrm{Ker~}\tilde{\varphi}$ is a closed subgroup scheme of $\tilde{M}^+$ and  is smooth and unipotent
 since it is isomorphic to an affine space  as an algebraic variety over $\kappa$.\\

Recall from Remark 3.1 that we  defined the functor $\underline{M}^{\prime}$ such that $(1+\underline{M}^{\prime})(R)=\underline{M}(R)$
inside $\mathrm{End}_{B\otimes_AR}(L \otimes_A R)$ for a flat $A$-algebra $R$.
Thus there is an isomorphism of set valued functors
$$1+ :  \underline{M}^{\prime} \longrightarrow \underline{M}, ~~~ m\mapsto 1+m,$$
where $m\in \underline{M}^{\prime}(R)$ for a flat $A$-algebra $R$.
We define a new operation $\star$ on $\underline{M}^{\prime}(R)$ such that $x\star y=x+y+xy$ for a flat $A$-algebra $R$.
Since $\underline{M}^{\prime}(R)$ is closed under addition and multiplication, it is also closed under the new operation $\star$.
Moreover, it has  $0$ as an identity element  with respect to $\star$.
Thus $\underline{M}^{\prime}$ may and shall be considered as a scheme of monoids with $\star$.
 We claim that the above morphism $1+$ is an isomorphism of monoid schemes.
 Namely, we claim  the following commutative diagram of schemes:
 \[\xymatrixcolsep{5pc}\xymatrix{
\underline{M}^{\prime}\times \underline{M}^{\prime} \ar[d]^{\star} \ar[r]^{(1+)\times (1+)} &\underline{M}\times
\underline{M}\ar[d]^{multiplication}\\
\underline{M}^{\prime} \ar[r]^{1+} &\underline{M}}\]
Since all schemes are irreducible and smooth, it suffices to check the commutativity of the diagram at the level of flat $A$-points
as explained in the third paragraph from below in Remark 3.2, and it is obvious.

Since $\underline{M}^{\ast}$ is an open subscheme of $\underline{M}$,  $(1+)^{-1}(\underline{M}^{\ast})$ is an open subscheme of $\underline{M}^{\prime}$.
The composite of the following three morphisms
 \[\xymatrixcolsep{5pc}\xymatrix{
(1+)^{-1}(\underline{M}^{\ast})  \ar[r]^{(1+)} &\underline{M}^{\ast} \ar[r]^{inverse} &\underline{M}^{\ast}\ar[r]^{(1+)^{-1}}&
(1+)^{-1}(\underline{M}^{\ast})}\]
defines the inverse morphism on the scheme of monoids $(1+)^{-1}(\underline{M}^{\ast})$ with respect to the operation $\star$.
Thus we can see that $(1+)^{-1}(\underline{M}^{\ast})$ is a  group scheme with respect to $\star$ and the morphism $1+$ is
an isomorphism of group schemes between
$(1+)^{-1}(\underline{M}^{\ast})$ and $\underline{M}^{\ast}$.

Let $R$ be a $\kappa$-algebra.
Since the morphism $1+$ is an isomorphism of monoid schemes between $\underline{M}^{\prime}$ and $\underline{M}$,
we can write each element of $\underline{M}(R)$ as $1+x$ with $x \in \underline{M}^{\prime}(R)$. 
Here,  $1+x$ means the image of $x$ under the morphism $1+$ at the level of $R$-points.
Note that  $\underline{M}^{\prime}(R)$ is a $B\otimes_AR$-algebra for any $A$-algebra $R$ with respect to the original multiplication on it,
 not  the operation $\star$.
In particular,  $\underline{M}^{\prime}(R)$ is a $(B/2B)\otimes_AR$-algebra for any $\kappa$-algebra $R$.
Therefore, we consider the subfunctor $\underline{\pi M^{\prime}} : R \mapsto (\pi\otimes 1) \underline{M}^{\prime}(R)$ of $\underline{M}^{\prime}\otimes\kappa$
and the subfunctor $\tilde{M}^1:R\mapsto 1+\underline{\pi M^{\prime}}(R)$ of $\mathrm{Ker~}\tilde{\varphi}$.
Here, by $1+\underline{\pi M^{\prime}}(R)$, we mean the image of $\underline{\pi M^{\prime}}(R)$
inside $\underline{M}(R) (=\tilde{M}(R))$ under the morphism $1+$ at the level of $R$-points.
That   $1+\underline{\pi M^{\prime}}(R)$ is contained in $\mathrm{Ker~}\tilde{\varphi}(R)$ can easily be checked
 by observing the construction of $\tilde{\varphi}$.
The multiplication on $\tilde{M}^1$  is as follows:
For given two elements $1+\pi x$ and $1+\pi y$ in $\tilde{M}^1(R)$,
based on the above commutative diagram, the product of $1+\pi x$ and $1+\pi y$  is
$$(1+\pi x)\cdot(1+\pi y)=1+\pi x\star \pi y=1+(\pi (x+y)+\pi^2(xy))=1+\pi (x+y).$$
Here, $\pi$ stands for $\pi\otimes 1 \in B\otimes_AR$.
Then we have the following  lemma.
\begin{Lem}
(i) The functor $\tilde{M}^1$ is representable by a smooth, connected, unipotent group scheme over $\kappa$.
Moreover, $\tilde{M}^1$ is a closed normal subgroup of $\mathrm{Ker~}\tilde{\varphi} $.

(ii) The quotient group scheme $\mathrm{Ker~}\tilde{\varphi}/\tilde{M}^1$ represents the functor
$$R\mapsto \mathrm{Ker~}\tilde{\varphi}(R)/\tilde{M}^1(R)$$
by Lemma A.1 and is smooth, connected, and unipotent.\\  
\end{Lem}
\begin{proof}
Let $R$ be a $\kappa$-algebra. In the proof, $\pi$ stands for $\pi\otimes 1 \in B\otimes_AR$.
To show that  $\tilde{M}^1(R)$ is a  subgroup of $\mathrm{Ker~}\tilde{\varphi}(R)$,
it suffices to show that  the inverse $1+x'$ of $1+\pi x$ in $\mathrm{Ker~}\tilde{\varphi}(R)$
is contained in $\tilde{M}^1(R)$.
From the identity
$$(1+x')(1+\pi x)=1+x'\star\pi x=1+(x'+\pi x+\pi x'x)=1+0,$$
we see that $x'$ is an element of $\underline{\pi M}'(R)$ so that $1+x'$ is an element of $\tilde{M}^1(R)$,
since $\underline{M}^{\prime}(R)$ is closed under multiplication and addition which implies $ x+ x'x \in \underline{M}'(R)$.

Then the first sentence of (i) follows by a theorem of Lazard which is stated at the beginning of Appendix A
since $\tilde{M}^1$ is isomorphic to an affine space of dimension $n^2$  as an algebraic variety over $\kappa$.

To show that $\tilde{M}^1(R)$ is a  normal subgroup of $\mathrm{Ker~}\tilde{\varphi}(R)$,
we choose an element $1+\pi x\in \tilde{M}^1(R)$ and $1+m\in \mathrm{Ker~}\tilde{\varphi}(R)$ with $m\in \underline{M}'(R)$.
Let $1+m'$ be the inverse of $1+m$ so that $(1+m')(1+m)=1$. Then we have the following identity:
\[(1+m')(1+\pi x)(1+m)=1+m'\star \pi x\star m=1+\pi(x+m'x+xm+m'xm).\]
Since $\underline{M}^{\prime}(R)$ is closed under multiplication and addition,
$x+m'x+xm+m'xm\in \underline{M}'(R)$ so that  $(1+m')(1+\pi x)(1+m)\in \tilde{M}^1(R)$.

For (ii), smoothness and connectedness are stable under quotienting by algebraic groups (Proposition 22.4 in \cite{KMRT})
and a quotient of a unipotent group is also a unipotent group by part (a) of the first corollary in Section 8.3 in \cite{W}.
\end{proof}

 This paragraph is a reproduction of 6.3.6 in \cite{GY}.
Recall that there is a closed immersion $\tilde{G}\rightarrow \tilde{M}$.
Notice that $\mathrm{Ker~}\varphi$ is the kernel of the composition $\tilde{G}\rightarrow \tilde{M} \rightarrow \tilde{M}/ \mathrm{Ker~}\tilde{\varphi}$.
We define $\tilde{G}^1$ as the kernel of the composition
\[ \tilde{G}\rightarrow \tilde{M} \rightarrow \tilde{M}/ \tilde{M}^1.\]
 Then $\tilde{G}^1$ is the kernel of the morphism $\mathrm{Ker~}\varphi\rightarrow \mathrm{Ker~}\tilde{\varphi}/\tilde{M}^1$
 and, hence, is a closed normal subgroup of $\mathrm{Ker~}\varphi$.
 The induced morphism $\mathrm{Ker~}\varphi/\tilde{G}^1\rightarrow \mathrm{Ker~}\tilde{\varphi}/\tilde{M}^1$ is a monomorphism, and thus
 $\mathrm{Ker~}\varphi/\tilde{G}^1$ is a closed subgroup scheme of $ \mathrm{Ker~}\tilde{\varphi}/\tilde{M}^1$
 by (Exp. $\mathrm{VI_B},$ Corollary 1.4.2 in \cite{SGA3}).

 \begin{Thm}
 $\tilde{G}^1$ is connected, smooth, and unipotent.
 Furthermore,   the underlying algebraic variety of $\tilde{G}^1$ over $\kappa$ is  an affine space of dimension
$$\sum_{i<j}n_in_j+\sum_{i:\mathrm{odd}}\frac{n_i^2+n_i}{2}+\sum_{i:\mathrm{even}}\frac{n_i^2-n_i}{2}+\#\{i:i\mathrm{~is~even~and~}L_i\mathrm{~is~\textit{of type I}}\}.$$
 \end{Thm}

 \begin{proof}
We prove this theorem by writing out a set of equations completely defining $\tilde{G}^1$
(after all there are so many different sets of equations defining $\tilde{G}^1$).
Let $R$ be a $\kappa$-algebra.
As explained in Remark 3.3.(2), we consider the given hermitian form $h$ as an element of $\underline{H}(R)$
and write it as a formal matrix $h=\begin{pmatrix} \pi^{i}\cdot h_i\end{pmatrix}$ with $(\pi^{i}\cdot h_i)$
for the $(i,i)$-block and $0$ for the remaining blocks. 
We also write $h$ as $(f_{i, j}, a_i\cdots f_i)$.
Recall that the notation $(f_{i, j}, a_i\cdots f_i)$ is defined and explained in Section 3.3
and explicit values of $(f_{i, j}, a_i\cdots f_i)$ for the  $h$ are given in Remark 3.3.(2).

We choose an element $m=(m_{i,j}, s_i\cdots w_i)\in (\mathrm{Ker~}\tilde{\varphi})(R)$ with a formal matrix interpretation
$m= \begin{pmatrix} \pi^{max\{0,j-i\}}m_{i,j} \end{pmatrix}$,
where the notation  $(m_{i,j}, s_i\cdots w_i)$ is explained in Section 3.2.
Then $h\circ m$ is an element of $\underline{H}(R)$ and
$(\mathrm{Ker~}\varphi)(R)$ is the set of $m$ such that  $h\circ m=(f_{i, j}, a_i\cdots f_i)$.
The action $h\circ m$ is explicitly described in Remark 3.5. Based on this,
we need to write the matrix product $h\circ m=\sigma({}^tm)\cdot h\cdot m$ formally.
To do that, we write each block of $\sigma({}^tm)\cdot h\cdot m$ as follows:



The diagonal $(i,i)$-block of  the formal matrix product $\sigma({}^tm)\cdot h\cdot m$ is the following:
\begin{multline}
\pi^i\left(\sigma({}^tm_{i,i})h_im_{i,i}+\sigma(\pi)\cdot\sigma({}^tm_{i-1, i})h_{i-1}m_{i-1, i}+\pi\cdot\sigma({}^tm_{i+1, i})h_{i+1}m_{i+1, i}\right)+\\
\pi^i\left((\sigma\pi)^2\cdot\sigma({}^tm_{i-2, i})h_{i-2}m_{i-2, i}+
 \pi^2\cdot\sigma({}^tm_{i+2, i})h_{i+2}m_{i+2, i}\right),
\end{multline}
where $0\leq i < N$. 
\\
The $(i,j)$-block of the formal matrix product  $\sigma({}^tm)\cdot h\cdot m$, where $i<j$, is the following:
\begin{equation}
\pi^j\left(\sum_{i\leq k \leq j} \sigma({}^tm_{k,i})h_km_{k,j}+\sigma(\pi)\cdot\sigma({}^tm_{i-1,i})h_{i-1}m_{i-1,j}+\pi\cdot\sigma({}^tm_{j+1,i})h_{j+1}m_{j+1,j}\right),
\end{equation}
where $0\leq i, j < N$. 


Before studying $\tilde{G}^1$,
we describe the conditions for an element $m\in \tilde{M}(R)$ as above to belong to the subgroup $\tilde{M}^1(R)$.
\begin{enumerate}
\item $m_{i,j}=\pi m_{i,j}^{\prime} \mathrm{~if~} i\neq j$;
\item $m_{i,i}=\mathrm{id}+\pi m_{i,i}^{\prime}$ if  $L_i$ is \textit{of type II};
\item $m_{i,i}= \begin{pmatrix} s_i&\pi y_i\\ \pi v_i&1+\pi z_i \end{pmatrix}=\begin{pmatrix} \mathrm{id}+\pi s_i^{\prime}&\pi^2 y_i^{\prime}\\ \pi^2 v_i^{\prime}&1+\pi^2 z_i^{\prime} \end{pmatrix}$ if $i$ is even and $L_i$ is \textit{of type} $\textit{I}^o$;
\item $ m_{i,i}=\begin{pmatrix} s_i&r_i&\pi t_i\\ \pi y_i&1+\pi x_i&\pi z_i\\ v_i&u_i&1+\pi w_i \end{pmatrix}=
 \begin{pmatrix} \mathrm{id}+\pi s_i^{\prime}&\pi r_i^{\prime}&\pi^2t_i^{\prime}\\ \pi^2y_i^{\prime}&1+\pi^2x_i^{\prime}&\pi^2z_i^{\prime}\\ \pi v_i^{\prime}&\pi u_i^{\prime}&1+\pi^2w_i^{\prime} \end{pmatrix}$
 if $i$ is even and $L_i$ is \textit{of type} $\textit{I}^e$.
\end{enumerate}
Here, all matrices having ${}^{\prime}$ in the superscription are considered as matrices with entries in $R$.
When $i$ is even and $L_i$ is \textit{of type} $\textit{I}$, we formally write $m_{i,i}=\mathrm{id}+\pi m_{i,i}^{\prime}$. 
Then $\tilde{G}^1(R)$ is the set of $m\in \tilde{M}^1(R)$ such that  $h\circ m=h=(f_{i, j}, a_i\cdots f_i)$.
Since $h\circ m$ is an element of $\underline{H}(R)$,  we can write $h\circ m$ as $(f_{i, j}', a_i'\cdots f_i')$.
In what follows,
we will write $(f_{i, j}', a_i'\cdots f_i')$ in terms of $h=(f_{i, j}, a_i\cdots f_i)$ and $m$, and will compare $(f_{i, j}', a_i'\cdots f_i')$ with
$(f_{i, j}, a_i\cdots f_i)$, in order to obtain a set of equations defining $\tilde{G}^1$.\\

If we put all these (1)-(4) into  (A.2), then we obtain
$$\pi^j\left(\sigma(1+\pi\cdot {}^tm_{i,i}')h_i\pi m_{i,j}'+\sigma(\pi\cdot {}^tm_{j,i}')h_j(1+\pi m_{j,j}')\right).$$
Therefore, $$f_{i,j}'=\left(\sigma(1+\pi\cdot {}^tm_{i,i}')h_i\pi m_{i,j}'+\sigma(\pi\cdot {}^tm_{j,i}')h_j(1+\pi m_{j,j}')\right),$$
where this equation is considered in $B\otimes_AR$ and $\pi$ stands for $\pi\otimes 1 \in B\otimes_AR$.
Thus each term having $\pi^2$ as a factor is $0$ and we have
\begin{equation}
f_{i,j}'=h_i\pi m_{i,j}'+\sigma(\pi\cdot {}^tm_{j,i}')h_j, \textit{where $i<j$}.
\end{equation}
This equation is of the form $f_{i,j}'=X+\pi Y$ since it is an equation in $B\otimes_AR$.
By letting  $f_{i,j}'=f_{i,j}=0$,
 we obtain \begin{equation}\bar{h}_i m_{i,j}'+{}^tm_{j,i}'\bar{h}_j=0, \textit{where $i<j$},\end{equation}
 where $\bar{h}_i$ (resp. $\bar{h}_j$) is obtained by letting each term in $h_i$ (resp. $h_j$) having $\pi$ as a factor zero
 so that  this equation is considered in $R$.
Note that $\bar{h}_i$ and $\bar{h}_j$ are invertible as matrices with entries in $R$ by Remark 3.3.
Thus $m_{i,j}'=\bar{h}_i^{-1}\cdot {}^tm_{j,i}'\cdot \bar{h}_j$. 
This induces that each entry of $m_{i,j}'$ is expressed as a linear combination of the entries of $m_{j,i}'$.
Thus there are exactly $n_in_j$ independent linear equations  among the entries of $m_{i,j}', m_{j,i}'$.\\
\textit{   }\\

Next, we put (1)-(4) into  (A.1).
Then we obtain
\begin{equation}\pi^i\left(\sigma(1+\pi\cdot {}^tm_{i,i}')h_i(1+\pi m_{i,i}')\right).\end{equation}
We interpret this so as to obtain equations defining $\tilde{G}^1$.
There are 4 cases, indexed by  (i), (ii), (iii), (iv), according to types of $L_i$.


\begin{enumerate}
\item[(i)] Assume that $i$ is odd. Then $\pi^ih_i=(\pi\cdot\sigma(\pi))^{(i-1)/2}\pi a_i$ as explained in Section 3.3 and thus we have
\begin{equation*}
a_i'=\sigma(1+\pi\cdot {}^tm_{i,i}')a_i(1+\pi m_{i,i}').
\end{equation*}
Here,  the non-diagonal entries of this equation are considered in $B\otimes_AR$ and each diagonal entry of $a_i'$ is of the form $\epsilon\pi x_i$ with $x_i\in R$.

Thus, we can cancel terms having $\pi^2$ as a factor and the above equation equals
\begin{equation*}   a_i'=a_i+\sigma(\pi)\cdot {}^tm_{i,i}'a_i+ \pi\cdot a_i m_{i,i}'.
\end{equation*}
By letting $a_i'=a_i$,  we have the following equation
\begin{equation*}   \sigma(\pi)\cdot {}^tm_{i,i}'a_i+ \pi\cdot a_i m_{i,i}'=0.
\end{equation*}

Since this is an equation in $B\otimes_AR$, it is of the form $X+\pi Y=0$.
Note that the reduction of $\epsilon$   mod $\pi$ is $1$.
We denote by $\bar{a}_i$ the reduction of $a_i$   mod $\pi$.
Thus we have
\begin{equation*}  {}^tm_{i,i}'\bar{a}_i+  \bar{a}_i m_{i,i}'=0.
\end{equation*}
This is a matrix equation over $R$, in a usual sense, and $\bar{a}_i$ is symmetric and the diagonal entries of $\bar{a}_i$ are $0$.
More precisely, $$\bar{a}_i=\begin{pmatrix} \begin{pmatrix} 0&1\\1&0\end{pmatrix}& &  \\ &\ddots &  \\ & & \begin{pmatrix} 0&1\\1&0\end{pmatrix}\end{pmatrix}.$$
Then we can see that there is no contribution coming from the diagonal entries of  $  {}^tm_{i,i}'\bar{a}_i+  \bar{a}_i m_{i,i}'=0$
and that there are exactly $(n_i^2-n_i)/2$ independent linear equations.
 Thus  $(n_i^2+n_i)/2$ entries of $m_{i,i}'$ determine all entries of $m_{i,i}'$.
Note that the conditions on $m_{i,i}'$ , viewed as a matrix with entries in $\kappa$,  are tantamount to this matrix
belonging to the Lie algebra of a symplectic group associated to an obvious alternating form given by $\bar{a}_i$.
Then $(n_i^2+n_i)/2$ is the dimension of this symplectic group.

For example, let $m_{i,i}'=\begin{pmatrix} x&y\\z&w\end{pmatrix}$ and $\bar{a}_i=\begin{pmatrix} 0&1\\1&0\end{pmatrix}$.
Then $${}^tm_{i,i}'\bar{a}_i+  \bar{a}_i m_{i,i}'=\begin{pmatrix} 2z&x+w\\x+w&2y\end{pmatrix}=\begin{pmatrix} 0&x+w\\x+w&0\end{pmatrix}.$$
Thus there is one linear equation $x+w=0$ and $x, y, z$ determine all entries of $m_{i,i}'$.\\

\item[(ii)] Assume that $i$ is even and $L_i$ is \textit{of type II}. This case is parallel to the above case.
Then $\pi^ih_i=(\pi\cdot\sigma(\pi))^{i/2} a_i$ as explained in Section 3.3 and we have
\begin{equation*}
a_i'=\sigma(1+\pi\cdot {}^tm_{i,i}')a_i(1+\pi m_{i,i}').
\end{equation*}
Here,  the non-diagonal entries of this equation are considered in $B\otimes_AR$ and each diagonal entry of $a_i'$ is of the form $2 x_i$ with $x_i\in R$.
Now,  the non-diagonal entries of $\sigma(\pi\cdot {}^tm_{i,i}')a_i(\pi m_{i,i}')$ are all $0$ since they contain $\pi^2$ as a factor.
In addition, the diagonal entries of $\sigma(\pi\cdot {}^tm_{i,i}')a_i(\pi m_{i,i}')$ are also $0$ since  they contain $\pi^4$ as  a factor.
Thus, the above equation equals
\begin{equation*}   a_i'=a_i+\sigma(\pi)\cdot {}^tm_{i,i}'a_i+ \pi\cdot a_i m_{i,i}'.
\end{equation*}
By letting $a_i'=a_i$,  we have the following equation
\begin{equation*}   \sigma(\pi)\cdot {}^tm_{i,i}'a_i+ \pi\cdot a_i m_{i,i}'=0.
\end{equation*}

Based on (2) of the description of $\underline{H}(R)$ for a $\kappa$-algebra $R$, which is explained in Section 3.3,
in order to investigate this equation,
we need to consider the non-diagonal entries of $\sigma(\pi)\cdot {}^tm_{i,i}'a_i+ \pi\cdot a_i m_{i,i}'$ as elements of $B\otimes_AR$
and the diagonal entries of $\sigma(\pi)\cdot {}^tm_{i,i}'a_i+ \pi\cdot a_i m_{i,i}'$ as of the form $2 x_i$ with $x_i\in R$.
Recall from Remark 3.3 that $$a_i=\begin{pmatrix} \begin{pmatrix} 0&1\\1&0\end{pmatrix}& &  \\ &\ddots &  \\ & & \begin{pmatrix} 0&1\\1&0\end{pmatrix}\end{pmatrix}.$$
Then we can see that each diagonal entry  as well as
each non-diagonal (upper triangular) entry of $\sigma(\pi)\cdot {}^tm_{i,i}'a_i+ \pi\cdot a_i m_{i,i}'$ produces a linear equation.
Thus there are exactly $(n_i^2+n_i)/2$ independent linear equations and $(n_i^2-n_i)/2$ entries of $m_{i,i}'$ determine all entries of $m_{i,i}'$.

For example, let $m_{i,i}'=\begin{pmatrix} x&y\\z&w\end{pmatrix}$ and $\bar{a}_i=\begin{pmatrix} 0&1\\1&0\end{pmatrix}$.
Then $$\sigma(\pi)\cdot {}^tm_{i,i}'a_i+ \pi\cdot a_i m_{i,i}'=\sigma(\pi)\begin{pmatrix} z&x\\w&y\end{pmatrix}+\pi\begin{pmatrix} z&w\\x&y\end{pmatrix}=\begin{pmatrix} (\sigma(\pi)+\pi)z&\sigma(\pi)x+\pi w\\\sigma(\pi)w+\pi x&(\sigma(\pi)+\pi)y \end{pmatrix}.$$
Recall that  $\sigma(\pi)=\epsilon\pi$ with $\epsilon\equiv 1$ mod $\pi$ and $\sigma(\pi)+\pi=2$, as explained at the beginning of Section 2.1.
Thus there are three linear equations $z=0, x+w=0, y=0 $ and $x$ determines every other entry of $m_{i,i}'$.\\

\item[(iii)] Assume that $i$ is even and $L_i$ is \textit{of type $I^o$}.
Then $\pi^ih_i=(\pi\cdot\sigma(\pi))^{i/2} \begin{pmatrix} a_i&\pi b_i\\ \sigma(\pi\cdot {}^t b_i) &1 +2c_i \end{pmatrix}$ as explained in Section 3.3 and we have
\begin{equation}
\begin{pmatrix} a_i'&\pi b_i'\\ \sigma(\pi\cdot {}^t b_i') &1 +2c_i' \end{pmatrix}=
\sigma(1+\pi\cdot {}^tm_{i,i}')\cdot\begin{pmatrix} a_i&\pi b_i\\ \sigma(\pi\cdot {}^t b_i) &1 +2c_i \end{pmatrix}\cdot(1+\pi m_{i,i}').
\end{equation}
Here,  the non-diagonal entries of $a_i'$ as well as the entries of $b_i'$ are considered in $B\otimes_AR$,
 each diagonal entry of $a_i'$ is of the form $2 x_i$ with $x_i\in R$, and $c_i'$ is in $R$.
In addition, $b_i=0, c_i=\bar{\gamma}_i$ as explained in Remark 3.3.(2) and
$a_i=\begin{pmatrix} \begin{pmatrix} 0&1\\1&0\end{pmatrix}& &  \\ &\ddots &  \\ & & \begin{pmatrix} 0&1\\1&0\end{pmatrix}\end{pmatrix}.$

Notice that in this case, $m_{i,i}'=\begin{pmatrix} s_i'&\pi y_i'\\ \pi v_i' &\pi z_i' \end{pmatrix}$.
Compute $\sigma(\pi\cdot {}^tm_{i,i}')\cdot\begin{pmatrix} a_i&0\\ 0 &1 +2c_i \end{pmatrix}\cdot (\pi m_{i,i}')$ formally and
this equals $\sigma(\pi)\pi\begin{pmatrix} {}^ts_i'a_is_i'+\pi^2X_i &\pi Y_i\\ \sigma(\pi\cdot {}^tY_i) &\pi^2Z_i \end{pmatrix}$
for certain matrices $X_i, Y_i, Z_i$ with suitable sizes.
Thus we can ignore the contribution from $\sigma(\pi\cdot {}^tm_{i,i}')\begin{pmatrix} a_i&0\\ 0 &1 +2c_i \end{pmatrix}(\pi m_{i,i}')$
 in Equation (A.6) and so Equation (A.6) equals
\begin{multline*}
\begin{pmatrix} a_i'&\pi b_i'\\ \sigma(\pi\cdot {}^t b_i') &1 +2c_i' \end{pmatrix}=
\begin{pmatrix} a_i&0\\ 0 &1 +2c_i \end{pmatrix}+
\sigma(\pi) \begin{pmatrix} {}^ts_i'&\sigma(\pi)\cdot {}^tv_i'\\ \sigma(\pi)\cdot {}^ty_i' &\sigma(\pi) z_i' \end{pmatrix}\begin{pmatrix} a_i&0\\ 0 &1 +2c_i \end{pmatrix} \\ + \pi \begin{pmatrix} a_i&0\\ 0 &1 +2c_i \end{pmatrix} \begin{pmatrix} s_i'&\pi y_i'\\ \pi v_i' &\pi z_i' \end{pmatrix}.
\end{multline*}

We interpret each block of the above equation below:


\begin{enumerate}
\item Firstly, we observe the $(1,1)$-block.
The computation associated to this block is similar to that for the above case (ii).
Hence  there are exactly $((n_i-1)^2+(n_i-1))/2$ independent linear equations and $((n_i-1)^2-(n_i-1))/2$ entries of $s_i'$ determine all entries of $s_i'$.

\item Secondly, we observe the $(1,2)$-block.
We can ignore the contribution from ${}^tv_i'c_i$  since it contains $\pi^3$ as a factor.
Then the $(1,2)$-block is
\begin{equation}\pi b_i'=\sigma(\pi)\pi\cdot (\epsilon\cdot {}^tv_i'+1/\epsilon\cdot a_iy_i').
\end{equation}
By letting $b_i'=b_i=0$, we have
\begin{equation*}\sigma(\pi)\cdot (\epsilon\cdot {}^tv_i'+1/\epsilon\cdot a_iy_i')=0
\end{equation*}
as an equation in $B\otimes_AR$.
Thus there are exactly $(n_i-1)$ independent linear equations among the entries of $v_i'$ and $y_i'$ and the entries of $v_i'$ determine all entries of $y_i'$.

\item Finally, we observe the $(2,2)$-block.
Then it is
\begin{equation} 1+2c_i'=1+2c_i+(\pi^2+(\sigma(\pi))^2)z_i'+2(\pi^2+(\sigma(\pi))^2)c_iz_i'.
\end{equation}
Since $\pi^2+(\sigma(\pi))^2=(\pi+\sigma(\pi))^2-2\sigma(\pi)\pi$,
we see that $\pi^2+(\sigma(\pi))^2$ contains $4$ as a factor. Thus by letting  $c_i'=c_i$,
 this equation is trivial.
\end{enumerate}

By combining three cases (a)-(c), there are exactly $((n_i-1)^2+(n_i-1))/2+(n_i-1)=(n_i^2+n_i)/2-1$
independent linear equations and $(n_i^2-n_i)/2+1$ entries of
$m_{i,i}'$ determine all entries of $m_{i,i}'$.\\

\item[(iv)] Assume that $i$ is even and $L_i$ is \textit{of type $I^e$}.
Then $\pi^ih_i=(\pi\cdot\sigma(\pi))^{i/2} \begin{pmatrix} a_i&b_i&\pi e_i\\ \sigma({}^tb_i) &1+2f_i&1+\pi d_i \\ \sigma(\pi \cdot {}^te_i) &\sigma(1+\pi d_i) &2c_i \end{pmatrix}$ as explained in Section 3.3 and we have
\begin{multline}
\begin{pmatrix} a_i'&b_i'&\pi e_i'\\ \sigma({}^tb_i') &1+2f_i'&1+\pi d_i' \\ \sigma(\pi \cdot {}^te_i') &\sigma(1+\pi d_i') &2c_i' \end{pmatrix}=\\
\sigma(1+\pi\cdot {}^tm_{i,i}')\cdot
\begin{pmatrix} a_i&b_i&\pi e_i\\ \sigma({}^tb_i) &1+2f_i&1+\pi d_i \\ \sigma(\pi \cdot {}^te_i) &\sigma(1+\pi d_i) &2c_i \end{pmatrix}
\cdot(1+\pi m_{i,i}').
\end{multline}
Here,  the non-diagonal entries of $a_i'$ as well as the entries of $b_i', e_i', d_i'$ are considered in $B\otimes_AR$,
 each diagonal entry of $a_i'$ is of the form $2 x_i$ with $x_i\in R$, and $c_i', f_i'$ are in $R$.
In addition, $b_i=0, d_i=0, e_i=0, f_i=0, c_i=\bar{\gamma}_i$ as explained in Remark 3.3.(2) and
$a_i=\begin{pmatrix} \begin{pmatrix} 0&1\\1&0\end{pmatrix}& &  \\ &\ddots &  \\ & & \begin{pmatrix} 0&1\\1&0\end{pmatrix}\end{pmatrix}.$


Notice that in this case, $m_{i,i}'=\begin{pmatrix}  s_i^{\prime}& r_i^{\prime}&\pi t_i^{\prime}\\ \pi y_i^{\prime}&\pi x_i^{\prime}&\pi z_i^{\prime}\\  v_i^{\prime}& u_i^{\prime}&\pi w_i^{\prime} \end{pmatrix}$.
We compute $$\sigma(\pi\cdot {}^tm_{i,i}')\cdot\begin{pmatrix} a_i&0&0\\ 0&1&1 \\ 0&1 &2c_i \end{pmatrix}\cdot(\pi m_{i,i}')$$ formally and
this equals $\sigma(\pi)\pi\begin{pmatrix} {}^ts_i'a_is_i'+\pi^2X_i & Y_i & \pi Z_i
\\ \sigma( {}^tY_i) &{}^tr_i'a_ir_i'+\pi^2X_i'&\pi Y_i' \\ \sigma(\pi\cdot {}^tZ_i)&\sigma(\pi\cdot  {}^tY_i') &\pi^2 Z_i'  \end{pmatrix}$
for certain matrices $X_i, Y_i, Z_i, X_i', Y_i', Z_i'$ with suitable sizes.
Thus we can ignore the contribution from this part in
Equation (A.9) and so Equation (A.9) equals
\begin{multline*}
\begin{pmatrix} a_i'&b_i'&\pi e_i'\\ \sigma({}^tb_i') &1+2f_i'&1+\pi d_i' \\ \sigma(\pi \cdot {}^te_i') &\sigma(1+\pi d_i') &2c_i' \end{pmatrix}=
\begin{pmatrix} a_i&0&0\\ 0 &1&1 \\ 0 &1 &2c_i \end{pmatrix}+\\
\sigma(\pi) \begin{pmatrix}  {}^ts_i'& \sigma(\pi)\cdot {}^ty_i'& {}^tv_i' \\{}^tr_i'&\sigma(\pi)x_i'&u_i'
\\  \sigma(\pi)\cdot {}^tt_i'& \sigma(\pi)z_i'&\sigma(\pi) w_i^{\prime} \end{pmatrix}\begin{pmatrix} a_i&0&0\\ 0 &1&1 \\ 0 &1 &2c_i \end{pmatrix}+ \pi \begin{pmatrix} a_i&0&0\\ 0 &1&1 \\ 0 &1 &2c_i \end{pmatrix} \begin{pmatrix}  s_i^{\prime}& r_i^{\prime}&\pi t_i^{\prime}\\ \pi y_i^{\prime}&\pi x_i^{\prime}&\pi z_i^{\prime}\\  v_i^{\prime}& u_i^{\prime}&\pi w_i^{\prime} \end{pmatrix}.
\end{multline*}

We interpret each block of the above equation below:

\begin{enumerate}
\item Let us observe the $(1,1)$-block.
The computation associated to this block is similar to that for the above case (ii).
Hence  there are exactly $((n_i-2)^2+(n_i-2))/2$ independent linear equations and $((n_i-2)^2-(n_i-2))/2$ entries of $s_i'$ determine all entries of $s_i'$.

\item We observe the $(1,2)$-block.
This gives
\begin{equation}b_i'=\pi(\epsilon^2\pi\cdot {}^ty_i'+\epsilon\cdot {}^tv_i'+a_ir_i').
\end{equation}
This is an equation in $B\otimes_AR$. By letting $b_i'=b_i=0$,
 there are exactly $(n_i-2)$ independent linear equations  among the entries of $v_i', r_i'$.

\item The $(1,3)$-block is
\begin{equation*} \pi e_i'=\pi^2(\epsilon^2\cdot {}^ty_i'+(2/\pi)\cdot\epsilon\cdot {}^tv_i'c_i+a_it_i').\end{equation*}
By letting $e_i'=e_i=0$,
we have
\begin{equation}  e_i'=\pi(\epsilon^2\cdot {}^ty_i'+(2/\pi)\cdot\epsilon\cdot {}^tv_i'c_i+a_it_i')=\pi(\epsilon^2\cdot {}^ty_i'+a_it_i')=
\pi({}^ty_i'+a_it_i')=0.\end{equation}
This is an equation in $B\otimes_AR$.
Thus there are exactly $(n_i-2)$ independent linear equations  among the entries of $y_i', t_i'$.

\item The $(2,3)$-block is
\begin{equation*}1+\pi d_i'=1+\sigma(\pi)(\sigma(\pi)x_i'+2u_i'c_i)+\pi^2(z_i'+w_i').\end{equation*} 
By letting $d_i'=d_i=0$, we have
\begin{equation}d_i'=\pi(\epsilon^2x_i'+ z_i'+w_i')=\pi(x_i'+ z_i'+w_i')=0.\end{equation}
This is an equation in $B\otimes_AR$.
Thus there is  exactly one independent linear equation  among the entries of $x_i', z_i', w_i'$.

\item The $(2,2)$-block is
\begin{equation*}1+2 f_i'=1+\sigma(\pi)(\sigma(\pi)x_i'+u_i')+\pi(\pi x_i'+u_i')=1+2u_i'+((\pi+\sigma(\pi))^2-2\pi\sigma(\pi))x_i'. 
\end{equation*}
By letting $f_i'=f_i=0$, we have
\begin{equation*}f_i'=u_i'+((\pi+\sigma(\pi))-\pi\sigma(\pi))x_i'=u_i'=0. 
\end{equation*}
This is an equation in $R$.
Thus there is  exactly one independent linear equation, which is $u_i'=0$.

\item The $(3,3)$-block is
\begin{multline}2c_i'=2c_i+\sigma(\pi)(\sigma(\pi)z_i'+2\sigma(\pi)w_i'c_i)+\pi(\pi z_i'
+2\pi w_i'c_i)=\\2c_i+((\pi+\sigma(\pi))^2-2\pi\sigma(\pi))(z_i'+2w_i'c_i).
\end{multline}
Since $((\pi+\sigma(\pi))^2-2\pi\sigma(\pi))$ contains $4$ as a factor,  by letting $c_i'=c_i$, this equation  is trivial.
\end{enumerate}
By combining  six cases (a)-(f),  there are exactly $((n_i-2)^2+(n_i-2))/2+2(n_i-2)+2=(n_i^2+n_i)/2-1$ independent linear equations and $(n_i^2-n_i)/2+1$ entries of
$m_{i,i}'$ determine all entries of $m_{i,i}'$.\\
\end{enumerate}

We now combine all the work done in this proof. Namely, we collect the above (i), (ii), (iii), (iv) which are  the interpretations of Equation (A.5), together with
Equation (A.4). 
Then there are exactly
$$\sum_{i<j}n_in_j+\sum_{i:\mathrm{odd}}\frac{n_i^2-n_i}{2}+\sum_{i:\mathrm{even}}\frac{n_i^2+n_i}{2}-\#\{i:i\mathrm{~is~even~and~}L_i\mathrm{~is~\textit{of type I}}\}$$ independent linear equations among the entries of $m$. Furthermore, all coefficients of these equations are in $\kappa$.
Therefore, we consider $\tilde{G}^1$ as a subvariety of $\tilde{M}^1$ determined by these linear equations.
Since  $\tilde{M}^1$ is an affine space of dimension $n^2$,
  the underlying algebraic variety of $\tilde{G}^1$ over $\kappa$ is  an affine space of dimension
$$\sum_{i<j}n_in_j+\sum_{i:\mathrm{odd}}\frac{n_i^2+n_i}{2}+\sum_{i:\mathrm{even}}\frac{n_i^2-n_i}{2}+\#\{i:i\mathrm{~is~even~and~}L_i\mathrm{~is~\textit{of type I}}\}.$$
This completes the proof by using a theorem of Lazard which is stated at the beginning of Appendix A. 
\end{proof}

Let $R$ be a $\kappa$-algebra.
We describe the functor of points of the scheme  $\mathrm{Ker~}\tilde{\varphi}/\tilde{M}^1$
by using points of the scheme $(\underline{M}'\otimes\kappa)/ \underline{\pi M}'$, based on Lemma A.3.
Recall from two paragraphs before Lemma A.3
 that $(1+)^{-1}(\underline{M}^{\ast})$, which is an open subscheme of  $\underline{M}'$, is a group scheme with the operation $\star$.
Let $\tilde{M}'$ be the special fiber of $(1+)^{-1}(\underline{M}^{\ast})$.
Since $\tilde{M}^{1}$ is a closed normal subgroup of $\tilde{M} (=\underline{M}^{\ast}\otimes\kappa)$ (cf. Lemma A.3.(i)),
$\underline{\pi M'}$, which is the inverse image of $\tilde{M}^{1}$ under the isomorphism $1+$, is a closed normal subgroup of $\tilde{M}'$.
Therefore, the morphism $1+$ induces the following isomorphism of group schemes, which is also denoted by $1+$,
$$1+  : \tilde{M}'/\underline{\pi M'}\longrightarrow \tilde{M}/\tilde{M}^{1}.$$
Note that $\tilde{M}'/\underline{\pi M'}(R)=\tilde{M}'(R)/\underline{\pi M'}(R)$ by Lemma A.1.
Thus each element of $(\mathrm{Ker~}\tilde{\varphi}/\tilde{M}^1)(R)$ is uniquely written as $1+\bar{x}$,
 where $\bar{x}\in \tilde{M}'(R)/  \underline{\pi M}'(R)$.
 Here, by $1+\bar{x}$, we mean the image of $\bar{x}$ under the morphism $1+$ at the level of $R$-points.

We still need a better description of an element of $(\mathrm{Ker~}\tilde{\varphi}/\tilde{M}^1)(R)$ by using a point of the scheme $(\underline{M}'\otimes\kappa)/ \underline{\pi M}'$.
Note that  $(\underline{M}'\otimes\kappa)/ \underline{\pi M}'$ is a quotient of group schemes with respect to the addition, whereas
$\tilde{M}'/  \underline{\pi M}'$ is a quotient of group schemes with respect to the operation $\star$.

We claim that the open immersion $\iota : \tilde{M}' \rightarrow \underline{M}'\otimes\kappa, x\mapsto x$ induces a monomorphism of schemes
$$\bar{\iota} : \tilde{M}'/  \underline{\pi M}' \rightarrow (\underline{M}'\otimes\kappa)/ \underline{\pi M}'.$$
Choose $x\in \tilde{M}'(R)$ and $\pi y\in \underline{\pi M}'(R)$ for a $\kappa$-algebra $R$.
Since $x\star \pi y=x+\pi(y+xy)$, both $x$ and $x\star \pi y$ give the same element in $((\underline{M}'\otimes\kappa)/ \underline{\pi M}')(R)$.
Thus the morphism $\bar{\iota}$ is well-defined.

In order to show that $\bar{\iota}$ is a monomorphism, choose $x, y \in \tilde{M}'(R)$ such that $x=y+\pi z$ with $\pi z\in \underline{\pi M}'(R)$.
Let $y' (\in \tilde{M}'(R))$ be the inverse of $y$ so that $y\star y'=y+y'+yy'=0$.
Then $\pi(z+y'z)$ is an element of $\underline{\pi M}'(R)$. We have the following identity:
\[x\star \pi(z+y'z)=(y+\pi z)\star \pi(z+y'z)=y+\pi(y+y'+yy')z=y.\]
Therefore, $x$ and $y$ give the same element in $(\tilde{M}'/  \underline{\pi M}')(R)$, which shows the injectivity of the above morphism.

Note that the operation $\star$ is closed in $\underline{M}'\otimes\kappa$ as mentioned in the third paragraph following Lemma A.2.
We can also easily check that the operation $\star$ is  well-defined on $(\underline{M}'\otimes\kappa)/ \underline{\pi M}'$,
which turns to be a scheme of monoids with respect to  $\star$,
and that the morphism $\bar{\iota}$ is a monomorphism of monoid schemes. 

To summarize, the morphism $1+  : \tilde{M}'/\underline{\pi M'}\longrightarrow \tilde{M}/\tilde{M}^{1}$ is an isomorphism of group schemes and
the morphism  $\bar{\iota} : \tilde{M}'/  \underline{\pi M}' \rightarrow (\underline{M}'\otimes\kappa)/ \underline{\pi M}'$ is a monomorphism
preserving the operation $\star$.
Therefore, each element of $(\mathrm{Ker~}\tilde{\varphi}/\tilde{M}^1)(R)$ is uniquely written as $1+\bar{x}$,
 where $\bar{x}\in (\underline{M}'\otimes\kappa)(R)/ \underline{\pi M}'(R)$.
Here, by $1+\bar{x}$, we mean $(1+)\circ \bar{\iota}^{-1}(\bar{x})$.
From now on to the end of this paper, we keep the notation $1+\bar{x}$ to express an element of $(\mathrm{Ker~}\tilde{\varphi}/\tilde{M}^1)(R)$
such that $\bar{x}$ is an element of $(\underline{M}'\otimes\kappa)(R)/ \underline{\pi M}'(R)$
which is a quotient of $R$-valued points of group schemes
with respect to addition.
Then the product of two elements $1+\bar{x}$ and $1+\bar{y}$ is the same as $1+\bar{x}\star \bar{y}$ $(=1+(\bar{x} + \bar{y}+\bar{x} \bar{y}))$.

\begin{Rmk}
By the above argument, we  write an element of $(\mathrm{Ker~}\tilde{\varphi}/\tilde{M}^1)(R)$ formally as
$m= \begin{pmatrix} \pi^{max\{0,j-i\}}m_{i,j} \end{pmatrix} $
with $s_i,\cdots, w_i$ as in Section 3.2 such that
 each entry of each of the matrices $(m_{i,j})_{i\neq j}, s_i, \cdots, w_i$ is in $(B\otimes_AR)/(\pi\otimes 1)(B\otimes_AR)\cong R$.
In particular, based on the description of $\mathrm{Ker~}\tilde{\varphi}(R)$ given at the paragraph following  Lemma A.2,
we have the following conditions on $m$:
\begin{enumerate}
\item  Assume that $i$ is even and $L_i$ is \textit{of type I}. Then $s_i=\mathrm{id}$.
\item  $m_{i,i}=\mathrm{id}$ if  $L_i$ is \textit{of type II}.
\item  Assume that $i$ is odd. Then $\delta_{i-1}e_{i-1}\cdot m_{i-1, i}+\delta_{i+1}e_{i+1}\cdot m_{i+1, i}=0$.
Here, $\delta_j, e_j$ are  as explained in the description of $\mathrm{Ker~}\tilde{\varphi}(R)$.
\end{enumerate}
\end{Rmk}

 \begin{Thm}
$\mathrm{Ker~}\varphi/\tilde{G}^1 $ is isomorphic to $ \textbf{A}^{l^{\prime}}\times (\mathbb{Z}/2\mathbb{Z})^{\beta}$ as a $\kappa$-variety,
 where $\textbf{A}^{l^{\prime}}$ is an affine space of dimension $l^{\prime}$.
 Here, 
  \begin{itemize}
  \item $l^{\prime}$ is such that \textit{$l^{\prime}$ + dim $\tilde{G}^1 = l$}.
  Notice that $l$ is defined in Lemma 4.6 and
  that the dimension of $\tilde{G}^1$ is given in Theorem A.4.
\item  $\beta$ is the number of even integers $j$ such that $L_j$ is of type I and $L_{j+2}$ is of type II.
   \end{itemize}
   \end{Thm}

   \begin{proof} Lemma A.1 and Theorem A.4 imply that $\mathrm{Ker~}\varphi/\tilde{G}^1 $ represents the functor
    $R\mapsto \mathrm{Ker~}\varphi(R)/\tilde{G}^1(R)$.
    Recall that $\mathrm{Ker~}\varphi/\tilde{G}^1$ is a closed subgroup scheme of $\mathrm{Ker~}\tilde{\varphi}/\tilde{M}^1$
    as explained at the paragraph just before Theorem A.4.
    Let $m=\begin{pmatrix} \pi^{max\{0,j-i\}}m_{i,j} \end{pmatrix}$     be an element of  $(\mathrm{Ker~}\tilde{\varphi}/\tilde{M}^1)(R)$
     such that $m$ belongs to $(\mathrm{Ker~}\varphi/\tilde{G}^1)(R)$. 
We want to find equations which $m$ satisfies.
Note that the entries of $m$ involve $(B\otimes_AR)/(\pi\otimes 1)(B\otimes_AR)$ as explained in Remark A.5.

 Recall that $h$ is the fixed hermitian form and we consider it as an element in $\underline{H}(R)$ as explained in Remark 3.3.(2).
 We write it as a formal matrix $h=\begin{pmatrix} \pi^{i}\cdot h_i\end{pmatrix}$ with $(\pi^{i}\cdot h_i)$
for the $(i,i)$-block and $0$ for the remaining blocks. 
    We choose a representative $1+x\in \mathrm{Ker~}\varphi(R)$ of $m$ so that $h\circ (1+x)=h$.
    Any other representative of $m$ in  $\mathrm{Ker~}\tilde{\varphi}(R)$ is
    of the form $(1+x)(1+\pi y)$ with $y \in \underline{M}'(R)$
    and we have  $h\circ (1+x)(1+\pi y)=h\circ (1+\pi y)$. 
    Notice that $h\circ (1+\pi y)$ is an element of $\underline{H}(R)$ so we express it as $(f_{i,j}', a_i' \cdots f_i')$.
    We also let  $h=(f_{i,j}, a_i \cdots f_i)$.
    Here, we follow notation from Section 3.3, the paragraph just before Remark 3.3.
    Recall that $h=(f_{i,j}, a_i \cdots f_i)$ is described explicitly in Remark 3.3.(2).
Now, $1+\pi y$ is an element of $\tilde{M}^1(R)$ and so we can use our result (Equations (A.3), (A.7), (A.8), (A.10), (A.11), (A.12), (A.13))
  stated in the proof of Theorem A.4 in order to compute $h\circ (1+\pi y)$.
Based on this, we enumerate equations which $m$ satisfies as follows:

\begin{enumerate}
\item Assume that $i<j$. By Equation (A.3) which involves an element of $\tilde{M}^1(R)$,
each entry of $f_{i,j}'$ has $\pi$ as a factor so that $f_{i,j}'\equiv f_{i,j} (=0)$
mod $(\pi\otimes 1)(B\otimes_AR)$.
In other words, the $(i,j)$-block of $h\circ (1+x)(1+\pi y)$ divided by $\pi^{max\{i, j\}}$ is $f_{i,j} (=0)$
modulo $(\pi\otimes 1)(B\otimes_AR)$, which is independent of the choice of $1+\pi y$.
Let $\tilde{m}\in \mathrm{Ker~}\tilde{\varphi}(R)$ be a lift of $m$.
Therefore, if we write the $(i, j)$-block of $\sigma({}^t\tilde{m})\cdot h\cdot \tilde{m}$ as $\pi^{max\{i, j\}}\mathcal{X}_{i,j}(\tilde{m})$,
where $\mathcal{X}_{i,j}(\tilde{m}) \in M_{n_i\times n_j}(B\otimes_AR)$,
then the image of $\mathcal{X}_{i,j}(\tilde{m})$ in $M_{n_i\times n_j}(B\otimes_AR)/(\pi\otimes 1)M_{n_i\times n_j}(B\otimes_AR)\cong M_{n_i\times n_j}(R)$
is independent of the choice of the lift $\tilde{m}$ of $m$.
Therefore, we may denote this image by $\mathcal{X}_{i,j}(m)$.
On the other hand, by Equation (A.2), we have the following identity:
\begin{equation}\mathcal{X}_{i,j}(m)=\sum_{i\leq k \leq j} \sigma({}^tm_{k,i})\bar{h}_km_{k,j} \mathrm{~if~}i<j.\end{equation}
We explain how to interpret the above equation.
We know that $\mathcal{X}_{i,j}(m)$ and $m_{k,k'}$ (with $k\neq k'$) are matrices with entries in $(B\otimes_AR)/(\pi\otimes 1)(B\otimes_AR)$,
whereas $m_{i,i}$ and $m_{j,j}$ are formal matrices as explained in Remark A.5.
Thus we consider $\bar{h}_k$, $m_{i,i}$, and $m_{j,j}$ as matrices with entries in $(B\otimes_AR)/(\pi\otimes 1)(B\otimes_AR)$
by letting $\pi$ be zero in each entry of  formal matrices $h_k$, $m_{i,i}$, and $m_{j,j}$.
Here  we keep using $m_{i,i}$ and $m_{j,j}$ for matrices with entries in $(B\otimes_AR)/(\pi\otimes 1)(B\otimes_AR)$
  in the above equation in order to simplify notation.
 Later in Equation (A.23), they are denoted by $\bar{m}_{i,i}$ and $\bar{m}_{j,j}$.
Then the right hand side is computed as a sum of products of matrices (involving the usual matrix addition and multiplication)
with entries in $(B\otimes_AR)/(\pi\otimes 1)(B\otimes_AR)$.
Thus, the assignment $m\mapsto \mathcal{X}_{i,j}(m)$ is polynomial in $m$.
Furthermore, since $m$ actually belongs to $\mathrm{Ker~}\varphi(R)/\tilde{G}^1(R)$,
 we have the following equation by the argument made at the beginning of this paragraph:
  $$\mathcal{X}_{i,j}(m)=f_{i,j}\textit{ mod $(\pi\otimes 1)(B\otimes_AR)$}=0.$$
Thus we get an $n_i\times n_j$ matrix $\mathcal{X}_{i,j}$ of polynomials on $\mathrm{Ker~}\tilde{\varphi}/\tilde{M}^1$
defined by Equation (A.14),
vanishing on the subscheme $\mathrm{Ker~}\varphi/\tilde{G}^1$.

Before moving to  the following steps, we fix  notation.
Let $m$ be an element in $(\mathrm{Ker~}\tilde{\varphi}/\tilde{M}^1)(R)$ and $\tilde{m}\in \mathrm{Ker~}\tilde{\varphi}(R)$ be its lift.
For any block $x_i$  of $m$, $\tilde{x}_i$ is denoted by the corresponding block of $\tilde{m}$ whose reduction is $x_i$.
Since $x_i$ is a block of an element of $(\mathrm{Ker~}\tilde{\varphi}/\tilde{M}^1)(R)$,
it involves $(B\otimes_AR)/(\pi\otimes 1)(B\otimes_AR)$ as explained in Remark A.5, whereas
$\tilde{x}_i$ involves $B\otimes_AR$.
In addition,
for a block $a_i$  of $h$, $\bar{a}_i$ is denoted by  the image  of $a_i$ in  $(B\otimes_AR)/(\pi\otimes 1)(B\otimes_AR)$.\\

\item Assume that $i$ is even and $L_i$ is \textit{of type} $\textit{I}^o$.
By Equation (A.7) which involves an element of $\tilde{M}^1(R)$,
each entry of $b_i'$ has $\pi$ as a factor so that $b_i'\equiv b_i=0$ mod $(\pi\otimes 1)(B\otimes_AR)$.
Let $\tilde{m}\in \mathrm{Ker~}\tilde{\varphi}(R)$ be a lift of $m$.
By using an argument similar to the paragraph   just before Equation (A.14) of Step (1),
 if we write the $(1, 2)$-block of the $(i, i)$-block of the formal matrix product
$\sigma({}^t\tilde{m})\cdot h\cdot \tilde{m}$ as $(\pi\cdot\sigma(\pi))^{i/2}\cdot \pi\mathcal{X}_{i,1,2}(\tilde{m})$,
where $\mathcal{X}_{i,1,2}(\tilde{m}) \in M_{(n_i-1)\times 1}(B\otimes_AR)$,
then the image of $\mathcal{X}_{i,1,2}(\tilde{m})$ in $M_{(n_i-1)\times 1}(B\otimes_AR)/(\pi\otimes 1)M_{(n_i-1)\times 1}(B\otimes_AR)$
is independent of the choice of the lift $\tilde{m}$ of $m$.
Therefore, we may denote this image by $\mathcal{X}_{i,1,2}(m)$.
As Equation (A.14) of Step (1), we need to express $\mathcal{X}_{i,1,2}(m)$ as matrices.
Recall that $\pi^ih_i=(\pi\cdot\sigma(\pi))^{i/2} \begin{pmatrix} a_i&0\\ 0 &1 +2c_i \end{pmatrix}
=\pi^i\cdot\epsilon^{i/2}\begin{pmatrix} a_i&0\\ 0 &1 +2c_i \end{pmatrix}$ and  $\epsilon\equiv 1$ mod $\pi\otimes 1$.
We  write $m_{i,i}$ as $\begin{pmatrix} id&\pi y_i\\ \pi v_i&1+\pi z_i \end{pmatrix}$
and  $\tilde{m}_{i,i}$ as $\begin{pmatrix} \tilde{s}_i&\pi \tilde{y}_i\\ \pi \tilde{v}_i&1+\pi \tilde{z}_i \end{pmatrix}$
such that $\tilde{s}_i=\mathrm{id}$ mod $\pi \otimes 1$.
Then
\begin{equation}
\sigma({}^t\tilde{m}_{i,i})h_i\tilde{m}_{i,i}=\epsilon^{i/2}\begin{pmatrix}\sigma({}^t\tilde{s}_i)&\sigma(\pi\cdot {}^t \tilde{v}_i)\\ \sigma(\pi\cdot {}^t \tilde{y}_i)&1+\sigma(\pi \tilde{z}_i) \end{pmatrix}
 \begin{pmatrix} a_i&0\\ 0 &1 +2c_i \end{pmatrix} \begin{pmatrix} \tilde{s}_i&\pi \tilde{y}_i\\ \pi \tilde{v}_i&1+\pi \tilde{z}_i \end{pmatrix}.
\end{equation}
Then the $(1,2)$-block of $\sigma({}^t\tilde{m}_{i,i})h_i\tilde{m}_{i,i}$ is $\epsilon^{i/2}\pi(a_i\tilde{y}_i+\epsilon\sigma({}^t\tilde{v}_i))+\pi^2(\ast)$ for a certain polynomial $(\ast)$.
Therefore, by observing the $(1, 2)$-block  of Equation (A.1), we have
\[
\mathcal{X}_{i,1,2}(m)=\bar{a}_iy_i+{}^tv_i+\mathcal{P}^i_{1, 2}.
\]
 Here, $\mathcal{P}^i_{1, 2}$ is a polynomial with variables in the entries  of $m_{i-1, i}, m_{i+1, i}$.
 Note that this is an equation in $(B\otimes_AR)/(\pi\otimes 1)(B\otimes_AR)$.
 Thus $\epsilon$, which is appeared in the $(1,2)$-block of $\sigma({}^t\tilde{m}_{i,i})h_i\tilde{m}_{i,i}$,
 has been ignored since  $\epsilon\equiv 1$ mod $\pi\otimes 1$.
 Furthermore, since $m$ actually belongs to $\mathrm{Ker~}\varphi(R)/\tilde{G}^1(R)$,
 we have the following equation by the argument made at the beginning of this paragraph:
  \begin{equation}
  \mathcal{X}_{i,1,2}(m)=\bar{a}_iy_i+{}^tv_i+\mathcal{P}^i_{1, 2}=\bar{b}_i=0.
  \end{equation}
Thus we get  polynomials $\mathcal{X}_{i,1,2}$ on $\mathrm{Ker~}\tilde{\varphi}/\tilde{M}^1$,
vanishing on the subscheme $\mathrm{Ker~}\varphi/\tilde{G}^1$.\\



\item  Assume that $i$ is even and $L_i$ is \textit{of type} $\textit{I}^e$.
The argument used in this step is similar to that of the above Step (2).
By Equations (A.10), (A.11), and (A.12) which involve an element of $\tilde{M}^1(R)$,
each entry of $b_i', e_i', d_i'$ has $\pi$ as a factor so that $b_i'\equiv b_i=0, e_i'\equiv e_i=0, d_i'\equiv d_i=0$ mod
$(\pi\otimes 1)(B\otimes_AR)$.
Let $\tilde{m}\in \mathrm{Ker~}\tilde{\varphi}(R)$ be a lift of $m$.
By using an argument similar to the paragraph   just before Equation (A.14) of Step (1),
if we write the $(1, 2), (1,3), (2,3)$-blocks of the $(i, i)$-block of the formal matrix product
$\sigma({}^t\tilde{m})\cdot h\cdot \tilde{m}$ as $(\pi\cdot\sigma(\pi))^{i/2}\cdot \pi\mathcal{X}_{i,1,2}(\tilde{m})$,
$(\pi\cdot\sigma(\pi))^{i/2}\cdot \pi\mathcal{X}_{i,1,3}(\tilde{m})$, $(\pi\cdot\sigma(\pi))^{i/2}\cdot \pi\mathcal{X}_{i,2,3}(\tilde{m})$,
respectively, where 
$\mathcal{X}_{i,1,2}(\tilde{m}) \textit{ and }\mathcal{X}_{i,1,3}(\tilde{m}) \in M_{(n_i-2)\times 1}(B\otimes_AR)$ and $\mathcal{X}_{i,2,3}(\tilde{m}) \in B\otimes_AR$,
then the images of $\mathcal{X}_{i,1,2}(\tilde{m})$ and $\mathcal{X}_{i,1,3}(\tilde{m})$ in $M_{(n_i-2)\times 1}(B\otimes_AR)/(\pi\otimes 1)M_{(n_i-2)\times 1}(B\otimes_AR)$ and the image of $\mathcal{X}_{i,2,3}(\tilde{m})$ in $(B\otimes_AR)/(\pi\otimes 1)(B\otimes_AR)$
are independent of the choice of the lift $\tilde{m}$ of $m$.
Therefore, we may denote these images by $\mathcal{X}_{i,1,2}(m)$, $\mathcal{X}_{i,1,3}(m)$, and $\mathcal{X}_{i,2,3}(m)$, respectively.
As Equation (A.14) of Step (1), we need to express $\mathcal{X}_{i,1,2}(m)$, $\mathcal{X}_{i,1,3}(m)$, and $\mathcal{X}_{i,2,3}(m)$ as matrices.
Recall that $\pi^ih_i=(\pi\cdot\sigma(\pi))^{i/2} \begin{pmatrix} a_i&0&0\\ 0 &1&1 \\ 0 &1 &2c_i \end{pmatrix}
=\pi^i\cdot\epsilon^{i/2}\begin{pmatrix} a_i&0&0\\ 0 &1&1 \\ 0 &1 &2c_i \end{pmatrix}$  and $\epsilon\equiv 1$ mod $\pi\otimes 1$.
We  write $m_{i,i}$ as $\begin{pmatrix} id&r_i&\pi t_i\\ \pi y_i&1+\pi x_i&\pi z_i\\ v_i&u_i&1+\pi w_i \end{pmatrix}$
and $\tilde{m}_{i,i}$ as $\begin{pmatrix} \tilde{s}_i&\tilde{r}_i&\pi \tilde{t}_i\\ \pi \tilde{y}_i&1+\pi \tilde{x}_i&\pi \tilde{z}_i\\ \tilde{v}_i&\tilde{u}_i&1+\pi \tilde{w}_i \end{pmatrix}$ such that $\tilde{s}_i=\mathrm{id}$ mod $\pi \otimes 1$.
Then
\begin{equation}
\sigma({}^t\tilde{m}_{i,i})h_i\tilde{m}_{i,i}=\epsilon^{i/2}\begin{pmatrix}\sigma({}^t\tilde{s}_i)&\sigma(\pi\cdot {}^t \tilde{y}_i)&\sigma( {}^t \tilde{v}_i)\\
\sigma({}^t \tilde{r}_i)&1+\sigma(\pi \tilde{x}_i)&\sigma(\tilde{u}_i)\\\sigma(\pi\cdot {}^t \tilde{t}_i)&\sigma(\pi\cdot {}^t \tilde{z}_i)&1+\sigma(\pi \tilde{w}_i) \end{pmatrix}
 \begin{pmatrix} a_i&0&0\\ 0 &1&1 \\ 0 &1 &2c_i \end{pmatrix}
\begin{pmatrix} \tilde{s}_i&\tilde{r}_i&\pi \tilde{t}_i\\ \pi \tilde{y}_i&1+\pi \tilde{x}_i&\pi \tilde{z}_i\\ \tilde{v}_i&\tilde{u}_i&1+\pi \tilde{w}_i \end{pmatrix}.
\end{equation}
Then the $(1,2)$-block of $\sigma({}^t\tilde{m}_{i,i})h_i\tilde{m}_{i,i}$ is $\epsilon^{i/2}(a_i\tilde{r}_i+\sigma({}^t\tilde{v}_i))+\pi(\ast)$,
 the $(1, 3)$-block is $\epsilon^{i/2}\pi(a_i\tilde{t}_i+\epsilon\sigma({}^t\tilde{y}_i)+\sigma({}^t\tilde{v}_i)\tilde{z}_i)
 +\pi^2(\ast\ast)$,
 and the $(2, 3)$-block is $\epsilon^{i/2}(1+\pi(\sigma({}^t\tilde{r}_i)a_i\tilde{t}_i+\epsilon\sigma(\tilde{x}_i)+\tilde{z}_i+\tilde{w}_i+\sigma(\tilde{u}_i)\tilde{z}_i)
 +\pi^2(\ast\ast\ast))$
 for  certain polynomials $(\ast), (\ast\ast), (\ast\ast\ast)$.
Therefore, by observing the $(1, 2), (1,3), (2,3)$-blocks of  Equation (A.1) again, we have
\[\left \{
  \begin{array}{l}
  \mathcal{X}_{i,1,2}(m)=\bar{a}_ir_i+{}^tv_i;\\
  \mathcal{X}_{i,1,3}(m)=\bar{a}_i t_i+{}^ty_i+{}^tv_iz_i+\mathcal{P}^i_{1, 3};\\
  \mathcal{X}_{i,2,3}(m)={}^tr_i\bar{a}_it_i+x_i+z_i+ w_i+u_iz_i+\mathcal{P}^i_{2, 3}.\\
    \end{array} \right.\]
Here, $\mathcal{P}^i_{1, 3}, \mathcal{P}^i_{2, 3}$ are suitable polynomials with variables in the entries  of $m_{i-1, i}, m_{i+1, i}$.
These equations are considered in $(B\otimes_AR)/(\pi\otimes 1)(B\otimes_AR)$.
Since $m$ actually belongs to $\mathrm{Ker~}\varphi(R)/\tilde{G}^1(R)$,
 we have the following equation by the argument made at the beginning of this paragraph:
  \begin{equation}
\left \{
  \begin{array}{l}
 \mathcal{X}_{i,1,2}(m)=\bar{a}_ir_i+{}^tv_i=\bar{b}_i=0;\\
\mathcal{X}_{i,1,3}(m)=\bar{a}_i t_i+{}^ty_i+{}^tv_iz_i+\mathcal{P}^i_{1, 3}=\bar{e}_i=0; \\
\mathcal{X}_{i,2,3}(m)={}^tr_i\bar{a}_it_i+x_i+z_i+ w_i+u_iz_i+\mathcal{P}^i_{2, 3}=\bar{d}_i=0.
    \end{array} \right.
  \end{equation}
Thus we get  polynomials $\mathcal{X}_{i,1,2}, \mathcal{X}_{i,1,3}, \mathcal{X}_{i,2,3}$ on $\mathrm{Ker~}\tilde{\varphi}/\tilde{M}^1$,
vanishing on the subscheme $\mathrm{Ker~}\varphi/\tilde{G}^1$.\\

\item Assume that $i$ is even and $L_i$ is \textit{of type I}.
By Equations (A.8) and (A.13) which involve an element of $\tilde{M}^1(R)$, $c_i'\equiv c_i=0$  mod $(\pi\otimes 1)(B\otimes_AR)$.
Let $\tilde{m}\in \mathrm{Ker~}\tilde{\varphi}(R)$ be a lift of $m$.
 By using an argument similar to the paragraph   just before Equation (A.14) of Step (1),
 if we write the $(2, 2)$-block (when $L_i$ is \textit{of type $I^o$}) or the $(3,3)$-block (when $L_i$ is \textit{of type $I^e$})
 of the $(i, i)$-block of $h\circ \tilde{m}=\sigma({}^t\tilde{m})\cdot h\cdot \tilde{m}$
  as $(\pi\cdot\sigma(\pi))^{i/2}\cdot (1+2\mathcal{X}_{i,i}(\tilde{m}))$ or
  $(\pi\cdot\sigma(\pi))^{i/2}\cdot (2\mathcal{X}_{i,i}(\tilde{m}))$ respectively,
where $\mathcal{X}_{i,i}(\tilde{m}) \in B\otimes_AR$,
then the image of $\mathcal{X}_{i,i}(\tilde{m})$ in $(B\otimes_AR)/(\pi\otimes 1)(B\otimes_AR)$
is independent of the choice of the lift $\tilde{m}$ of $m$.
Therefore, we may denote this image by $\mathcal{X}_{i,i}(m)$.
As Equation (A.14) of Step (1), we need to express $\mathcal{X}_{i,i}(m)$ as matrices.
By observing Equations (A.15) and (A.17), the  $(2, 2)$-block (when $L_i$ is \textit{of type $I^o$}) or the $(3,3)$-block (when $L_i$ is \textit{of type $I^e$})  of the formal matrix product $\sigma({}^t\tilde{m}_{i,i})h_i\tilde{m}_{i,i}$ is
$(\epsilon^{i/2} \textit{ if $L_i$ is \textit{of type $I^o$}}) + \epsilon^{i/2}(2c_i+(\pi+\sigma(\pi))\tilde{z}_i+\pi\sigma(\pi)\tilde{z}_i^2)+4(\ast)$ for a certain polynomial $(\ast)$.
Therefore, by observing the  $(2, 2)$-block (when $L_i$ is \textit{of type $I^o$}) or the $(3,3)$-block (when $L_i$ is \textit{of type $I^e$})
of Equation (A.1) again, we have
\begin{multline*}
 \mathcal{X}_{i,i}(\tilde{m})= \frac{1}{\pi^2}((\pi+\sigma(\pi))\tilde{z}_i+\pi\sigma(\pi)\tilde{z}_i^2+\sigma({}^t\tilde{m}_{i-1,i}^{\prime})\cdot\sigma(\pi) h_{i-1}\cdot \tilde{m}_{i-1,i}^{\prime}\\+\sigma({}^t\tilde{m}_{i+1,i}^{\prime})\cdot\pi h_{i+1}\cdot \tilde{m}_{i+1,i}^{\prime}+
  \sigma({}^t\tilde{m}_{i-2,i}^{\prime})\cdot\sigma(\pi)^2 h_{i-2}\cdot \tilde{m}_{i-2,i}^{\prime}+\sigma({}^t\tilde{m}_{i+2,i}^{\prime})\cdot \pi^2 h_{i+2}\cdot \tilde{m}_{i+2,i}^{\prime}).
\end{multline*}
Here, $\tilde{m}_{j,i}^{\prime}$ is the last column vector of the matrix $\tilde{m}_{j,i}$.
Note that the right hand side is a formal polynomial with entries in $\tilde{m}$. 
This equation should be interpreted as follows.
We formally compute the right hand side and then it is of the form $1/\pi^2(\pi^2X)$.
 $\mathcal{X}_{i,i}(\tilde{m})$ is defined as the modified $X$ by letting  each term having $\pi^2$ as a factor in $X$ be zero.
It is a polynomial with entries in $B\otimes_AR$.
 $\mathcal{X}_{i,i}(m)$ is the image of $\mathcal{X}_{i,i}(\tilde{m})$ in $(B\otimes_AR)/(\pi\otimes 1)(B\otimes_AR)$.
Let $\alpha$ be the unit in $B$ such that  $\epsilon=1+\alpha\pi$, as explained in Section 2.1.
Then $(\pi+\sigma(\pi))z_i+\pi\sigma(\pi)z_i^2=(2+\alpha\pi)\pi z_i+(1+\alpha\pi)\pi^2z_i^2$ and so $\mathcal{X}_{i,i}(\tilde{m})$ is written as follows:
\begin{multline*}
\mathcal{X}_{i,i}(\tilde{m})= \frac{1}{\pi^2}(\alpha\pi^2 \tilde{z}_i+\pi^2\tilde{z}_i^2+\sigma({}^t\tilde{m}_{i-1,i}^{\prime})\cdot\sigma(\pi) h_{i-1}\cdot \tilde{m}_{i-1,i}^{\prime}+\sigma({}^t\tilde{m}_{i+1,i}^{\prime})\cdot\pi h_{i+1}\cdot \tilde{m}_{i+1,i}^{\prime}\\+
  \sigma({}^t\tilde{m}_{i-2,i}^{\prime})\cdot\sigma(\pi)^2 h_{i-2}\cdot \tilde{m}_{i-2,i}^{\prime}+\sigma({}^t\tilde{m}_{i+2,i}^{\prime})\cdot \pi^2 h_{i+2}\cdot \tilde{m}_{i+2,i}^{\prime}).
  \end{multline*}
  We can then write $\mathcal{X}_{i,i}(m)$ by using $m$ and $\tilde{m}$ as follows:
  \begin{multline}
\mathcal{X}_{i,i}(m)=(\bar{\alpha}z_i+z_i^2+  {}^tm_{i-2,i}^{\prime}\cdot \bar{h}_{i-2}\cdot m_{i-2,i}^{\prime}+m_{i+2,i}^{\prime}\cdot \bar{h}_{i+2}\cdot m_{i+2,i}^{\prime})\\
+\frac{1}{\pi^2}(\sigma({}^t\tilde{m}_{i-1,i}^{\prime})\cdot\sigma(\pi) h_{i-1}\cdot \tilde{m}_{i-1,i}^{\prime}+\sigma({}^t\tilde{m}_{i+1,i}^{\prime})\cdot\pi h_{i+1}\cdot \tilde{m}_{i+1,i}^{\prime}).
 \end{multline}
 Here, $\bar{\alpha}$ is the image of $\alpha$ in $\kappa$ and
$m_{j,i}^{\prime}$ is the last column vector of the matrix $m_{j,i}$.
Note that the $\sigma$-action on $(B\otimes_AR)/(\pi\otimes 1)(B\otimes_AR)$ is trivial
and so we remove $\sigma$ in the first line of the above equation.
 Here, why we do not express $\mathcal{X}_{i,i}(m)$ based only on the entries  in $m$ as in Steps (1)-(3) is that
  two terms involving $h_{i-1}$ and $h_{i+1}$ have only $\pi$ as a factor which makes the expression with $m$ complicated
   in a sense of notation.
Thus, in the above expression of $\mathcal{X}_{i,i}(m)$, the first line is just a polynomial in $(B\otimes_AR)/(\pi\otimes 1)(B\otimes_AR)$
and the second line is interpreted as explained above as a formal expression.
Note that the second line is independent of the choice of lifts $\tilde{m}_{i-1,i}^{\prime}$ and $\tilde{m}_{i+1,i}^{\prime}$
of $m_{i-1,i}^{\prime}$ and $m_{i+1,i}^{\prime}$ respectively as explained in the first paragraph of Step (4).
For example, let $\pi h_{i+1}=\begin{pmatrix} 2&\pi\\   \sigma(\pi)&2b \end{pmatrix}$ with $b\in A$ and
let $\tilde{m}_{i+1,i}^{\prime}=\begin{pmatrix} x_1+\pi x_2\\   y_1+\pi y_2 \end{pmatrix}$
 such that $m_{i+1,i}^{\prime}=\begin{pmatrix} x_1\\   y_1\end{pmatrix}$.
 By Section 2.1, we may assume that  $\pi+\sigma(\pi)=2$ and $\pi\cdot\sigma(\pi)=\epsilon\pi^2=2u$ with $\epsilon\equiv 1$ mod $\pi$ and a unit $u\in A$.
Then as a part of $\mathcal{X}_{i,i}(m)$, we can see that
$$\frac{1}{\pi^2}\sigma({}^t\tilde{m}_{i+1,i}^{\prime})\cdot\pi h_{i+1}\cdot \tilde{m}_{i+1,i}^{\prime}=
\frac{1}{u}(x_1^2+x_1y_1+by_1^2).$$

Since $m$ actually belongs to $\mathrm{Ker~}\varphi(R)/\tilde{G}^1(R)$,
 we have the following equation by the argument made at the beginning of this paragraph:
 \begin{multline}
\mathcal{F}_i : \mathcal{X}_{i,i}(m)= (\bar{\alpha}z_i+z_i^2+  {}^tm_{i-2,i}^{\prime}\cdot \bar{h}_{i-2}\cdot m_{i-2,i}^{\prime}+m_{i+2,i}^{\prime}\cdot \bar{h}_{i+2}\cdot m_{i+2,i}^{\prime})\\
+\frac{1}{\pi^2}(\sigma({}^t\tilde{m}_{i-1,i}^{\prime})\cdot\sigma(\pi) h_{i-1}\cdot \tilde{m}_{i-1,i}^{\prime}+\sigma({}^t\tilde{m}_{i+1,i}^{\prime})\cdot\pi h_{i+1}\cdot \tilde{m}_{i+1,i}^{\prime})=\bar{c}_i =0.\\
  \end{multline}
Thus we get  polynomials $\mathcal{X}_{i,i}$ on $\mathrm{Ker~}\tilde{\varphi}/\tilde{M}^1$,
vanishing on the subscheme $\mathrm{Ker~}\varphi/\tilde{G}^1$.\\

\item We now choose an even integer $j$ such that $L_{j}$ is \textit{of type I} and $L_{j+2}$ is \textit{of type II} (possibly zero, by our convention).
For each such  $j$, there is a non-negative integer $m_j$ such that $L_{j-2l}$ is \textit{of type I} for every $l$ with $0\leq l\leq m_j$ and
$L_{j-2(m_j+1)}$ is \textit{of type II}.
Then we claim that the sum of equations $$\sum_{l=0}^{m_j} \frac{1}{\bar{\alpha}^2}\mathcal{F}_{j-2l}$$ is the same as
\begin{equation}
\sum_{l=0}^{m_j} (\frac{z_{j-2l}}{\bar{\alpha}}+(\frac{z_{j-2l}}{\bar{\alpha}})^2)=
\left(\sum_{l=0}^{m_j}\frac{z_{j-2l}}{\bar{\alpha}}\right)\left(\sum_{l=0}^{m_j}(\frac{z_{j-2l}}{\bar{\alpha}})+1\right)=0.
\end{equation}
Here, $\bar{\alpha}$ is the image of $\alpha$ in $\kappa$ and
 we consider this equation in $(B\otimes_AR)/(\pi\otimes 1)(B\otimes_AR)$.
We postpone the proof of this claim to Lemma A.7. \\

\end{enumerate}

Let $G^{\ddag}$ be the subfunctor of $ \mathrm{Ker~}\tilde{\varphi}/\tilde{M}^1$ consisting of those $m$ 
satisfying Equations (A.14), (A.16), (A.18) and (A.20).
Note that such  $m$ also satisfies Equation (A.21).
In Lemma A.8 below, we will prove that      
$G^{\ddag}$ is represented by a smooth  closed subscheme of $ \mathrm{Ker~}\tilde{\varphi}/\tilde{M}^1$
and is isomorphic to $ \textbf{A}^{l^{\prime}}\times (\mathbb{Z}/2\mathbb{Z})^{\beta}$ as a $\kappa$-variety,
 where $\textbf{A}^{l^{\prime}}$ is an affine space of dimension $l^{\prime}$.
 Here,
\[
 l^{\prime}=\sum_{i<j}n_in_j-\sum_{i:\mathrm{odd~and~} L_i:\textit{bound}}n_i+\\
 \sum_{i:\mathrm{even~and~}L_i:\textit{of type $I^o$}}(n_i-1)
 +\sum_{i:\mathrm{even~and~}L_i:\textit{of type $I^e$}}(2n_i-2).
\]
 For ease of notation, let $G^{\dag}=\mathrm{Ker~}\varphi/\tilde{G}^1$.
 Since $G^{\dag}$ and $G^{\ddag}$ are both closed subschemes of $ \mathrm{Ker~}\tilde{\varphi}/\tilde{M}^1$
 and $G^{\dag}(\bar{\kappa})\subset G^{\ddag}(\bar{\kappa})$, $(G^{\dag})_{\mathrm{red}}$ is a closed subscheme of $(G^{\ddag})_{\mathrm{red}}=G^{\ddag}$.
 It is easy to check that $\mathrm{dim~} G^{\dag} = \mathrm{dim~}G^{\ddag}$ 
 since  $\mathrm{dim~} G^{\dag} =\mathrm{dim~}\mathrm{Ker~}\varphi - \mathrm{dim~}\tilde{G}^1=l-\mathrm{dim~}\tilde{G}^1$
 and $\mathrm{dim~}G^{\ddag}=l'=l-\mathrm{dim~}\tilde{G}^1$.
 Here, $\mathrm{dim~}\mathrm{Ker~}\varphi = l$ is given in Lemma 4.6
 and dim $\tilde{G}^1$ is given in Theorem A.4.\\

 We claim that $(G^{\dag})_{\mathrm{red}}$ contains  at least one (closed) point of each connected component of $G^{\ddag}$.
Choose an even integer $j$ such that $L_j$ is \textit{of type I} and $L_{j+2}$ is \textit{of type II} (possibly zero, by our convention).
Consider the closed subgroup scheme $F_j$ of $\tilde{G}$ defined by the following equations:
\begin{itemize}
\item $m_{i,k}=0$ \textit{if $i\neq k$};
\item $m_{i,i}=\mathrm{id}$ \textit{if $i\neq j$};
\item and for $m_{j,j}$,
\[\left \{
  \begin{array}{l l}
  s_j=\mathrm{id~}, y_j=0, v_j=0 & \quad  \textit{if $L_i$ is \textit{of type} $\textit{I}^o$};\\
  s_j=\mathrm{id~}, r_j=t_j=y_j=v_j=u_j=w_j=0 & \quad \textit{if $L_i$ is \textit{of type} $\textit{I}^e$}.\\
    \end{array} \right.\]
\end{itemize}

We will prove in Lemma A.9 below
that each element of $F_j(R)$ for a $\kappa$-algebra $R$ satisfies $(z_j^1/\bar{\alpha})+(z_j^1/\bar{\alpha})^2=0$,
where $z_j=z_j^1+\pi z_j^2$, and
that $F_j$ is isomorphic to $ \textbf{A}^{1} \times \mathbb{Z}/2\mathbb{Z}$ as a $\kappa$-variety,
where $\textbf{A}^{1}$ is an affine space of dimension $1$. 


Notice that $F_j$ and $F_{j^{\prime}}$ commute with each other for all even integers $j\neq j^{\prime}$,
in the sense that $f_j\cdot f_{j^{\prime}}=f_{j^{\prime}}\cdot f_j$, where $f_j\in F_j $ and $ f_{j^{\prime}}\in F_{j^{\prime}}$.
Let $F=\prod_{j}F_j$.
Then $F$ is  smooth and is a closed subgroup scheme of $\mathrm{Ker~}\varphi$ as mentioned in the proof of Theorem 4.11. 
If $F^{\dag}$  is the image of $F$ in $G^{\dag}$, then it is smooth and thus  a closed subscheme of $(G^{\dag})_{\mathrm{red}}$.
By observing Equation (A.21) and $(z_j^1/\bar{\alpha})+(z_j^1/\bar{\alpha})^2=0$ above, 
we can easily see that $F^{\dag}$ contains at least one (closed) point of each connected component of $G^{\ddag}$ and this proves our claim.\\

Combining this fact with dim $G^{\dag}$ = dim $G^{\ddag}$, we conclude that $(G^{\dag})_{\mathrm{red}}\simeq G^{\ddag}$,
and hence, $G^{\dag}=G^{\ddag}$ because $G^{\dag}$ is a subfunctor of $G^{\ddag}$.
This completes the proof.
\end{proof}

\begin{Lem}
Choose an even integer $j$ such that $L_{j}$ is \textit{of type I} and $L_{j+2}$ is \textit{of type II} (possibly zero, by our convention).
For such  $j$, there is a non-negative integer $m_j$ such that $L_{j-2l}$ is \textit{of type I} for every $l$ with $0\leq l\leq m_j$ and
$L_{j-2(m_j+1)}$ is \textit{of type II}.
Then   the sum of the equations $$\sum_{l=0}^{m_j} \frac{1}{\bar{\alpha}^2}\mathcal{F}_{j-2l}$$ equals
\begin{equation*}
\sum_{l=0}^{m_j} (\frac{z_{j-2l}}{\bar{\alpha}}+(\frac{z_{j-2l}}{\bar{\alpha}})^2)=
\left(\sum_{l=0}^{m_j}\frac{z_{j-2l}}{\bar{\alpha}}\right)\left(\sum_{l=0}^{m_j}(\frac{z_{j-2l}}{\bar{\alpha}})+1\right)=0.
\end{equation*}
\end{Lem}

\begin{proof}
Our strategy to prove this lemma is the following.
We will first prove that
for each odd integer $i$, the terms containing an $h_i$ add to zero in the sum $\sum_{l=0}^{m_j} \frac{1}{\alpha^2}\mathcal{F}_{j-2l}$.
Then we will show that for each even integer $i$, the terms containing an $\bar{h}_i$ add to zero
in the sum $\sum_{l=0}^{m_j} \frac{1}{\alpha^2}\mathcal{F}_{j-2l}$,
so that only the terms containing the $z_i$ remain.


We recall the notations introduced in the above theorem.
Let $m$ be an element in $(\mathrm{Ker~}\tilde{\varphi}/\tilde{M}^1)(R)$ and $\tilde{m}\in \mathrm{Ker~}\tilde{\varphi}(R)$ be its lift.
For any block $x_i$  of $m$, $\tilde{x}_i$ denotes the corresponding block of $\tilde{m}$ whose reduction is $x_i$.
Since $x_i$ is a block of an element of $(\mathrm{Ker~}\tilde{\varphi}/\tilde{M}^1)(R)$,
its entries are elements of $(B\otimes_AR)/(\pi\otimes 1)(B\otimes_AR) \cong R$ as explained in Remark A.5, whereas
entries of $\tilde{x}_i$ are elements of $B\otimes_AR$.
In addition,
for a block $a_i$  of $h$, where we consider $h$ as an element of $\underline{H}(R)$ as explained in Remark 3.3.(2),
  $\bar{a}_i$ is denoted by  the image  of $a_i$ in  $(B\otimes_AR)/(\pi\otimes 1)(B\otimes_AR)$,
 which is mentioned in Step (i) of the proof of Theorem A.4.
If we write $h$ as a formal matrix $h=\begin{pmatrix} \pi^{i}\cdot h_i\end{pmatrix}$ with $(\pi^{i}\cdot h_i)$
for the $(i,i)$-block and $0$ for the remaining blocks, 
then recall from the paragraph following Equation (A.14)
that  $\bar{h}_k$ is the matrix  with entries in $(B\otimes_AR)/(\pi\otimes 1)(B\otimes_AR) \cong R$
by letting $\pi$ be zero in each entry of the formal matrix $h_k$.\\

To help our computation, we write $\bar{h}_i$. Note that $\epsilon (\in B) \equiv 1$ mod $\pi$.
\begin{equation}
\bar{h}_i=\left\{
  \begin{array}{l l}
 \begin{pmatrix} \begin{pmatrix} 0&1\\1&0\end{pmatrix}& & & \\ &\ddots & & \\ & &\begin{pmatrix} 0&1\\1&0\end{pmatrix}& \\ & & & 1 \end{pmatrix}
     & \quad  \textit{if $i$ is even and $L_i$ is \textit{of type $I^o$}};\\
\begin{pmatrix} \begin{pmatrix} 0&1\\1&0\end{pmatrix}& & & \\ &\ddots & & \\ & &\begin{pmatrix} 0&1\\1&0\end{pmatrix}& \\ & & & \begin{pmatrix} 1&1\\1&0\end{pmatrix} \end{pmatrix}
 & \quad  \textit{if $i$ is even and $L_i$ is \textit{of type $I^e$}};\\
 \begin{pmatrix} \begin{pmatrix} 0&1\\1&0\end{pmatrix}& &  \\ &\ddots &  \\ & &\begin{pmatrix} 0&1\\1&0\end{pmatrix} \end{pmatrix}
    & \quad  \textit{if $i$ is odd or if $i$ is even and  $L_i$ is \textit{of type II}}.
      \end{array} \right.
\end{equation}

Recall that $m_{i,i}$ is a formal matrix as described in Remark A.5, not a matrix
in $M_{n_i\times n_i}((B\otimes_AR)/(\pi\otimes 1)(B\otimes_AR))$,
whereas $m_{i,j}$ for $i\neq j$ is a matrix with entries in $(B\otimes_AR)/(\pi\otimes 1)(B\otimes_AR)$.
Thus we need to modify $m_{i,i}$ into a matrix with entries in $(B\otimes_AR)/(\pi\otimes 1)(B\otimes_AR)$
in order to use Equation (A.14) as explained in the paragraph following Equation (A.14).
We define $\bar{m}_{i,i} (\in M_{n_i\times n_i}((B\otimes_AR)/(\pi\otimes 1)(B\otimes_AR)))$
to be obtained from $m_{i,i}$ by letting $\pi$ be zero in each entry of  the formal matrix $m_{i,i}$.
 $\bar{m}_{i,i}$ is described as follows.
\begin{equation}
 \bar{m}_{i,i}=\left\{
  \begin{array}{l l}
 \begin{pmatrix} id & 0\\ 0& 1 \end{pmatrix}
     & \quad  \textit{if $i$ is even and $L_i$ is \textit{of type $I^o$}};\\
 \begin{pmatrix} id &r_i&0 \\ 0&1& 0 \\ v_i &u_i&1 \end{pmatrix}
 & \quad  \textit{if $i$ is even and $L_i$ is \textit{of type $I^e$}};\\
 id   & \quad  \textit{if $i$ is even and  $L_i$ is \textit{of type II}};\\
 \textit{id}    & \quad  \textit{if $i$ is odd}.
      \end{array} \right.
\end{equation}
In addition, if $i$ is odd, then we have
\begin{equation}
\delta_{i-1}e_{i-1}\cdot m_{i-1, i}+\delta_{i+1}e_{i+1}\cdot m_{i+1, i}=0.
\end{equation}
Here, $\delta_j, e_j$ are  as explained in the description of $\mathrm{Ker~}\tilde{\varphi}(R)$, the paragraph following Lemma A.2.

We choose an even integer $k$ (assuming $m_j>0$)  such that $j-2(m_j-1) \leq k\leq j$  so that both $L_k$ and $L_{k-2}$ are \textit{of type I}.
We observe $\sigma({}^t\tilde{m}_{k-1,k}^{\prime})\cdot\sigma(\pi) h_{k-1}\cdot \tilde{m}_{k-1,k}^{\prime}$ in $\mathcal{F}_k$
and $\sigma({}^t\tilde{m}_{k-1,k-2}^{\prime})\cdot\sigma(\pi) h_{k-1}\cdot \tilde{m}_{k-1,k-2}^{\prime}$ in $\mathcal{F}_{k-2}$ (cf. Equation (A.20)).
We claim that
\begin{equation}
\frac{1}{\pi^2}\left(\sigma({}^t\tilde{m}_{k-1,k}^{\prime})\cdot\sigma(\pi) h_{k-1}\cdot \tilde{m}_{k-1,k}^{\prime}+\sigma({}^t\tilde{m}_{k-1,k-2}^{\prime})\cdot\sigma(\pi) h_{k-1}\cdot \tilde{m}_{k-1,k-2}^{\prime}\right)=0.
\end{equation}
Note that this equation is interpreted as explained in the paragraph following Equation (A.19). 

We use Equation (A.14) for $i=k-1$ and $j=k$ so that we have
\begin{equation}
{}^t\bar{m}_{k-1, k-1}\bar{h}_{k-1}m_{k-1,k}={}^tm_{k, k-1}\bar{h}_{k}\bar{m}_{k,k}.
\end{equation}
Note that this equation is over $(B\otimes_AR)/(\pi\otimes 1)(B\otimes_AR)$.
Indeed, there is a $\sigma$-action in Equation (A.14) but it  is trivial over $(B\otimes_AR)/(\pi\otimes 1)(B\otimes_AR)$.
Recall that $m_{k-1,k}'$ is the last column vector of $m_{k-1,k}$.
Let $e_{k-1}=(0,\cdots, 0, 1)$ be of size $1\times n_{k}$.
Then $m_{k-1,k}'=m_{k-1,k}\cdot {}^te_{k-1}$.
We multiply both sides of the above equation by  ${}^te_{k-1}$ on the right.
Then the left hand side is ${}^t\bar{m}_{k-1, k-1}\bar{h}_{k-1}m_{k-1,k}\cdot {}^te_{k-1}=\bar{h}_{k-1}m_{k-1,k}'$
since ${}^t\bar{m}_{k-1, k-1}=id$.
The right hand side is ${}^tm_{k, k-1}\bar{h}_{k}\bar{m}_{k,k}\cdot {}^te_{k-1}$.
Since  $\bar{m}_{k,k}\cdot {}^te_{k-1}$ is the last column vector of $\bar{m}_{k,k}$,
 $\bar{m}_{k,k}\cdot {}^te_{k-1}={}^te_{k-1}$ by Equation (A.23) so that
${}^tm_{k, k-1}\bar{h}_{k}\bar{m}_{k,k}\cdot {}^te_{k-1}={}^tm_{k, k-1}\bar{h}_{k}\cdot {}^te_{k-1}$.
Furthermore,
$\bar{h}_k$ is symmetric over $(B\otimes_AR)/(\pi\otimes 1)(B\otimes_AR)$
 and so ${}^tm_{k, k-1}\bar{h}_{k}\cdot {}^te_{k-1}={}^t(e_{k-1}\cdot \bar{h}_km_{k, k-1})$.
Then  based on the matrix form of $\bar{h}_k$  in Equation (A.22),
we have that $e_{k-1}\cdot \bar{h}_k$ is the same as $e_k$, where $e_k$ is defined in the paragraph following Lemma A.2.
(There, $e_j$ is defined when $j$ is even and $L_j$ is \textit{of type I})
In conclusion, Equation (A.26) induces the equation  
\begin{equation}
\bar{h}_{k-1}m_{k-1,k}'={}^t(e_k\cdot m_{k, k-1})
\end{equation}
over $(B\otimes_AR)/(\pi\otimes 1)(B\otimes_AR)$.

We again use Equation (A.14) for $i=k-2$ and $j=k-1$ so that we have
\begin{equation}
{}^t\bar{m}_{k-2, k-2}\bar{h}_{k-2}m_{k-2,k-1}={}^tm_{k-1, k-2}\bar{h}_{k-1}\bar{m}_{k-1,k-1}.
\end{equation}
Note that this equation is over $(B\otimes_AR)/(\pi\otimes 1)(B\otimes_AR)$.
Since $k-1$ is odd,  $\bar{m}_{k-1, k-1}=id$ by Equation (A.23).
Recall that $m_{k-1,k-2}'$ is the last column vector of $m_{k-1,k-2}$.
Let $e_{k-1}'=(0,\cdots, 0, 1)$ of size $1\times n_{k-2}$. Then $m_{k-1,k-2}'=m_{k-1,k-2}\cdot {}^te_{k-1}'$.
We multiply both sides of the above equation by $e_{k-1}'$ on the left.
Then the right hand side is $e_{k-1}'\cdot{}^t m_{k-1, k-2}\bar{h}_{k-1}\bar{m}_{k-1,k-1}={}^tm_{k-1,k-2}'\bar{h}_{k-1}$.
Note that $\bar{m}_{k-2, k-2}\cdot {}^te_{k-1}'={}^te_{k-1}'$ by Equation (A.23)
since this is the last column vector of $\bar{m}_{k-2, k-2}$. 
Thus in the left hand side,  $e_{k-1}'\cdot{}^t\bar{m}_{k-2, k-2}\bar{h}_{k-2}m_{k-2,k-1}=e_{k-1}'\cdot \bar{h}_{k-2}m_{k-2,k-1}$.
Based on the matrix form of $\bar{h}_k$  with even integer $k$ in Equation (A.22), 
 $e_{k-1}'\cdot \bar{h}_{k-2}$ is the same as $e_{k-2}$, where $e_k$ is defined  in the paragraph following Lemma A.2.
In conclusion, Equation (A.28) induces the equation 
\begin{equation}
{}^tm_{k-1,k-2}'\bar{h}_{k-1}=e_{k-2}\cdot m_{k-2,k-1}
\end{equation}
over $(B\otimes_AR)/(\pi\otimes 1)(B\otimes_AR)$.

Now we observe Equations (A.27) and (A.29).
Based on the matrix form of $\bar{h}_{k-1}$  with an odd integer $k-1$  in Equation (A.22), 
we have that $\bar{h}_{k-1}\cdot \bar{h}_{k-1}=id$  and $\bar{h}_{k-1}$  is symmetric.
 Thus, by multiplying $\bar{h}_{k-1}$ to Equations (A.27) and (A.29), we obtain $m_{k-1,k}'={}^t(e_k\cdot m_{k, k-1}\cdot \bar{h}_{k-1})$ and
 ${}^tm_{k-1,k-2}'=e_{k-2}\cdot m_{k-2,k-1}\cdot \bar{h}_{k-1}$, respectively, as equations over $(B\otimes_AR)/(\pi\otimes 1)(B\otimes_AR)$.

On the other hand, we observe that $k-1$ is odd and both $L_{k-2}$ and $L_k$ are \textit{of type I}.
Thus $e_{k-2}\cdot m_{k-2,k-1}=e_k\cdot m_{k, k-1}$ by Equation (A.24). 
We multiply $\bar{h}_{k-1}$ to this equation  and so  obtain
$$e_{k-2}\cdot m_{k-2,k-1}\cdot \bar{h}_{k-1}=e_k\cdot m_{k, k-1}\cdot \bar{h}_{k-1}$$
 as  an equation over $(B\otimes_AR)/(\pi\otimes 1)(B\otimes_AR)$.
Therefore, we have the equation
$$m_{k-1,k}'=m_{k-1,k-2}'.$$
As mentioned   at the paragraph following Equation (A.19),
Equation (A.25) is independent of the choice of a lift of $m_{k-1,k}'$ and $m_{k-1,k-2}'$.
Therefore, two terms in Equation (9.25) are same and this verifies our claim.

In the case of $\mathcal{F}_j$,
following the proof of Equation (A.29),
we have ${}^tm_{j+1, j}'\cdot \bar{h}_{j+1}=e_{j}\cdot m_{j, j+1}$. 
Since $L_{j+2}$ is \textit{of type II} (possibly zero, by our convention), $e_{j}\cdot m_{j, j+1}=0$  by Equation (A.24).
Thus,  the term involving  $h_{j+1}$ in $\mathcal{F}_{j}$ is zero.
In the case of $j-2m_j$, where $m_j\geq 0$,  the term involving  $h_{j-2m_j-1}$ in $\mathcal{F}_{j-2m_j}$ is zero
in a manner similar to that of the  above case of $\mathcal{F}_{j}$.

To summarize, for each odd integer $i$, the terms containing an $h_i$ add to zero
in $\sum_{l=0}^{m_j} \frac{1}{\alpha^2}\mathcal{F}_{j-2l}$.\\

We now prove that for each even integer $i$, the terms containing an $\bar{h}_i$ add to zero
in $\sum_{l=0}^{m_j} \frac{1}{\alpha^2}\mathcal{F}_{j-2l}$.
We again choose an even integer  $k$ (assuming $m_j>0$)  such that $j-2(m_j-1) \leq k\leq j$
  so that both $L_k$ and $L_{k-2}$ are \textit{of type I}.
We observe ${}^tm_{k-2,k}^{\prime}\cdot\bar{h}_{k-2}\cdot m_{k-2,k}^{\prime}$ in $\mathcal{F}_k$
and ${}^tm_{k,k-2}^{\prime}\cdot \bar{h}_{k}\cdot m_{k,k-2}^{\prime}$ in $\mathcal{F}_{k-2}$,
and we claim that
\begin{equation}
{}^tm_{k-2,k}^{\prime}\cdot\bar{h}_{k-2}\cdot m_{k-2,k}^{\prime}+ {}^tm_{k,k-2}^{\prime}\cdot \bar{h}_{k}\cdot m_{k,k-2}^{\prime}=0,
\end{equation}
as an equation over $(B\otimes_AR)/(\pi\otimes 1)(B\otimes_AR)$.
Let $\widetilde{m_{k-2,k}'}$ be the $(n_{k-2} \times n_k)$-th entry (resp. $((n_{k-2}-1) \times n_k)$-th entry) of $m_{k-2,k}$
when $L_{k-2}$ is \textit{of type $I^o$} (resp. $I^e$).
We can also define $\widetilde{m_{k,k-2}'}$ as the $(n_k \times n_{k-2})$-th entry (resp. $((n_{k}-1) \times n_{k-2})$-th entry) of $m_{k,k-2}$
when $L_{k}$ is \textit{of type $I^o$} (resp. $I^e$).
Then the above equation (A.30) is the same as
\begin{equation}
\widetilde{m_{k-2,k}'}^2+\widetilde{m_{k,k-2}'}^2=0.
\end{equation}

We use Equation (A.14) for $i=k-2$ and $j=k$ so that we have
\begin{equation}
{}^t\bar{m}_{k-2, k-2}\bar{h}_{k-2}m_{k-2,k}+{}^tm_{k-1, k-2}\bar{h}_{k-1}m_{k-1,k}+{}^tm_{k, k-2}\bar{h}_{k}\bar{m}_{k,k}=0.
\end{equation}
Let $\widetilde{e_k}=(0,\cdots, 0, 1)$ of size $1\times n_{k}$ and $\widetilde{e_{k-2}}=(0,\cdots, 0, 1)$ of size $1\times n_{k-2}$.
Then we have
\begin{equation}
\widetilde{e_{k-2}}\cdot {}^t\bar{m}_{k-2, k-2}\bar{h}_{k-2}m_{k-2,k}\cdot {}^t\widetilde{e_k}=\widetilde{m_{k-2,k}'}
\end{equation}
since $m_{k-2,k}\cdot {}^t\widetilde{e_k}=m_{k-2,k}'$ and
$ \bar{m}_{k-2, k-2}\cdot {}^t\widetilde{e_{k-2}}={}^t\widetilde{e_{k-2}}$.
We also have
\begin{equation}
\widetilde{e_{k-2}}\cdot {}^tm_{k, k-2}\bar{h}_{k}\bar{m}_{k,k}\cdot {}^t\widetilde{e_k}=\widetilde{m_{k,k-2}'}
\end{equation}
since  $\bar{m}_{k,k}\cdot {}^t\widetilde{e_k}={}^t\widetilde{e_k}$ and  $m_{k, k-2}\cdot {}^t\widetilde{e_{k-2}}=m_{k, k-2}'$.
Note that we use Equations (A.22) and (A.23) for our matrix computation.
On the other hand, due to the fact that $\bar{h}_{k-1}\cdot \bar{h}_{k-1}=id$ and $\bar{h}_{k-1}$ is symmetric, we have
\begin{equation}
\widetilde{e_{k-2}}\cdot {}^tm_{k-1, k-2}\bar{h}_{k-1}m_{k-1,k}\cdot {}^t\widetilde{e_k}
=(\widetilde{e_{k-2}}\cdot {}^tm_{k-1, k-2}\bar{h}_{k-1})\cdot  \bar{h}_{k-1}\cdot (\bar{h}_{k-1}m_{k-1,k}\cdot {}^t\widetilde{e_k}).
\end{equation}
Now, $\widetilde{e_{k-2}}\cdot {}^tm_{k-1, k-2}\bar{h}_{k-1}={}^tm_{k-1, k-2}'\bar{h}_{k-1}=e_{k-2}\cdot m_{k-2, k-1}$ by Equation (A.29)
and $\bar{h}_{k-1}m_{k-1,k}\cdot {}^t\widetilde{e_k}=\bar{h}_{k-1}m_{k-1,k}'={}^t(e_k\cdot m_{k, k-1})$ by Equation (A.27).
Since $e_{k-2}\cdot m_{k-2, k-1}=e_k\cdot m_{k, k-1}$ by Equation (A.24), the above equation (A.35) equals
\begin{equation}
 (\widetilde{e_{k-2}}\cdot {}^tm_{k-1, k-2}\bar{h}_{k-1})\cdot  \bar{h}_{k-1}\cdot (\bar{h}_{k-1}m_{k-1,k}\cdot \widetilde{e_k})=
(e_k\cdot m_{k, k-1})\cdot \bar{h}_{k-1}\cdot {}^t(e_k\cdot m_{k, k-1})=0.
\end{equation}

We now combine Equations (A.33), (A.34),   and (A.36).
Namely, if we multiply $\widetilde{e_{k-2}}$ to the left of each side in Equation (A.32)
and we multiply ${}^t\widetilde{e_k}$ to the right of each side in Equation (A.32), then we have
\begin{equation}
\widetilde{m_{k-2,k}'}+0+\widetilde{m_{k,k-2}'}=0
\end{equation}
and so Equations (A.31) and (A.30) are proved.

In the case of $\mathcal{F}_j$,
the term ${}^tm_{j+2,j}^{\prime}\cdot \bar{h}_{j+2}\cdot m_{j+2,j}^{\prime}=0$ since $L_{j+2}$ is \textit{of type II} (possibly zero, by our convention).
Similarly, the term ${}^tm_{j-2m_j-2,j-2m_j}^{\prime}\cdot \bar{h}_{j-2m_j-2}\cdot m_{j-2m_j-2,j-2m_j}^{\prime}$ of $\mathcal{F}_{j-2m_j}$,
 where $m_j \geq 0$,  is $0$ since $L_{j-2m_j-2}$ is \textit{of type II}.
Here, we  use Equation (A.22) for our matrix multiplication.

To summarize, for each even integer $i$, the terms containing an $h_i$ add to zero
in $\sum_{l=0}^{m_j} \frac{1}{\alpha^2}\mathcal{F}_{j-2l}$.\\

Therefore, the sum of equations $\sum_{l=0}^{m_j} \frac{1}{\bar{\alpha}^2}\mathcal{F}_{j-2l}$ equals
$$\sum_{l=0}^{m_j}\frac{1}{\bar{\alpha}^2}(\bar{\alpha}\bar{z}_{j-2l}+\bar{z}_{j-2l}^2)=0.$$
This is the same as
\begin{equation}
\sum_{l=0}^{m_j} (\frac{\bar{z}_{j-2l}}{\bar{\alpha}}+(\frac{\bar{z}_{j-2l}}{\bar{\alpha}})^2)=
\left(\sum_{l=0}^{m_j}\frac{\bar{z}_{j-2l}}{\bar{\alpha}}\right)\left(\sum_{l=0}^{m_j}(\frac{\bar{z}_{j-2l}}{\bar{\alpha}})+1\right)=0.
\end{equation}
 This completes the proof of the lemma.
\end{proof}

\begin{Lem}
Let $G^{\ddag}$ be the subfunctor of $ \mathrm{Ker~}\tilde{\varphi}/\tilde{M}^1$ consisting of those $m$   
satisfying Equations (A.14), (A.16), (A.18), and (A.20). Note that such  $m$ then satisfies Equation (A.21) as well.
Then  $G^{\ddag}$ is represented by a smooth  closed subscheme of $ \mathrm{Ker~}\tilde{\varphi}/\tilde{M}^1$
and is isomorphic to $ \textbf{A}^{l^{\prime}}\times (\mathbb{Z}/2\mathbb{Z})^{\beta}$ as a $\kappa$-variety,
 where $\textbf{A}^{l^{\prime}}$ is an affine space of dimension $l^{\prime}$.
 Here,
\[
 l^{\prime}=\sum_{i<j}n_in_j-\sum_{\textit{i:odd and $L_i$:bound}}n_i+\\
 \sum_{\textit{i:even and $L_i$:of type $I^o$}}(n_i-1)
 +\sum_{\textit{i:even and $L_i$:of type $I^e$}}(2n_i-2).
\]
\end{Lem}

\begin{proof}
Let $\mathcal{J}$ be the set of even integers $j$ such that $L_j$ is \textit{of type I} and $L_{j+2}$ is \textit{of type II} (possibly empty, by our convention).
Note that Equation (A.20) implies Equation (A.21) by Lemma A.7. 
 Equation (A.21) implies that $G^{\ddag}$ is disconnected with at least $2^\beta$ connected components (Exercise 2.19 of \cite{H}).
Here $\beta=\# \mathcal{J}$.
Let $\mathcal{J}_1$ and $\mathcal{J}_2$ be a pair of two (possibly empty) subsets  of $\mathcal{J}$
such that $\mathcal{J}$ is the disjoint union of $\mathcal{J}_1$ and $\mathcal{J}_2$.
Let  $\widetilde{G^{\ddag}_{\mathcal{J}_1, \mathcal{J}_2}}$ be the subfunctor of $ \mathrm{Ker~}\tilde{\varphi}/\tilde{M}^1$ consisting of those $m$   
satisfying Equations (A.14), (A.16), (A.18), and (A.20),
 the equation $\sum_{l=0}^{m_j}\frac{z_{j-2l}}{\bar{\alpha}}=0$ for any $j\in \mathcal{J}_1$,
 and  the equation $\sum_{l=0}^{m_j}\frac{z_{j-2l}}{\bar{\alpha}}=1$ for any $j\in \mathcal{J}_2$.
Here $m_j$ is the integer associated to $j$ defined in Lemma A.7.
We claim that
 $\widetilde{G^{\ddag}_{\mathcal{J}_1, \mathcal{J}_2}}$ is represented by a smooth  closed subscheme of $ \mathrm{Ker~}\tilde{\varphi}/\tilde{M}^1$
and is isomorphic to $ \textbf{A}^{l^{\prime}}$.
Since the scheme $G^{\ddag}$ is a direct product of $\widetilde{G^{\ddag}_{\mathcal{J}_1,
\mathcal{J}_2}}$'s for any such pair of $\mathcal{J}_1, \mathcal{J}_2$ by Exercise 2.19 of \cite{H},
the lemma follows from this claim. 

It is obvious that $\widetilde{G^{\ddag}_{\mathcal{J}_1, \mathcal{J}_2}}$ is represented by a closed subscheme of $ \mathrm{Ker~}\tilde{\varphi}/\tilde{M}^1$
since equations defining $\widetilde{G^{\ddag}_{\mathcal{J}_1, \mathcal{J}_2}}$ as a subfunctor of $ \mathrm{Ker~}\tilde{\varphi}/\tilde{M}^1$ are all polynomials.
Thus it suffices to show that $\widetilde{G^{\ddag}_{\mathcal{J}_1, \mathcal{J}_2}}$ is isomorphic to an affine space $ \textbf{A}^{l^{\prime}}$.
Our strategy to show this is that 
 the coordinate ring of $\widetilde{G^{\ddag}_{\mathcal{J}_1, \mathcal{J}_2}}$ is isomorphic to a polynomial ring.
To do that, we use the following trick over and over.
We consider the polynomial ring $\kappa[x_1, \cdots, x_n]$ and it quotient ring $\kappa[x_1, \cdots, x_n]/(x_1+P(x_2, \cdots, x_n))$.
Then the quotient ring $\kappa[x_1, \cdots, x_n]/(x_1+P(x_2, \cdots, x_n))$
 is isomorphic to $\kappa[x_2, \cdots, x_n]$ and in this case we say that \textit{$x_1$ can be eliminated by $x_2, \cdots, x_n$}.

By the description of an element of $(\mathrm{Ker~}\tilde{\varphi}/\tilde{M}^1)(R)$ in Remark A.5,
we see that $\mathrm{Ker~}\tilde{\varphi}/\tilde{M}^1$ is isomorphic to an affine space of dimension
$$2\sum_{i<j}n_in_j-\sum_{\textit{i:odd and $L_i$:bound}}n_i+ \sum_{\textit{i:even and $L_i$:of type $I^o$}}(2n_i-1)
 +\sum_{\textit{i:even and $L_i$:of type $I^e$}}(4n_i-4)$$
 with variables $(m_{i,j})_{i\neq j}, (y_i, v_i, z_i)_{\textit{i:even and $L_i$:of type $I^o$}},
 (r_i, t_i, y_i, v_i, x_i, z_i, u_i, w_i)_{\textit{i:even and $L_i$:of type $I^e$}}$ such that $\delta_{i-1}e_{i-1}\cdot m_{i-1, i}+\delta_{i+1}e_{i+1}\cdot m_{i+1, i}=0$ with $i$ odd.
Here, $\delta_j, e_j$ are  as explained in the description of $\mathrm{Ker~}\tilde{\varphi}(R)$, the paragraph right after Lemma A.2.

From now on, we eliminate suitable variables based on Equations (A.14), (A.16), (A.18), and (A.20),
 the equations $\sum_{l=0}^{m_j}\frac{z_{j-2l}}{\alpha}=0$ for all $j\in \mathcal{J}_1$,
and the equations $\sum_{l=0}^{m_j}\frac{z_{j-2l}}{\alpha}=1$ for all $j\in \mathcal{J}_2$ .

\begin{enumerate}
\item We first observe Equation (A.14). For two integers $i, j$ with $i<j$, we have
$${}^tm_{j,i}\bar{h}_j\bar{m}_{j,j}=\sum_{i\leq k \leq j-1} {}^tm_{k,i}\bar{h}_km_{k,j}  \textit{ over } (B\otimes_AR)/(\pi\otimes 1)(B\otimes_AR).$$
By Equation (A.23),   $\bar{m}_{j,j}= id$ if $L_j$ is not \textit{of type $I^e$}.
Thus the above equation equals
  ${}^tm_{j,i}\bar{h}_j=\sum_{i\leq k \leq j-1} {}^tm_{k,i}\bar{h}_km_{k,j}$ over $(B\otimes_AR)/(\pi\otimes 1)(B\otimes_AR)$
if $L_j$ is not \textit{of type $I^e$}.
Since $\bar{h}_j$  is a nonsingular matrix by Equation (A.22),  $m_{j,i}$ can be eliminated by the right hand side.
If $L_j$ is  \textit{of type $I^e$}, then
$\bar{m}_{j,j}$  is $\begin{pmatrix} id &r_j&0 \\ 0&1& 0 \\ v_j &u_j&1 \end{pmatrix}$ by Equation (A.23).
We write $\bar{h}_j$   as $\begin{pmatrix} a_j&0&0\\0&1&1\\0&1&0 \end{pmatrix}$ (cf. Equation (A.22)).
Then $\bar{h}_j\bar{m}_{j,j}=\begin{pmatrix} a_j &a_jr_j&0 \\ v_j&1+u_j& 1 \\ 0 &1&0 \end{pmatrix}$.
To compute ${}^tm_{j,i}\bar{h}_j\bar{m}_{j,j}$, we write ${}^tm_{j,i}=\begin{pmatrix} A_j &B_j&C_j  \end{pmatrix}$
so that $${}^tm_{j,i}\bar{h}_j\bar{m}_{j,j}=\begin{pmatrix} A_ja_j+B_jv_j &A_ja_jr_j+B_j(1+u_j)+C_j&B_j  \end{pmatrix}.$$
By first observing the $(1,3)$-block of ${}^tm_{j,i}\bar{h}_j\bar{m}_{j,j}$, $B_j$ can be eliminated by $\sum_{i\leq k \leq j-1} {}^tm_{k,i}\bar{h}_km_{k,j}$.
Then we observe the $(1,1)$-block of ${}^tm_{j,i}\bar{h}_j\bar{m}_{j,j}$.
Since $a_j$ is a nonsingular matrix, we see that $A_j$ can be eliminated by $\sum_{i\leq k \leq j-1} {}^tm_{k,i}\bar{h}_km_{k,j}$
with $B_jv_j$.
By observing the $(1,2)$-block of ${}^tm_{j,i}\bar{h}_j\bar{m}_{j,j}$,
 $C_j$ can be eliminated by $\sum_{i\leq k \leq j-1} {}^tm_{k,i}\bar{h}_km_{k,j}$ with $A_ja_jr_j+B_j(1+u_j)$.
 Therefore, all lower triangular blocks $m_{j,i}$ (with $j>i$) can be eliminated by upper triangular blocks $m_{i,j}$ together with
 $r_j, v_j, u_j$ (resp. $r_i, v_i, u_i$) if $L_j$ (resp. $L_i$) is \textit{of type $I^e$}.
 Here $r_i, v_i, u_i$ as nontrivial blocks of $\bar{m}_{i,i}$, if $L_i$ is \textit{of type $I^e$},
 which is appeared in the right hand side of the above equation.

 On the other hand, the equation $\delta_{i-1}e_{i-1}\cdot m_{i-1, i}+\delta_{i+1}e_{i+1}\cdot m_{i+1, i}=0$ for an odd integer $i$,
 which is one equation defining $\mathrm{Ker~}\tilde{\varphi}/\tilde{M}^1$ (cf. Remark A.5.(3)),
    should be rewritten in terms of upper triangular blocks.
 To do that, we use Equation (A.27) with $i=k-1$.
 Note that the only assumption needed in Equation (A.27) is that $L_k$ is \textit{of type $I$}.
Thus the above equation is the same as
$$\delta_{i-1}e_{i-1}\cdot m_{i-1, i}+\delta_{i+1}{}^t(\bar{h}_im_{i,i+1}')=0.$$

\item We secondly observe Equation (A.16).
If  $L_i$ is \textit{of type $I^o$}, then $v_i$ can be eliminated by $y_i$ and $m_{i-1,i}, m_{i,i+1}$.\\

\item Next, we observe Equation (A.18).
By $\mathcal{X}_{i,1,2}$, $v_i$ can be eliminated by $r_i$.
By $\mathcal{X}_{i,1,3}$, $y_i$ can be eliminated by $t_i, v_i, z_i$ and entries from $m_{i-1,i}, m_{i,i+1}$.
 By  $\mathcal{X}_{i,2,3}$, $x_i$ can be eliminated by $r_i, t_i, z_i, w_i, u_i$  and entries from $m_{i-1,i}, m_{i,i+1}$.\\

\item Finally, we observe $\frac{1}{\bar{\alpha}^2}\mathcal{F}_i$, instead of $\mathcal{F}_i$ (Equation (A.20)), 
together with   
 equations $\sum_{l=0}^{m_j}\frac{z_{j-2l}}{\alpha}=0$ with $j\in \mathcal{J}_1$
and  equations $\sum_{l=0}^{m_j}\frac{z_{j-2l}}{\alpha}=1$ with $j\in \mathcal{J}_2$ .
    Note that  $\frac{1}{\bar{\alpha}^2}\mathcal{F}_i$ is equivalent to $\mathcal{F}_i$ since $\alpha$ is a unit in $B$.
    For each   $j\in \mathcal{J}$, there is a non-negative integer $m_j$ such that $L_{j-2l}$ is \textit{of type I}
    for every $l$ with $0\leq l\leq m_j$ and $L_{j-2(m_j+1)}$ is \textit{of type II} (cf. Lemma A.7).


To analyze these equations, we investigate $\frac{1}{\bar{\alpha}^2}\mathcal{F}_{j-2l}$ for a fixed $j\in \mathcal{J}$.
Firstly assume that $m_j\geq 1$.
Since we have eliminated all lower triangular blocks in Step (1),
we need to replace lower triangular blocks appeared in $\frac{1}{\bar{\alpha}^2}\mathcal{F}_{j-2l}$
  by suitable upper triangular blocks.
If $m_j\geq 2$, then we choose an integer $l$ such that $0<l<m_j$.
By definition, $\frac{1}{\bar{\alpha}^2}\mathcal{F}_{j-2l}$ is
\begin{multline*}
\frac{1}{\bar{\alpha}^2\cdot\pi^2}(\sigma({}^t\tilde{m}_{j-2l-1,j-2l}')\cdot\sigma(\pi) h_{j-2l-1}\cdot
\tilde{m}_{j-2l-1,j-2l}'+\sigma({}^t\tilde{m}_{j-2l+1,j-2l}')\cdot \pi h_{j-2l+1}\cdot
\tilde{m}_{j-2l+1,j-2l}')\\+
  \frac{z_{j-2l}}{\bar{\alpha}}+(\frac{z_{j-2l}}{\bar{\alpha}})^2+
  \frac{{}^tm_{j-2l-2, j-2l}'\cdot\bar{h}_{j-2l-2}\cdot m_{j-2l-2, j-2l}'}{\bar{\alpha}^2}+
  \frac{{}^tm_{j-2l+2, j-2l}'\cdot\bar{h}_{j-2l+2}\cdot m_{j-2l+2, j-2l}'}{\bar{\alpha}^2}=0.
 \end{multline*}
 Here, the first line is interpreted as explained in the paragraph following Equation (A.19)
 and the second line is a polynomial in $(B\otimes_AR)/(\pi\otimes 1)(B\otimes_AR)$.
 We claim that the equation $\frac{1}{\bar{\alpha}^2}\mathcal{F}_{j-2l}$ is the same as the following:
\begin{multline}
\frac{1}{\bar{\alpha}^2\cdot\pi^2}(\sigma({}^t\tilde{m}_{j-2l-1,j-2l}')\cdot\sigma(\pi) h_{j-2l-1}\cdot
\tilde{m}_{j-2l-1,j-2l}'+(e_{j-2l}\cdot \sigma(\tilde{m}_{j-2l, j-2l+1}))\cdot\pi h_{j-2l+1}^3\cdot {}^t(e_{j-2l}\cdot \tilde{m}_{j-2l, j-2l+1}))\\+
  \frac{z_{j-2l}}{\bar{\alpha}}+(\frac{z_{j-2l}}{\bar{\alpha}})^2+(\frac{\widetilde{m_{j-2l-2, j-2l}'}}{\bar{\alpha}})^2+(\frac{\widetilde{m_{j-2l, j-2l+2}'}}{\bar{\alpha}})^2=0.
 \end{multline}
The second line easily follows from the definition of $\widetilde{m_{k-2,k}'}$ and $\widetilde{m_{k,k-2}'}$
(given in the paragraph following Equation (A.30)) combined with Equation (A.37).
For the first line, we observe  Equation (A.29) with $k-2=j-2l$ which gives the identity $m_{j-2l+1, j-2l}'=\bar{h}_{j-2l+1}\cdot {}^t(e_{j-2l}\cdot m_{j-2l, j-2l+1})$ over $(B\otimes_AR)/(\pi\otimes 1)(B\otimes_AR)$.
 Note that the only assumption needed in Equation (A.29) is that $L_{k-2}$ is \textit{of type $I$}.
 Then $h_{j-2l+1}\cdot {}^t(e_{j-2l}\cdot \tilde{m}_{j-2l, j-2l+1})$ is a lift of $\bar{h}_{j-2l+1}\cdot {}^t(e_{j-2l}\cdot m_{j-2l, j-2l+1})$.
The first line is independent of the choice of a lift $\tilde{m}_{j-2l+1,j-2l}'$ of  $m_{j-2l+1,j-2l}'$
  as explained at the paragraph following Equation (A.19).
  This fact completes our claim.
The above equation is equivalent to
\begin{multline}
\frac{1}{\bar{\alpha}^2\cdot\pi^2}(\sigma({}^t\tilde{m}_{j-2l-1,j-2l}')\cdot\sigma(\pi) h_{j-2l-1}\cdot
\tilde{m}_{j-2l-1,j-2l}'+(e_{j-2l}\cdot \sigma(\tilde{m}_{j-2l, j-2l+1}))\cdot\pi h_{j-2l+1}^3\cdot {}^t(e_{j-2l}\cdot \tilde{m}_{j-2l, j-2l+1}))\\+
(\frac{z_{j-2l}}{\bar{\alpha}}+\frac{\widetilde{m_{j-2l-2, j-2l}'}}{\bar{\alpha}}+\frac{\widetilde{m_{j-2l, j-2l+2}'}}{\bar{\alpha}})
+(\frac{z_{j-2l}}{\bar{\alpha}}+\frac{\widetilde{m_{j-2l-2, j-2l}'}}{\bar{\alpha}}+\frac{\widetilde{m_{j-2l, j-2l+2}'}}{\bar{\alpha}})^2=\\
(\frac{\widetilde{m_{j-2l-2, j-2l}'}}{\bar{\alpha}}+\frac{\widetilde{m_{j-2l, j-2l+2}'}}{\bar{\alpha}})
 \end{multline}
by adding $(\frac{\widetilde{m_{j-2l-2, j-2l}'}}{\bar{\alpha}}+\frac{\widetilde{m_{j-2l, j-2l+2}'}}{\bar{\alpha}})$ in both side.

For $\frac{1}{\bar{\alpha}^2}\mathcal{F}_{j-2m_j}$, we observe that  $L_{j-2m_j-2}$ is \textit{of type II}.
By Equation (A.27) with $k=j-2m_j$, we have
${}^tm_{j-2m_j-1,j-2m_j}^{\prime}=e_{j-2m_j}\cdot m_{j-2m_j, j-2m_j-1}\bar{h}_{j-2m_j-1}$.
 Here we use the fact that $\bar{h}_{j-2m_j-1}^2=id$ (cf. Equation (A.22)).   
 Note that the only assumption needed in Equation (A.27) is that $L_k$ is \textit{of type $I$}.
 On the other hand, the equation in Remark A.5.(3), when $i=j-2m_j-1$, is  $e_{j-2m_j}\cdot m_{j-2m_j, j-2m_j-1}=0$
since $L_{j-2m_j-2}$ is \textit{of type II}.
Thus ${}^tm_{j-2m_j-1,j-2m_j}^{\prime}=0$.
Therefore,
$\frac{1}{\bar{\alpha}^2}\mathcal{F}_{j-2m_j}$ is
\begin{multline}
\frac{1}{\bar{\alpha}^2\cdot\pi^2}((e_{j-2m_j}\cdot \sigma(\tilde{m}_{j-2m_j, j-2m_j+1}))\cdot\pi h_{j-2m_j+1}^3\cdot {}^t(e_{j-2m_j}\cdot \tilde{m}_{j-2m_j, j-2m_j+1}))\\+
  \frac{z_{j-2m_j}}{\bar{\alpha}}+(\frac{z_{j-2m_j}}{\bar{\alpha}})^2+(\frac{\widetilde{m_{j-2m_j, j-2m_j+2}'}}{\bar{\alpha}})^2=0.
 \end{multline}
This equation is equivalent to
\begin{multline}
\frac{1}{\bar{\alpha}^2\cdot\pi^2}((e_{j-2m_j}\cdot \sigma(\tilde{m}_{j-2m_j, j-2m_j+1}))\cdot\pi h_{j-2m_j+1}^3\cdot {}^t(e_{j-2m_j}\cdot \tilde{m}_{j-2m_j, j-2m_j+1}))\\+
  (\frac{z_{j-2m_j}}{\bar{\alpha}}+\frac{\widetilde{m_{j-2m_j, j-2m_j+2}'}}{\bar{\alpha}})+(\frac{z_{j-2m_j}}{\bar{\alpha}}+\frac{\widetilde{m_{j-2m_j, j-2m_j+2}'}}{\bar{\alpha}})^2
  =\frac{\widetilde{m_{j-2m_j, j-2m_j+2}'}}{\bar{\alpha}}
 \end{multline}
by adding $\frac{\widetilde{m_{j-2m_j, j-2m_j+2}'}}{\bar{\alpha}}$ in both side.

We emphasize that it is unnecessary to investigate $\mathcal{F}_{j}$ since the equation $\sum_{l=0}^{m_j}\frac{z_{j-2l}}{\bar{\alpha}}=0$
 (resp. $\sum_{l=0}^{m_j}\frac{z_{j-2l}}{\bar{\alpha}}=1$) if $j\in \mathcal{J}_1$ (resp. if $j\in \mathcal{J}_2$) already implies
Equation (A.21)  so that $\sum_{l=0}^{m_j} \frac{1}{\bar{\alpha}^2}\mathcal{F}_{j-2l}=0$.

 We now observe Equations (A.40) and (A.42).
 We introduce a new variable $z_{j-2l}'$
 which is
 $\frac{z_{j-2l}}{\bar{\alpha}}+\frac{\widetilde{m_{j-2l-2, j-2l}'}}{\bar{\alpha}}+\frac{\widetilde{m_{j-2l, j-2l+2}'}}{\bar{\alpha}}$
 (resp. $\frac{z_{j-2m_j}}{\bar{\alpha}}+\frac{\widetilde{m_{j-2m_j, j-2m_j+2}'}}{\bar{\alpha}}$) if $0 < l < m_j$ (resp. $l=m_j$).
 Then $z_{j-2l}$ can be eliminated by $z_{j-2l}', \frac{\widetilde{m_{j-2l-2, j-2l}'}}{\bar{\alpha}}, \frac{\widetilde{m_{j-2l, j-2l+2}'}}{\bar{\alpha}}$.
 In addition, by observing Equations (A.40) and (A.42), $\frac{\widetilde{m_{j-2l-2, j-2l}'}}{\bar{\alpha}}+\frac{\widetilde{m_{j-2l, j-2l+2}'}}{\bar{\alpha}}$ can be eliminated by
 $z_{j-2l}'$ and $m_{j-2l-1,j-2l}', m_{j-2l, j-2l+1}$.
Furthermore, the equation $\sum_{l=0}^{m_j}\frac{z_{j-2l}}{\bar{\alpha}}=0$
 (resp. $\sum_{l=0}^{m_j}\frac{z_{j-2l}}{\bar{\alpha}}=1$) if $j\in \mathcal{J}_1$ (resp. if $j\in \mathcal{J}_2$) induces that
$z_j$ can be eliminated by $z_{j-2l}'$,  $m_{j-2l-1,j-2l}'$, $m_{j-2l, j-2l+1}$ with $0 < l \leq m_j$.\\

If $m_j=0$, then we can show that the equation $\frac{1}{\bar{\alpha}^2}\mathcal{F}_{j}$ is the same as
$$\frac{z_{j}}{\bar{\alpha}}+(\frac{z_{j}}{\bar{\alpha}})^2=0$$
by using an argument similar to that used in the proof of Equation (A.41).
Then the equation $\frac{z_{j}}{\bar{\alpha}}=0$
 (resp. $\frac{z_{j}}{\bar{\alpha}}=1$) if $j\in \mathcal{J}_1$ (resp. if $j\in \mathcal{J}_2$) induces that
$z_j$ can be eliminated.\\
\end{enumerate}

We now combine all cases (1)-(4) observed above.
\begin{enumerate}
\item[(a)] By (1), we eliminate $\sum_{i<j}n_in_j$ variables. 
\item[(b)] By (2), we eliminate $\sum_{\textit{i:even and $L_i$:of type $I^o$}}(n_i-1)$ variables.
\item[(c)] By (3), we eliminate $\sum_{\textit{i:even and $L_i$:of type $I^e$}}(2(n_i-2)+1)$ variables.
\item[(d)] By (4), we eliminate $\#\{i:i\mathrm{~is~even~and~}L_i\mathrm{~is~\textit{of type I}}\}$ variables.
\end{enumerate}

Recall from the third paragraph of the proof that $ \mathrm{Ker~}\tilde{\varphi}/\tilde{M}^1$ is isomorphic to an affine space of dimension
$$2\sum_{i<j}n_in_j-\sum_{\textit{i:odd and $L_i$:bound}}n_i+ \sum_{\textit{i:even and $L_i$:of type $I^o$}}(2n_i-1)
 +\sum_{\textit{i:even and $L_i$:of type $I^e$}}(4n_i-4).$$
  Thus, $\widetilde{G^{\ddag}_{\mathcal{J}_1, \mathcal{J}_2}}$ is isomorphic to an affine space of dimension
\begin{multline}\Big(2\sum_{i<j}n_in_j-\sum_{\textit{i:odd and $L_i$:bound}}n_i+ \sum_{\textit{i:even and $L_i$:of type $I^o$}}(2n_i-1)
 +\sum_{\textit{i:even and $L_i$:of type $I^e$}}(4n_i-4)\Big)-\\
 \Big(\sum_{i<j}n_in_j+\sum_{\textit{i:even and $L_i$:of type $I^o$}}(n_i-1)+\sum_{\textit{i:even and $L_i$:of type $I^e$}}(2(n_i-2)+1)+\\
 \#\{i:i\mathrm{~is~even~and~}L_i\mathrm{~is~\textit{of type I}}\Big).
 \end{multline}

Therefore, the dimension of $\widetilde{G^{\ddag}_{\mathcal{J}_1, \mathcal{J}_2}}$ is
\begin{equation}
\sum_{i<j}n_in_j-\sum_{\textit{i:odd and $L_i$:bound}}n_i+\sum_{\textit{i:even and $L_i$:of type $I^o$}}(n_i-1)+\sum_{\textit{i:even and $L_i$:of type $I^o$}}(2n_i-2).
\end{equation}
\end{proof}

\begin{Lem}
Let $F_j$ be  the closed subgroup scheme  of $\tilde{G}$ defined by the following equations:
\begin{itemize}
\item $m_{i,k}=0$ \textit{if $i\neq k$};
\item $m_{i,i}=\mathrm{id}$ \textit{if $i\neq j$};
\item and for $m_{j,j}$,
\[\left \{
  \begin{array}{l l}
  s_j=\mathrm{id~}, y_j=0, v_j=0 & \quad  \textit{if $L_i$ is \textit{of type} $\textit{I}^o$};\\
  s_j=\mathrm{id~}, r_j=t_j=y_j=v_j=u_j=w_j=0 & \quad \textit{if $L_i$ is \textit{of type} $\textit{I}^e$}.\\
    \end{array} \right.\]
\end{itemize}
Then  $F_j$ is isomorphic to $ \textbf{A}^{1} \times \mathbb{Z}/2\mathbb{Z}$ as a $\kappa$-variety,
where $\textbf{A}^{1}$ is an affine space of dimension $1$, and has exactly two connected components.
\end{Lem}
\begin{proof}
A matrix form of an element $m$ of $F_j(R)$ for a $\kappa$-algebra $R$ is
\[\begin{pmatrix} id&0& & \ldots& & &0\\ 0&\ddots&& & & &\\ & &id& & & & \\  \vdots & & &m_{j,j} & & &\vdots
\\ & & & & id & & \\  & & & & &\ddots &0 \\ 0& & &\ldots & &0 &id \end{pmatrix}\]
such that \[m_{j,j}=\left\{
\begin{array}{l l}
\begin{pmatrix}id&0\\0&1+\pi z_j \end{pmatrix} & \quad \textit{if $L_j$ is of type $I^o$};\\
\begin{pmatrix}id&0&0\\0&1+\pi x_j&\pi z_j\\0&0&1 \end{pmatrix} & \quad \textit{if $L_j$ is of type $I^e$}.
\end{array}\right.\]
To prove the lemma, we observe the matrix equation $\sigma({}^tm)\cdot h\cdot m=h$.
Recall that $h$, as an element of $\underline{H}(R)$, is as explained in Remark 3.3.(2).
Based on Equations (A.1) and (A.2),
the diagonal $(i,i)$-blocks of $\sigma({}^tm)\cdot h\cdot m=h$ with $i\neq j$ are trivial
and the non-diagonal blocks of $\sigma({}^tm)\cdot h\cdot m=h$ are also trivial.
The $(j, j)$-block of $\sigma({}^tm)\cdot h\cdot m$ is
\[\left\{
\begin{array}{l l}
\pi^j\cdot\begin{pmatrix}a_j&0\\0&(1+\sigma(\pi z_j))\cdot (1+2\bar{\gamma}_j)\cdot (1+\pi z_j) \end{pmatrix} & \quad \textit{if $L_j$ is of type $I^o$};\\
\pi^j\cdot\begin{pmatrix}a_j&0&0\\0&(1+\sigma(\pi x_j))(1+\pi x_j)&(1+\sigma(\pi x_j))(1+\pi z_j)\\
0&(1+\sigma(\pi z_j))(1+\pi x_j)&(1+\pi z_j)\sigma(\pi z_j)+\pi z_j+2\bar{\gamma}_j \end{pmatrix} & \quad \textit{if $L_j$ is of type $I^e$}.
\end{array}\right.\]
We write $x_j=x_j^1+\pi x_j^2$ and $z_j=z_j^1+\pi z_j^2$,
where $x_j^1, x_j^2, z_j^1, z_j^2 \in R \subset R\otimes_AB$ and $\pi$ stands for $1\otimes \pi\in R\otimes_AB$.
When $L_j$ is \textit{of type $I^o$}, by observing the $(2, 2)$-block of the above, we obtain the equation
\[\bar{\alpha}(z_j^1)+ (z_j^1)^2=0.\]
Recall that  $\alpha$ is the unit in $B$ such that $\epsilon=1+\alpha\pi$ as explained in Section 2.1,
and $\bar{\alpha}$ is the image of $\alpha$ in $\kappa$.

Then this equation is equivalent to
\[(z_j^1/\bar{\alpha})+ (z_j^1/\bar{\alpha})^2=0\]
by dividing by $\bar{\alpha}^2$ in both sides.
Therefore, in this case, $F_j$ is isomorphic to $ \textbf{A}^{1} \times \mathbb{Z}/2\mathbb{Z}$ as a $\kappa$-variety.

When $L_j$ is \textit{of type $I^e$}, by observing the $(2, 2)$-block of the above, we obtain the equation
\[\bar{\alpha}(x_j^1)+ (x_j^1)^2=0.\]
We also  observe the $(2, 3)$-block of the above, and we obtain two equations
\[x_j^1+z_j^1=0, ~~~ \bar{\alpha}x_j^1+x_j^2+z_j^2+\bar{\alpha} x_j^1z_j^1=0.   \]
By observing the $(3, 3)$-block of the above, we obtain the equation
\[\bar{\alpha}(z_j^1)+ (z_j^1)^2=0.\]
By combining all these, we see that $F_j$ is isomorphic to $ \textbf{A}^{1} \times \mathbb{Z}/2\mathbb{Z}$ as a $\kappa$-variety.
\end{proof}

We introduce the final lemma in order to prove Lemma 4.6 below.
This lemma is about the number of connected components in a short exact sequence of algebraic groups.
\begin{Lem}
Assume that there is a short exact sequence  of linear algebraic groups over $\kappa$
\[1\longrightarrow A \longrightarrow B\longrightarrow C\longrightarrow 1.\]
Let $\pi_0(B)$ be the component group of $B$ which is defined as the Spectrum of the largest separable subalgebra $\pi_0(\kappa[B])$
of $\kappa[B]$, where
 $\kappa[B]$  is the coordinate ring of $B$.
 Let $\#(\pi_0(B))$ be  the order of $\pi_0(B)$, which is defined as
 the dimension of $\pi_0(\kappa[B])$ as a $\kappa$-vector space.
 Note that  $B$ is connected if and only if $\pi_0(B)$ is trivial if and only if $\#(\pi_0(B))=1$.
Thus $\#(\pi_0(B))$ is the number of connected components of $B\otimes_{\kappa}\bar{\kappa}$.
Then $$\#(\pi_0(B))\leqq \#(\pi_0(A))\cdot\#(\pi_0(C)).$$
Moreover, the equality holds if $A$ is connected and in this case, $\pi_0(B)=\pi_0(C)$.
\end{Lem}

\begin{proof}
By definition of a component group, there exists a surjective morphism  $\pi : B \longrightarrow \pi_0(B)$ whose kernel is connected.
Let $A^{\prime} (\subseteq \pi_0(B))$ be the image of $A$ under the morphism $\pi$.
Notice that $A^{\prime}$ is a normal subgroup of $\pi_0(B)$ and that $\#(A^{\prime})\leq\#(\pi_0(A))$.
Then the morphism $\pi$ induces a surjective morphism from $C$ to $\pi_0(B)/A^{\prime}$ and so $\#(\pi_0(B)/A^{\prime})\leq\#(\pi_0(C))$.
Therefore, $\#(\pi_0(B))\leqq \#(\pi_0(A))\cdot\#(\pi_0(C)).$

It is clear that $\#(\pi_0(C))\leqq \#(\pi_0(B))$.
Thus, if $A$ is connected, then $\#(\pi_0(C))= \#(\pi_0(B))$.
In this case, since there exists a surjective morphism from $B$ to $\pi_0(C)$ (through $C$),
there exists a surjective morphism from $\pi_0(B)$ to $\pi_0(C)$.
Since  $\#(\pi_0(C))= \#(\pi_0(B))$,
we can conclude that $\pi_0(B)=\pi_0(C)$.
\end{proof}
\textit{  }

We finally prove Lemma 4.6.
\begin{proof}
 We start with the following short exact sequence
\[1\rightarrow \tilde{G}^1 \rightarrow \mathrm{Ker~}\varphi\rightarrow\mathrm{Ker~}\varphi/\tilde{G}^1\rightarrow 1.\]
It is obvious that $\mathrm{Ker~}\varphi$ is smooth  by Theorems A.4 and A.6.
$\mathrm{Ker~}\varphi$ is also unipotent since it is a subgroup of a unipotent group $\tilde{M}^+$.
Since $\tilde{G}^1$ is connected by Theorem A.4, the component group of $\mathrm{Ker~}\varphi$ is the same as that of $\mathrm{Ker~}\varphi/\tilde{G}^1$ by  Lemma A.10.
Moreover, the dimension of $\mathrm{Ker~}\varphi$ is the sum of the dimension of $\tilde{G}^1$
and the dimension of $\mathrm{Ker~}\varphi/\tilde{G}^1$. This completes the proof.
\end{proof}

\section{Examples} \label{App:AppendixB}
In this appendix, we provide an example with a unimodular lattice $(L, h)$ of rank 1.
Let $L$ be $B\textit{e}$ of rank 1 hermitian lattice with hermitian form $h(le, l'e)=\sigma(l)l'$. 
With this lattice, we construct the smooth integral model and its special fiber and compute the local density.

\subsection{Naive construction (without using our technique)}
We first construct the smooth integral model and its special fiber, without using any technique introduced in this paper.
If we write an element of $L$ as $x+\pi y$ where $x, y\in A$,
then it is easy to see that a naive integral model $\underline{G}'$ is $\mathrm{Spec~}A[x,y]/(x^2+(\pi +\sigma(\pi))xy+\pi\sigma(\pi)y^2-1)$.
As mentioned in Section 2.1, we may assume that $\pi +\sigma(\pi)=2$ and $\pi\sigma(\pi)=2u$ for a unit $u\in A$.
We remark that $\underline{G}'$ is smooth if $p\neq 2$, and in this case its special fiber is
$\mathrm{Spec~}\kappa[x,y]/(x^2-1)=\textbf{A}^1\times \mu_2$ as a $\kappa$-variety.
However, if $p=2$, then its special fiber is no longer smooth since $\kappa[x,y]/(x^2-1)=\kappa[x,y]/(x-1)^2$ is non-reduced.
Some of the difficulty in the case $p=2$ arises from this.
The associated smooth integral model is
obtained by a finite sequence of \textit{dilatations} (at least once) of $\underline{G}'$ (cf. \cite{BLR}).

On the other hand, the difficulty can also be explained in terms of quadratic forms.
Namely, the smoothness of any scheme over $A$ should be closely related to the smoothness of its special fiber.
If we define a function $q : L\longrightarrow A, l\mapsto h(l,l)$,
then $q$  mod 2   is a quadratic form over $\kappa$.
Therefore, the associated smooth integral model should contain information about this quadratic form,
which is  more subtle than quadratic forms over a field of characteristic not equal 2.

To construct the smooth integral model, we observe the characterization of $\underline{G}$ such that $\underline{G}(R)=\underline{G}'(R)$
for an \'etale $A$-algebra $R$.
Thus any element of \underline{G}(R) is of the form $x+\pi y$ such that $x^2+2xy+2uy^2=1$.
Therefore, $(x-1)^2$ is contained in the ideal $(2)$ of $R$ so that we can rewrite $x=1+2x'$ since $R$ is \'etale over $A$.
With this, any element of \underline{G}(R) is of the form $1+2x'+\pi y$ such that
$y+uy^2+2(x'+(x')^2+x'y)=0$.
We consider the affine scheme $\mathrm{Spec~}A[x,y]/(y+uy^2+2(x+x^2+xy))$.
 Its special fiber  is then reduced and smooth.
 Thus, this affine scheme is the desired smooth integral model $\underline{G}$.
 Furthermore, its special fiber $\mathrm{Spec~}\kappa[x,y]/(y+uy^2)$ is isomorphic to $\textbf{A}^1\times \mathbb{Z}/2\mathbb{Z}$
 as a $\kappa$-variety so that the number of rational points is $2f$, where $f$ is the cardinality of $\kappa$.

\subsection{Construction following our technique}
Let $q$ be the function defined over $L$ such that
$$q : L\longrightarrow A, l\mapsto h(l,l).$$
If we write $l=x+\pi y$ such that $x, y \in A$, then $q(l)=h(x+\pi y, x+\pi y)=x^2+(\pi +\sigma(\pi))xy+\pi\cdot \sigma(\pi)y^2$.
Thus $q$  mod 2  is an additive polynomial over $\kappa$.
Let $B(L)$ be the sublattice of $L$ such that $B(L)/\pi L$ is the kernel of the additive polynomial  $q$  mod 2 on  $L/\pi L$.
In this case, $B(L)=\pi L$.

For an    \'etale $A$-algebra $R$ with  $g\in \mathrm{Aut}_{B\otimes_AR}(L\otimes_AR, h\otimes_AR)$,
 it is easy to see that $g$ induces the identity on $L/B(L)=L/\pi L$.
Based on this, we construct the following functor from the category of commutative flat $A$-algebras to the category of monoids as follows. For any commutative flat $A$-algebra $R$, set
    $$\underline{M}(R) = \{m \in \mathrm{End}_{B\otimes_AR}(L \otimes_A R)\} ~|~ \textit{$m$ induces the identity on $L\otimes_A R/ B(L)\otimes_A R$}\}.$$
    This functor $\underline{M}$ is then representable by a polynomial ring and has the structure of a scheme of monoids.
    Let $\underline{M}^{\ast}(R)$ be the set of invertible elements in  $\underline{M}(R)$ for any commutative $A$-algebra $R$.
Then    $\underline{M}^{\ast}$ is representable by a group scheme which is an open subscheme of $\underline{M}$ (Section 3.2).
    Thus  $\underline{M}^{\ast}$ is smooth.
    As a matrix, each element of $\underline{M}^{\ast}(R)$ for a flat $A$-algebra $R$ can be written as $\begin{pmatrix} 1+\pi z \end{pmatrix}$.

    We define another functor from the category of commutative flat $A$-algebras to the category of sets as follows. For any commutative flat $A$-algebra $R$,
    let $\underline{H}(R)$ be the set of hermitian forms $f$ on $L\otimes_{A}R$ (with values in $B\otimes_AR$)
    such that
    $f(a,a)$ mod 2 = $h(a, a)$ mod 2, where $a \in L \otimes_{A}R$.
    As a matrix, each element of $\underline{M}^{\ast}(R)$ for a flat $A$-algebra $R$ is $\begin{pmatrix} 1+2 c \end{pmatrix}$.

    Then for any flat $A$-algebra $R$, the group $\underline{M}^{\ast}(R)$ acts on the right of $\underline{H}(R)$
    by $f\circ m = \sigma({}^tm)\cdot f\cdot m$ and
    this action is represented by an action morphism (Theorem 3.4)
     \[\underline{H} \times \underline{M}^{\ast} \longrightarrow \underline{H} .\]
        Let $\rho$ be the morphism $\underline{M}^{\ast} \rightarrow \underline{H}$ defined by $\rho(m)=h \circ m$,
  which is obtained from the above action morphism.
  As a matrix, for a flat $A$-algebra $R$,  $$\rho(m)=\rho(\begin{pmatrix} 1+\pi z \end{pmatrix})
     =\begin{pmatrix} 1+\pi z+\sigma(\pi z)+\pi\sigma(\pi)\cdot z\sigma(z) \end{pmatrix}.$$

  Then $\rho$ is smooth of relative dimension 1 (Theorem 3.6).
   Let $\underline{G}$ be the stabilizer of $h$ in $\underline{M}^{\ast}$.
 The group scheme $\underline{G}$ is smooth, and $\underline{G}(R)=\mathrm{Aut}_{B\otimes_AR}(L\otimes_A R,h\otimes_A R)$ for any \'{e}tale $A$-algebra $R$ (Theorem 3.8).\\

We now describe the structure of the special fiber $\tilde{G}$ of $\underline{G}$.
For a $\kappa$-algebra $R$,
each element of $\underline{M}(R)$ (resp. $\underline{H}(R)$)  can be written as a formal matrix $m=\begin{pmatrix} 1+\pi z \end{pmatrix}$
(resp. $f=\begin{pmatrix} 1+2 c \end{pmatrix}$).
Firstly, it is easy to see that $B_0=Y_0=\pi L$ so that the morphism $\varphi$ in Section 4.1 is trivial.

For the component groups,
as explained in Theorem 4.11, there is a surjective morphism from $\tilde{G}$ to  $\mathbb{Z}/2\mathbb{Z}$.
 Let us  describe this morphism explicitly below.
It is easy to see that  $L^0=M_0=L$ and $C(L^0)=M_0'=L$.
Here, we follow notation of Section 4.2.
Since $M_0=L$ is \textit{of type $I^o$}, there exists a morphism from the special fiber $\tilde{G}$ $(=G_0)$ to the special fiber of the smooth integral model associated to
$M_0'\oplus C(L^0)= L\oplus L$ \textit{of type $I^e$} as explained in the argument 2 just before Remark 4.10.
Remark 4.10 tells us how to describe this morphism as formal matrices. 
Let $(e_1, e_2)$ be a basis for $L\oplus L$ so that the associated Gram matrix of   the hermitian lattice $L\oplus L$
with respect to this basis is $\begin{pmatrix} 1& 0 \\ 0& 1  \end{pmatrix}$.
Then we consider the basis $(e_1, e_1+e_2)$, with respect to which the morphism described in Remark 4.10 is given as
\[\begin{pmatrix} 1+\pi z \end{pmatrix} \longrightarrow \begin{pmatrix} 1& -\pi z \\ 0& 1+\pi z \end{pmatrix}.\]
We now construct a morphism from the special fiber of the smooth integral model associated to
$M_0'\oplus C(L^0)= L\oplus L$ to $\mathbb{Z}/2\mathbb{Z}$
and describe the image of $\begin{pmatrix} 1& -\pi z \\ 0& 1+\pi z \end{pmatrix}$ in $\mathbb{Z}/2\mathbb{Z}$.

Let $R$ be a $\kappa$-algebra.
The Gram matrix for the hermitian lattice $L\oplus L$ with respect to the basis $(e_1, e_1+e_2)$  is $\begin{pmatrix} 1& 1 \\ 1& 2 \end{pmatrix}$.
Since $L\oplus L$  is \textit{unimodular of type $I^e$},
an $R$-point of the special fiber associated to $L\oplus L$ with respect to this basis is expressed as the  formal matrix
$\begin{pmatrix} 1+\pi x'& \pi z' \\ u'& 1+\pi w' \end{pmatrix}$,  as explained in Section 3.2. 
Based on the argument 1 following Definition 4.9,
 the morphism mapping to $\mathbb{Z}/2\mathbb{Z}$ factors through the special fiber associated to $C(L\oplus L)$,
 composed with the Dickson invariant associated to the corresponding orthogonal group.
$C(L\oplus L)$ is then generated by $(\pi e_1, e_1+e_2)$  and  is $\pi^1$-modular.
 Thus there is no congruence condition on an element of the smooth integral model associated to $C(L\oplus L)$ as explained in Section 3.2.
Write $x'=x'_1+\pi x'_2$, $y'=y'_1+\pi y'_2$, and $z'=z'_1+\pi z'_2$.
The image of $\begin{pmatrix} 1+\pi x'& \pi z' \\ u'& 1+\pi w' \end{pmatrix}$ in the special fiber associated to $C(L\oplus L)$
is  $\begin{pmatrix} 1+\pi x'_1&  z'_1+\pi z'_2 \\ \pi u_1'& 1+\pi w'_1 \end{pmatrix}$.
Since $C(L\oplus L)$ is $\pi^1$-modular with rank 2, there is a morphism from the special fiber associated to $C(L\oplus L)$ to the orthogonal group associated to $C(L\oplus L)/\pi C(L\oplus L)$, as described in Theorem 4.4 or Remark 4.7.
Then the image of $\begin{pmatrix} 1+\pi x'_1&  z'_1+\pi z'_2 \\ \pi u_1'& 1+\pi w'_1 \end{pmatrix}$ in this orthogonal group is
$\begin{pmatrix} 1&  z'_1 \\ 0& 1 \end{pmatrix}$.
The Dickson invariant of $\begin{pmatrix} 1&  z'_1 \\ 0& 1 \end{pmatrix}$ is $z'_1/\bar{\alpha}$
 as mentioned in Step (1) of the proof of Theorem 4.11.
Here, $\alpha$ is the unit in $B$ such that  $\epsilon=1+\alpha\pi$ as explained in Section 2.1,
 and $\bar{\alpha}$ is the image of $\alpha$ in $\kappa$.

In conclusion, the image of $\begin{pmatrix} 1+\pi z \end{pmatrix}$, which is an element of $\tilde{G}(R)$ for a $\kappa$-algebra $R$,
 in $\mathbb{Z}/2\mathbb{Z}$ is $z_1/\bar{\alpha}$, where we write $z=z_1+\pi z_2$.
On the other hand, the equation defining $\tilde{G}$ is $\bar{\alpha} z_1+z_1^2=0$ which is equivalent to $ \frac{z_1}{\bar{\alpha}}+(\frac{z_1}{\bar{\alpha}})^2=0$.
Thus, the morphism from $\tilde{G}$ to $\mathbb{Z}/2\mathbb{Z}$ is surjective.
Therefore the maximal reductive quotient of $\tilde{G}$  is  $\mathbb{Z}/2\mathbb{Z}$  and using Remark 5.3,
$$\#(\tilde{G}(\kappa))=\#(\mathbb{Z}/2\mathbb{Z})\cdot \#(\textbf{A}^1)=2f,$$
where $f$ is the cardinality of $\kappa$.
Based on Theorem 5.2, the local density is
$$\beta_L=f^{0}\cdot 2f=2f.$$


\begin{thebibliography}{99}

    \bibitem[BLR90]{BLR} S. Bosch, W. L$\ddot{\mathrm{u}}$tkebohmert, and M. Raynaud,  \textit{N$\acute{\mathrm{e}}$ron Models}, Ergeb. Math. Grenzgeb.(3) 21, Springer, Berlin, 1990
 \bibitem[Cho15]{C1} S. Cho, \textit{Group schemes and local densities of quadratic lattices in residue characteristic 2}, Compositio Math.
 Vol. 151, 793-827, 2015
 \bibitem[Cho1]{C2} S. Cho, \textit{Group schemes and local densities of ramified hermitian lattices in residue characteristic 2 Part II}, preprint
  \bibitem[CS88]{CS} J. H. Conway  and J. A. Sloane, \textit{Low-dimensional lattices, IV: The mass formula}, Proc. Roy. Soc. London Ser. A 419, 259-286, 1988
\bibitem[DG70]{DG} M. Demazure, P. Gabriel, \textit{Groupes $Alg\acute{e}briques$}, Tome I, Masson et Cie, 1970
\bibitem[SGA3]{SGA3} M. Demazure and A. Grothendieck, \textit{Sch\'emas en groups, I, II, III},
S\'eminaire de G\'eom\'etrie Alg\'ebrique du Bois-Marie 1962/64 (SGA 3), Lecture Notes in Math. 151, 152, 153, Springer, Berlin, 1970
\bibitem[GY00]{GY}  W. T. Gan and J.-K. Yu,  \textit{Group schemes and local densities}, Duke Math. J. 105, 497-524, 2000
\bibitem[Har77]{H}  R. Hartshorne, \textit{Algebraic Geometry}, Grad. Texts in Math. 52, Springer, New York, 1977
\bibitem[Hir98]{H1} Y. Hironaka, \textit{Local zeta functions on Hermitian forms and its application to local densities}, J. Number Theory 71 no. 1, 40-64, 1998
\bibitem[Hir99]{H2} Y. Hironaka, \textit{Spherical functions and local densities on Hermitian forms}, J. Math. Soc. Japan 51 no. 3, 553-581, 1999
\bibitem[HS00]{HS} Y. Hironaka and F. Sato, \textit{Local densities of representations of quadratic forms over $p$-adic integers (the non-dyadic case)}, Jounal of Number Theory, vol. 83, 106-136, 2000
\bibitem[Jac62]{J} R. Jacobowitz, Hermitian forms over local fields, American J. Math. 84, 551-465, 1962
\bibitem[Kit93]{K}  Y. Kitaoka, \textit{Arithmetic of Quadratic Forms}, Cambridge Tracts in Math. 106, Cambridge Univ. Press, Cambridge, 1993
\bibitem[KMRT98]{KMRT} M.-A. Knus,  A. Merkurjev, M. Rost, and J.-P. Tignol, \textit{The Book of Involutions}, AMS Colloquium Publications, Vol. 44, 1998
\bibitem[Mis00]{Mis} M. Mischler,  \textit{Local densities of Hermitian forms}, Quadratic forms and their applications (Dublin, 1999), 201-208, Contemp. Math., 272, Amer. Math. Soc., Providence, RI, 2000
\bibitem[Pal65]{P} G. Pall, \textit{The Weight of a Genus of Positive n-ary Quadratic Forms}, Proc. Sympos. Pure Math. 8, Amer. Math. Soc., Providence, 95-105, 1965
\bibitem[Sah60]{Sa} C.-H. Sah, \textit{Quadratic forms over fields of characteristic 2}, American J. Math. 82, 812-830, 1960
\bibitem[Wat79]{W} W. C. Waterhouse, \textit{Introduction to Affine Group Schemes}, Grad. Texts in Math. 66, Springer, New York, 1979
\bibitem[Wat76]{Wa} G. L. Watson, \textit{The 2-adic density of a quadratic form}, Mathematika 23, 94-106, 1976
\bibitem[Yu]{Yu} Jiu-Kang Yu, \textit{Smooth models associated to concave functions in Bruhat-Tits theory}, preprint



\end{thebibliography}
\end{document}